\documentclass[a4paper,11pt]{article}
\usepackage{amsmath, amsfonts, amscd, amssymb, amsthm, enumerate, xypic}

\def\ul{\text{ul}}
\def\lr{\text{lr}}
\def\md{\text{m}}
\def\defterm{\textbf}

\DeclareMathOperator{\Kbar}{\overline{\mathbb{K}}}
\DeclareMathOperator{\id}{\operatorname{id}}

\DeclareMathOperator{\card}{\#\,}
\DeclareMathOperator{\Mat}{\operatorname{M}}

\DeclareMathOperator{\NT}{\operatorname{NT}}
\DeclareMathOperator{\GL}{\operatorname{GL}}
\DeclareMathOperator{\Ker}{\operatorname{Ker}}
\DeclareMathOperator{\Vect}{\operatorname{span}}
\DeclareMathOperator{\im}{\operatorname{Im}}
\DeclareMathOperator{\tr}{\operatorname{tr}}

\DeclareMathOperator{\Sp}{\operatorname{Sp}}
\DeclareMathOperator{\rk}{\operatorname{rk}}

\renewcommand{\setminus}{\smallsetminus}
\renewcommand{\epsilon}{\varepsilon}

\def\K{\mathbb{K}}

\def\C{\mathbb{C}}

\def\N{\mathbb{N}}

\renewcommand{\L}{\mathbb{L}}

\def\calA{\mathcal{A}}
\def\calB{\mathcal{B}}

\def\calD{\mathcal{D}}
\def\calE{\mathcal{E}}
\def\calF{\mathcal{F}}
\def\calG{\mathcal{G}}
\def\calH{\mathcal{H}}

\def\calL{\mathcal{L}}
\def\calM{\mathcal{M}}

\def\calU{\mathcal{U}}
\def\calV{\mathcal{V}}
\def\calW{\mathcal{W}}
\def\calX{\mathcal{X}}

\def\lcro{\mathopen{[\![}}
\def\rcro{\mathclose{]\!]}}

\theoremstyle{definition}
\newtheorem{Def}{Definition}[section]
\newtheorem{Not}[Def]{Notation}

\theoremstyle{plain}
\newtheorem{theo}{Theorem}[section]
\newtheorem{prop}[theo]{Proposition}
\newtheorem{cor}[theo]{Corollary}
\newtheorem{lemme}[theo]{Lemma}
\newtheorem{claim}[theo]{Claim}

\theoremstyle{plain}
\newtheorem{conj}{Conjecture}[section]

\theoremstyle{remark}
\newtheorem{Rems}{Remarks}[section]
\newtheorem{Rem}[Rems]{Remark}

\title{Spaces of matrices with few eigenvalues}
\author{Cl\'ement de Seguins Pazzis\footnote{Universit\'e de Versailles Saint-Quentin-en-Yvelines, Laboratoire de Math\'ematiques
de Versailles, 45 avenue des Etats-Unis, 78035 Versailles cedex, France}
\footnote{e-mail address: dsp.prof@gmail.com}}

\begin{document}

\thispagestyle{plain}

\maketitle
\begin{abstract}
Let $\K$ be a (commutative) field with characteristic not $2$,
and $\calV$ be a linear subspace of matrices of $\Mat_n(\K)$ that have at most two eigenvalues in $\K$
(respectively, at most one non-zero eigenvalue in $\K$). We prove that $\dim \calV \leq \dbinom{n}{2}+2$ provided that $n \geq 3$
(respectively, $\dim \calV \leq \dbinom{n}{2}+1$).

We also classify, up to similarity, the linear subspaces of $\Mat_n(\K)$ in which every matrix has at most two eigenvalues
(respectively, at most one non-zero eigenvalue) in an algebraic closure of $\K$
and which have the maximal dimension among such spaces.
\end{abstract}

\vskip 2mm
\noindent
\emph{AMS MSC:} 15A30, 15A18

\vskip 2mm
\noindent
\emph{Keywords:} linear subspace, spectrum, dimension, simultaneous triangularization

\tableofcontents

\section{Introduction and main results}

\subsection{Main issues}

In this article, we let $\K$ be an arbitrary (commutative) field, and we denote by $\Kbar$ an algebraic closure of it.
We denote by $\Mat_n(\K)$ the algebra of $n$ by $n$ square matrices with entries in $\K$, by
$\GL_n(\K)$ the group of invertible elements of $\Mat_n(\K)$,
by $\NT_n(\K)$ the linear subspace of $\Mat_n(\K)$ consisting of the strictly upper-triangular matrices,
and by $\frak{sl}_n(\K)$ the linear subspace of $\Mat_n(\K)$ consisting of its trace zero matrices.
Given an extension $\L$ of the field $\K$, the set of eigenvalues in $\L$ of a matrix $M \in \Mat_n(\K)$ is denoted by $\Sp_\L(M)$.

Two subsets $\calV$ and $\calW$ of $\Mat_n(\K)$ are called \textbf{similar},
and we write $\calV \simeq \calW$, if $\calW=P\calV P^{-1}$ for some $P \in \GL_n(\K)$
(this means that $\calV$ and $\calW$ represent, in different bases, the same set of endomorphisms of $\K^n$).

\paragraph{}
Standing out in the modern geometric theory of matrix spaces is Gerstenhaber's seminal work on vector
spaces of nilpotent matrices \cite{Gerstenhaber}. In it, he proves that the dimension of a linear subspace
of $\Mat_n(\K)$ which consists solely of nilpotent matrices is always less than or equal to $\dbinom{n}{2}$,
and that equality occurs only for spaces that are similar to $\NT_n(\K)$.
Gerstenhaber actually stated and proved this result only for fields with more than $n-1$ elements, but
we now have several proofs that his statement holds for all fields (see \cite{dSPGerstenhaberskew,Serezhkin}).

A general problem which is closely connected to Gerstenhaber's theorem is
the understanding of the structure of linear subspaces of $\Mat_n(\K)$ with an upper bound on the cardinality of the spectrum
of their elements.
Let us say that a linear subspace $\calV$ of $\Mat_n(\K)$ is:
\begin{itemize}
\item a \textbf{$k$-spec space} when $\card \Sp_\K(M) \leq k$ for all $M \in \calV$;
\item a \textbf{$\overline{k}$-spec space} when $\card \Sp_{\Kbar}(M) \leq k$ for all $M \in \calV$.
\end{itemize}
Following Gerstenhaber, two obvious issues are:
\begin{enumerate}[(i)]
\item What is the largest dimension for a $k$-spec subspace of $\Mat_n(\K)$
(respectively, for a $\overline{k}$-spec subspace of $\Mat_n(\K)$)?
\item What are, up to similarity, the spaces with maximal dimension?
\end{enumerate}

Partial results on those difficult problems are already known. Let us first recall some recent results on
$1$-spec spaces:

\begin{theo}[de Seguins Pazzis \cite{dSPsoleeigenvalue}]\label{1spectheorem}
Let $\calV$ be a $1$-spec subspace of $\Mat_n(\K)$. Then, $\dim \calV \leq \dbinom{n}{2}+1$ unless
$\K$ has characteristic $2$ and $n=2$, in which case $\dim \calV \leq 3$.
\end{theo}

Of course, these upper bounds also hold for $\overline{1}$-spec spaces, and they are optimal for both types of spaces,
as shown by the following examples:
\begin{enumerate}[(i)]
\item The space $\K I_n\oplus \NT_n(\K)$ of all upper-triangular matrices with equal diagonal entries
is a $\overline{1}$-spec space with dimension $\dbinom{n}{2}+1$.
\item If $\K$ has characteristic $2$, then $\frak{sl}_2(\K)$ is a $3$-dimensional $\overline{1}$-spec subspace of $\Mat_2(\K)$.
\end{enumerate}
A full classification of $\overline{1}$-spec spaces with maximal dimension is also given in \cite{dSPsoleeigenvalue}.

\vskip 2mm
Earlier, Loewy and Radwan had tackled the case of $\overline{2}$-spec and $\overline{3}$-spec subspaces:

\begin{theo}[Loewy and Radwan \cite{Loewy}]\label{LoewyRadwantheorem}
Assume that $\K$ has characteristic $0$.
\begin{enumerate}[(a)]
\item If $n \geq 3$, then every $\overline{2}$-spec subspace of $\Mat_n(\K)$
has dimension less than or equal to $\dbinom{n}{2}+2$.
\item If $n \geq 4$, then every $\overline{3}$-spec subspace of $\Mat_n(\K)$
has dimension less than or equal to $\dbinom{n}{2}+4$.
\end{enumerate}
\end{theo}

Prior to \cite{Loewy}, statement (a) had been proved by Omladi\v c and \v Semrl \cite{Omladic} for the field of complex numbers, provided that $n$ be odd.

The following example from \cite{Loewy} shows that the upper bounds in Theorem \ref{LoewyRadwantheorem} are optimal:
given $k \in \lcro 1,n\rcro$, one considers the space $\calV$ of all matrices of the form
$$\begin{bmatrix}
[?]_{(k-1) \times (k-1)} & [?]_{(k-1) \times (n+1-k)} \\
[0]_{(n-k-1) \times (k-1)} & T
\end{bmatrix},$$
where $T$ is an $(n+1-k) \times (n+1-k)$ upper-triangular matrix with equal diagonal entries.
Then, $\calV$ is a $\overline{k}$-spec subspace of $\Mat_n(\K)$ and
$$\dim \calV=\binom{n}{2}+\binom{k}{2}+1.$$
Loewy and Radwan conjectured that $\dbinom{n}{2}+\dbinom{k}{2}+1$ is an upper bound for the dimension of every
$\overline{k}$-spec subspace of $\Mat_n(\K)$ if $\K$ has characteristic zero and $n>k$.

\vskip 2mm
In this article, we shall expand the knowledge on $k$-spec spaces in two ways:
first of all, we shall extend Loewy and Radwan's upper bound for the dimension of a $\overline{2}$-spec space
to $2$-spec spaces for an arbitrary field of characteristic not 2 (the necessity of this assumption will be explained later on).
In addition, we shall classify the similarity classes of
$\overline{2}$-spec spaces with maximal dimension for an arbitrary field of characteristic not 2,
thus generalizing an earlier result of Omladi\v c and \v Semrl \cite{Omladic}.
We believe that the structure of $2$-spec spaces with maximal dimension is out of reach with the current methods.

Our study will involve two additional conditions on vector spaces of square matrices, which we now introduce.

\begin{Def}
Let $k \in \N$. \\
A \textbf{$k^\star$-spec} subspace of $\Mat_n(\K)$ is a linear subspace $\calV$
in which every matrix $M$ satisfies $\# \bigl(\Sp_\K(M) \setminus \{0\}) \leq k$. \\
A \textbf{$\overline{k}^\star$-spec} subspace of $\Mat_n(\K)$ is a linear subspace $\calV$
in which every matrix $M$ satisfies $\# \bigl(\Sp_{\Kbar}(M) \setminus \{0\}) \leq k$.
\end{Def}

Here are two important observations, which help understand how those new types of subspaces are connected to the former ones:
\begin{itemize}
\item Every $k^\star$-spec subspace is a $(k+1)$-spec subspace.
\item For every $(k+1)$-spec subspace $\calV$, the subspace consisting of the matrices of $\calV$
with last column zero is a $k^\star$-spec space.
\end{itemize}
The same kind of relationship connects $\overline{k}^\star$-spec spaces to $\overline{k+1}$-spec spaces.

In the prospect of our main problem, it will prove extremely useful to investigate similar issues
on $k^\star$-spec spaces and $\overline{k}^\star$-spec spaces.
In short, understanding what goes on for $k^\star$-spec spaces (respectively, for $\overline{k}^\star$-spec spaces)
is a necessary first step for the analysis of $(k+1)$-spec spaces (respectively, of $\overline{k+1}$-spec spaces).
In this light, it is necessary that we recall the following recent result, which generalizes
the upper-bound statement in Gerstenhaber's theorem:

\begin{theo}[Quinlan \cite{Quinlan}, de Seguins Pazzis \cite{dSPlargerank}]\label{0starspectheorem}
Every $0^\star$-spec subspace $\calV$ of $\Mat_n(\K)$ satisfies
$$\dim \calV \leq \binom{n}{2}.$$
\end{theo}

Note that, for fields with more than $2$ elements, a complete classification of $0^\star$-spec spaces with maximal dimension has
been obtained in \cite{dSPaffinenonsingular} (see also \cite{dSPAtkinsontoGerstenhaber} for a far shorter proof
with a more restrictive assumption on the cardinality of the underlying field).

\subsection{Additional notation}

Given integers $i$ and $j$ in $\lcro 1,n\rcro$, we denote by $E_{i,j}$ the matrix of $\Mat_n(\K)$
in which all the entries are zero except the one at the $(i,j)$-spot, which equals $1$.
The transpose of a matrix $M$ is denoted by $M^T$. The trace of a square matrix $M$ is denoted by $\tr(M)$.

Given respective subsets $\calA$ and $\calB$ of $\Mat_n(\K)$ and $\Mat_p(\K)$, we define
$$\calA \vee \calB:=\Biggl\{
\begin{bmatrix}
A & C \\
0 & B
\end{bmatrix} \mid (A,B,C) \in \calA \times \calB \times \Mat_{n,p}(\K)\Biggr\} \subset \Mat_{n+p}(\K).$$

\begin{Rem}
Let $\calA$, $\calA'$, $\calB$ and $\calB'$ be linear subspaces of $\Mat_n(\K)$, $\Mat_n(\K)$, $\Mat_p(\K)$ and $\Mat_p(\K)$, respectively.
If there exists a pair $(P,Q) \in \GL_n(\K) \times \GL_p(\K)$ such that $\calA'=P\calA P^{-1}$ and $\calB'=Q\calB Q^{-1}$, then,
setting $R:=P \oplus Q$, one checks that $R(\calA \vee \calB)R^{-1}=\calA' \vee \calB'$.
Therefore, $\calA \simeq \calA'$ and $\calB \simeq \calB'$ implies $\calA \vee \calB \simeq \calA' \vee \calB'$.
\end{Rem}

\subsection{Upper bounds on the dimension}

Now, we state our two main results on the dimensions of $1^\star$-spec spaces and $2$-spec spaces.
Let us start with $2$-spec spaces, for which we extend Loewy and Radwan's result:

\begin{theo}\label{2specinequality}
Let $n \geq 3$. Assume that $\K$ has characteristic not $2$.
Then, every $2$-spec subspace $\calV$ of $\Mat_n(\K)$ satisfies
$$\dim \calV \leq \binom{n}{2}+2.$$
\end{theo}

The upper bound $\dbinom{n}{2}+2$ is optimal, as is demonstrated by the example of the space
of all $n\times n$ upper triangular matrices in which the first $n-1$ diagonal entries are equal.

The result fails if $\K$ has characteristic $2$, for in that case
$\frak{sl}_2(\K)\vee \bigl(\K I_{n-2} \oplus \NT_{n-2}(\K)\bigr)$ is a $\overline{2}$-spec subspace
of $\Mat_n(\K)$ with dimension $\dbinom{n}{2}+3$. Moreover, for $n=4$, the space
$\frak{sl}_2(\K)\vee \frak{sl}_2(\K)$ is a $\overline{2}$-spec subspace
of $\Mat_n(\K)$ with dimension $\dbinom{n}{2}+4$. We believe that the correct result for characteristic $2$ fields
is the following one:

\begin{conj}\label{car2conj1}
Assume that $\K$ has characteristic $2$ and more than $2$ elements. Let $\calV$ be a $2$-spec subspace of $\Mat_n(\K)$.
If $n=3$ or $n\geq 5$, then $\dim \calV \leq \dbinom{n}{2}+3$. If $n=4$, then $\dim \calV \leq \dbinom{n}{2}+4$.
\end{conj}

Note that if $\K$ has two elements, then any linear subspace of $\Mat_n(\K)$ is a $2$-spec space!

\vskip 3mm
For $1^\star$-spec spaces and fields of characteristic not $2$, it happens
that the upper bound is the same one as for $1$-spec spaces:

\begin{theo}\label{1starspecinequality}
Assume that $\K$ has characteristic not $2$.
Then, every $1^\star$-spec subspace $\calV$ of $\Mat_n(\K)$ satisfies
$$\dim \calV \leq \binom{n}{2}+1.$$
\end{theo}

Obviously, the upper bound $\dbinom{n}{2}+1$ holds for $\overline{1}^\star$-spec spaces as well and is optimal for such spaces,
as is exemplified by $\K I_n \oplus \NT_n(\K)$.
Again, the result fails when $\K$ has characteristic $2$, as is demonstrated by the example of
$\frak{sl}_2(\K) \vee \NT_{n-2}(\K)$. Here is our conjecture for such fields:

\begin{conj}\label{car2conj2}
Assume that $\K$ has characteristic $2$ and more than $2$ elements.
Then, every $1^\star$-spec subspace $\calV$ of $\Mat_n(\K)$ satisfies
$$\dim \calV \leq \binom{n}{2}+2.$$
\end{conj}

Following the method laid out in \cite{dSPlargerank} and \cite{dSPsoleeigenvalue}, we shall prove
Theorem \ref{1starspecinequality} by induction, and then derive Theorem \ref{2specinequality} from it.
In this prospect, a key issue is the structure of the set of rank $1$ matrices in a $2$-spec subspace:
it will be discussed in Section \ref{rank1section}.

\subsection{Spaces of maximal dimension}

Let us now turn to our second set of results. We aim at classifying, up to similarity,
the $\overline{2}$-spec and $\overline{1}^\star$-spec subspaces of maximal dimension in $\Mat_n(\K)$.
For each problem, we will need to distinguish between fields of characteristic $3$ and the other fields of characteristic not $2$.
First of all, some additional notation is necessary:

\begin{Not}
For a non-empty subset $I \subset \lcro 1,n\rcro$, we define $\calD_I$ as the diagonal matrix $(d_{i,j})$ of $\Mat_n(\K)$
such that $d_{i,i}=1$ for every $i \in I$, and $d_{i,i}=0$ for every $i \in \lcro 1,n\rcro \setminus I$.
The number of rows in $\calD_I$ will generally be obvious from the context. \\
We also define
$$\calV^{(1^\star)}_I:=\K \calD_I \oplus \NT_n(\K)$$
and, if $I \neq \lcro 1,n\rcro$, we define
$$\calV^{(2)}_I:=\K I_n\oplus \K \calD_I \oplus \NT_n(\K).$$
\end{Not}
In other words, $\calV^{(1^\star)}_I$ is the space of all upper-triangular $n \times n$ matrices in which the diagonal entries with index in $I$
are equal, and the other diagonal entries are zero; $\calV^{(2)}_I$ is the space of all upper-triangular $n \times n$ matrices
in which the diagonal entries with index in $I$ are equal, and all the other diagonal entries are equal.
Obviously, $\calV^{(1^\star)}_I$ is a $\overline{1}^\star$-spec space with dimension $\dbinom{n}{2}+1$, and
$\calV^{(2)}_I$ is a $\overline{2}$-spec space with dimension $\dbinom{n}{2}+2$ (provided that $I \neq \lcro 1,n\rcro$).
The following result characterizes the case when two such spaces are similar (we shall prove it in Section \ref{uniquesection}):

\begin{prop}\label{uniquenessprop}
Let $I$ and $J$ be non-empty subsets of $\lcro 1,n\rcro$ (respectively, non-empty and proper subsets of $\lcro 1,n\rcro$).
Then, $\calV^{(1^\star)}_I \simeq \calV^{(1^\star)}_J$ if and only if $I=J$
(respectively,  $\calV^{(2)}_I \simeq \calV^{(2)}_J$ if and only if $I=J$ or $I=\lcro 1,n\rcro \setminus J$).
\end{prop}

\begin{Not}
One defines two linear subspaces of $\Mat_4(\K)$ as follows:
$$\calG_4(\K):=\Biggl\{ \begin{bmatrix}
A & B \\
0 & A
\end{bmatrix} \mid (A,B) \in \Mat_2(\K) \times \Mat_2(\K)\Biggr\}$$
and
$$\calG'_4(\K):=\Biggl\{ \begin{bmatrix}
A & B \\
0 & A^T
\end{bmatrix} \mid (A,B) \in \Mat_2(\K) \times \Mat_2(\K)\Biggr\}.$$
\end{Not}

Note that $\calG_4(\K)$ and $\calG'_4(\K)$ are obviously $\overline{2}$-spec subspaces of $\Mat_4(\K)$ with dimension $8=\dbinom{4}{2}+2$.

\paragraph{}
Now, we state our classification theorems, first of all for $\overline{1}^\star$-spec spaces and
fields of characteristic neither $2$ nor $3$:

\begin{theo}\label{1starspecequality}
Assume that $\K$ has characteristic neither $2$ nor $3$. \\
Let $\calV$ be a $\overline{1}^\star$-spec subspace of $\Mat_n(\K)$ with dimension $\dbinom{n}{2}+1$. \\
Then, there is a unique non-empty subset $I$ of $\lcro 1,n\rcro$ such that $\calV \simeq \calV^{(1^\star)}_I$.
\end{theo}

\begin{theo}\label{2specequality}
Assume that $\K$ has characteristic neither $2$ nor $3$. \\
Let $\calV$ be a $\overline{2}$-spec subspace of $\Mat_n(\K)$ with dimension $\dbinom{n}{2}+2$.
\begin{enumerate}[(a)]
\item If $n=3$ or $n \geq 5$, then there is a non-empty and proper subset $I$ of $\lcro 1,n\rcro$ such that
$\calV \simeq \calV^{(2)}_I$.
\item If $n=4$, then one and only one of the following conditions holds:
\begin{enumerate}[(i)]
\item Either there is a non-empty and proper subset $I$ of $\lcro 1,n\rcro$ such that
$\calV \simeq \calV^{(2)}_I$;
\item Or $\calV$ is similar to $\calG_4(\K)$;
\item Or $\calV$ is similar to $\calG'_4(\K)$.
\end{enumerate}
\end{enumerate}
\end{theo}

Theorem \ref{2specequality} was proved by Omladi\v c and \v Semrl for $n$ odd and $\K=\C$
(their proof is really specific to the field of complex numbers and has no obvious generalization, not even to
other algebraically closed fields of characteristic $0$).

In the above two theorems, we have discarded the characteristic $3$ case in addition to the characteristic $2$ case:
this is related to the exceptional $4$-dimensional $1$-spec spaces that have arisen in \cite{dSPsoleeigenvalue}:

\begin{Def}
A $\overline{1}$-spec subspace of $\Mat_3(\K)$ is called \defterm{exceptional}
if it has dimension $4$ and is not similar to $\K I_3 \oplus \NT_3(\K)$.
\end{Def}

A classification of such spaces, up to similarity, is given in Section 4 of \cite{dSPsoleeigenvalue}.

\begin{Not}
Let $p \in \lcro 0,n-3\rcro$ and $\calF$ be an exceptional $\overline{1}$-spec subspace of $\Mat_3(\K)$.
We set
$$\calW_{p,\calF}^{(1^\star)}:=\NT_p(\K) \vee \calF \vee \NT_{n-p-3}(\K)$$
and
$$\calW_{p,\calF}^{(2)}:=\K I_n +\calW_{p,\calF}^{(1^\star)}.$$
\end{Not}

One checks that $\calW_{p,\calF}^{(1^\star)}$ is a $\overline{1}^\star$-spec subspace of $\Mat_n(\K)$ with dimension $\dbinom{n}{2}+1$
and that, if $n \geq 4$, $\calW_{p,\calF}^{(2)}$ is a $\overline{2}$-spec subspace of $\Mat_n(\K)$ with dimension $\dbinom{n}{2}+2$.

For fields of characteristic $3$, we shall establish the following results:

\begin{theo}\label{car3theo1}
Assume that $\K$ has characteristic $3$.
Let $\calV$ be a $\overline{1}^\star$-spec subspace of $\Mat_n(\K)$ with dimension $\dbinom{n}{2}+1$. Then, one and only one of the following
properties holds:
\begin{enumerate}[(i)]
\item Either there exists a non-empty subset $I$ of $\lcro 1,n\rcro$ such that $\calV \simeq \calV^{(1^\star)}_I$;
\item Or there exists an exceptional $\overline{1}$-spec subspace $\calF$ of $\Mat_3(\K)$
and an integer $p \in \lcro 0,n-3\rcro$ such that $\calV \simeq \calW_{p,\calF}^{(1^\star)}$.
In that case, $p$ and the similarity class of $\calF$ are uniquely determined by the similarity class of $\calV$.
\end{enumerate}
We say that $\calV$ has type (I) (respectively, type (II)) when case (i) holds (respectively, when case (ii) holds).
\end{theo}

\begin{theo}\label{car3theo2}
Assume that $\K$ has characteristic $3$. Let $n \geq 3$ be an integer.
Let $\calV$ be a $\overline{2}$-spec subspace of $\Mat_n(\K)$ with dimension $\dbinom{n}{2}+2$. \\
If $n \not\in \{3,4,6\}$, then one and only one of the following conditions holds:
\begin{itemize}
\item[(i)] Either there is a non-empty and proper subset $I$ of $\lcro 1,n\rcro$ such that $\calV \simeq \calV^{(2)}_I$;
\item[(ii)] Or there exists an exceptional $\overline{1}$-spec subspace $\calF$ of $\Mat_3(\K)$
and an integer $p \in \lcro 0,n-3\rcro$ such that $\calV \simeq \calW_{p,\calF}^{(2)}$.
In that case, $p$ and the similarity class of $\calF$ are uniquely determined by the similarity class of $\calV$.
\end{itemize}
We say that $\calV$ has type (I) (respectively, type (II)) when case (i) holds (respectively, when case (ii) holds). \\
If $n=3$, then case (ii) cannot hold. \\
If $n=4$, two additional cases spring up:
\begin{itemize}
\item[(iii)] $\calV$ is similar to $\calG_4(\K)$;
\item[(iv)] $\calV$ is similar to $\calG'_4(\K)$.
\end{itemize}
If $n=6$, one additional case springs up:
\begin{itemize}
\item[(iii')] There are two exceptional $\overline{1}$-spec subspaces $\calF$ and $\calG$ of $\Mat_3(\K)$
such that $\calV \simeq \calF \vee \calG$. In that case, the similarity classes of $\calF$ and $\calG$ are
uniquely determined by that of $\calV$.
\end{itemize}
\end{theo}

Note that, given exceptional $\overline{1}$-spec subspaces $\calF$ and $\calG$ of $\Mat_3(\K)$,
the space $\calF \vee \calG$ is a $\overline{2}$-spec subspace of $\Mat_6(\K)$ with dimension $\dbinom{6}{2}+2$.

\paragraph{}
We finish by formulating some plausible conjectures for the characteristic $2$ case:

\begin{conj}\label{car2conj3}
Assume that $\K$ has characteristic $2$ and more than $2$ elements.
Let $\calV$ be a $\overline{1}^\star$-spec subspace of $\Mat_n(\K)$ with dimension $\dbinom{n}{2}+2$.
Then, there exists a unique $p \in \lcro 0,n-2\rcro$ such that
$$\calV \simeq \NT_p(\K) \vee \frak{sl}_2(\K) \vee \NT_{n-p-2}(\K).$$
\end{conj}

When $\K$ has characteristic $2$, we introduce the space
$$\calH(\K):=\K I_4\oplus \Biggl\{\begin{bmatrix}
0 & l_1 & l_2 & x \\
c_2 & 0 & y & c_1 \\
c_1 & x & 0 & c_2 \\
y & l_2 & l_1 & 0
\end{bmatrix} \mid (l_1,l_2,c_1,c_2,x,y)\in \K^6\Biggr\}.$$
By Theorem 1.7 of \cite{dSPsoleeigenvalue}, $\calH(\K)$ is a $7$-dimensional $\overline{1}$-spec linear subspace of $\Mat_4(\K)$
that is not similar to $\K I_4\oplus \NT_4(\K)$ (up to similarity, it is the only linear subspace of $\Mat_4(\K)$
which has those properties).

\begin{conj}\label{car2conj4}
Assume that $\K$ has characteristic $2$ and more than $2$ elements.
Let $\calV$ be a $\overline{2}$-spec subspace of $\Mat_n(\K)$, with $n \geq 3$.
\begin{enumerate}[(a)]
\item If $n \not\in \{4,6\}$ and $\dim \calV=\dbinom{n}{2}+3$, then
there is a unique integer $p \in \lcro 0,n-2\rcro$ such that
$\calV \simeq \K I_n\oplus \bigl(\NT_p(\K) \vee \frak{sl}_2(\K) \vee \NT_{n-p-2}(\K)\bigr)$.
\item If $n=4$ and $\dim \calV=\dbinom{n}{2}+4$, then
$\calV \simeq \frak{sl}_2(\K) \vee \frak{sl}_2(\K)$.
\item If $n=6$, then one and only of the following three properties holds:
\begin{enumerate}[(i)]
\item There is an integer $p \in \lcro 0,4\rcro$ such that
$\calV \simeq \K I_6\oplus \bigl(\NT_p(\K) \vee \frak{sl}_2(\K) \vee \NT_{4-p}(\K)\bigr)$.
In that case, $p$ is uniquely determined by the similarity class of $\calV$.
\item One has $\calV \simeq \frak{sl}_2(\K) \vee \calH(\K)$.
\item One has $\calV \simeq \calH(\K) \vee \frak{sl}_2(\K)$.
\end{enumerate}
\end{enumerate}
\end{conj}

\vskip 2mm
Theorems \ref{1starspecequality}, \ref{2specequality}, \ref{car3theo1} and \ref{car3theo2} will be proved
by the diagonal-compatibility method that has recently helped prove many similar results
(see \cite{dSPaffinenonsingular,dSPGerstenhaberskew,dSPsoleeigenvalue});
the core ideas of the technique are recalled in Section \ref{diagonalcompmethod}.

\subsection{Structure of the article}

Our starting point is a key lemma on the rank $1$ matrices in a $2$-spec subspace, which is stated and proved in
Section \ref{rank1section} (Proposition \ref{goodprop}).
Theorems \ref{1starspecinequality} and \ref{2specinequality} are successively derived from this lemma in Section \ref{upperboundssection}.

The remaining sections are devoted to the proof of classification Theorems \ref{1starspecequality},
\ref{2specequality}, \ref{car3theo1} and \ref{car3theo2}.
The uniqueness statements in those theorems are all proved in Section \ref{uniquesection}.
In Section \ref{basiclemmasection}, we gather basic lemmas that are used throughout the rest of the
article. In particular, the case $n=2$ in Theorems \ref{1starspecequality} and \ref{car3theo1} is solved there as
an easy consequence of a lemma on $2 \times 2$ matrices. Next, we study a special case in
Theorems \ref{1starspecequality} and \ref{car3theo1}: when there exists a $\calV$-good vector $x$ (see Section \ref{rank1section} for the terminology)
which spans the image of some $M \in \calV$ with $\tr M \neq 0$, we shall prove that $\calV$ is similar to $\calV^{(1^\star)}_{\{n\}}$
by relying upon Gerstenhaber's theorem (Section \ref{equalitysection1}).
In the same section, we outline the main strategy for the rest of the proofs, that is the diagonal-compatibility method.
For fields of characteristic not 3, the other cases in Theorem \ref{1starspecequality}
are tackled in Section \ref{equalitysection2} by using an induction process together with the diagonal-compatibility method.
A similar strategy is used in Section \ref{equalitysection3} to derive Theorem \ref{2specequality}
from Theorem \ref{1starspecequality}, with additional difficulties concentrated mostly in the case $n=4$.

The remaining three sections are devoted to the difficult case of fields of characteristic $3$. In Section \ref{basiclemmasectioncar3},
we gather basic results that are specific to this case. We shall remind the reader of some
key results from \cite{dSPsoleeigenvalue} at this point. Afterwards,
Theorem \ref{car3theo1} will be proved by induction using the diagonal-compatibility method (Section \ref{equalitysection4}),
and finally Theorem \ref{car3theo2} will be derived from Theorem \ref{car3theo1} (Section \ref{equalitysection5}).

\section{On the rank $1$ matrices in a $2$-spec subspace}\label{rank1section}

\subsection{The main result}

Given a finite-dimensional vector space $E$ over $\K$, we adopt the same definition of
$k$-spec subspaces and $k^\star$-spec subspaces for linear subspaces of $\calL(E)$
as for linear subspaces of square matrices with entries in $\K$.

\begin{Def}
Let $E$ be an $n$-dimensional vector space over $\K$, and
$\calV$ be a linear subspace of $\calL(E)$.
A non-zero vector $x \in E$ is called \defterm{$\calV$-good} when no operator $u \in \calV$
satisfies both $\im u=\K x$ and $\tr(u)=0$. The vector $x$ is called \defterm{$\calV$-bad} otherwise.
\end{Def}

Here is our key result:

\begin{prop}\label{goodprop}
Assume that $\K$ has characteristic not $2$.
Let $\calV$ be a linear subspace of $\calL(E)$, where $E$ is an $n$-dimensional vector space over $\K$.
In each one of the following situations, there exists a $\calV$-good vector:
\begin{enumerate}[(a)]
\item $n \geq 3$ and $\calV$ is a $2$-spec subspace of $\calL(E)$;
\item $n \geq 2$ and $\calV$ is a $1^\star$-spec subspace of $\calL(E)$.
\end{enumerate}
\end{prop}

\begin{Rem}
This result fails when $\K$ has characteristic $2$: in that case indeed $\calV=\{0_{n-2}\} \vee \frak{sl}_2(\K)$
is a $1^\star$-spec subspace of $\Mat_n(\K)$ for which no vector is good.
\end{Rem}

In a similar spirit as the proof of Proposition 10 in \cite{dSPlargerank}, we wish to proceed by induction.
Yet, the statement of Proposition \ref{goodprop} does not lend itself to such a method very well.
To get around this difficulty, we shall sharpen the statement as follows:

\begin{Def}
Let $E$ be an $n$-dimensional vector space over $\K$. A \textbf{$2$-complex} of $E$
is an $n$-tuple $(E_i)_{1 \leq i \leq n}$ of linear subspaces of $E$ such that
$\dim E_i=\lfloor \frac{i+1}{2}\rfloor$ for all $i \in \lcro 1,n\rcro$.
\end{Def}

\begin{lemme}\label{goodlemma}
Assume that $\K$ has characteristic not $2$. Let $E$ be an $n$-dimensional vector space over $\K$, and $\calV$
be a $2$-spec subspace of $\calL(E)$.
If $n=2$, assume furthermore that $\calV$ is a $1^\star$-spec subspace of $\calL(E)$.
Then, the union of a $2$-complex of $E$ never contains every $\calV$-good vector of $E$.
\end{lemme}

That of course implies Proposition \ref{goodprop}.
The proof of Lemma \ref{goodlemma} requires two preliminary results:
one on $2$-spec spaces spanned by matrices of rank $1$ and trace $0$ (Corollary \ref{majdim1star} in Section \ref{generatedbyrank1}),
and one on finite coverings of a vector space by linear subspaces (Lemma \ref{coveringlemma} in Section \ref{coveringsection}).
Once they are stated and established, we will use them to prove Lemma \ref{goodlemma} by induction on $n$ (Section \ref{goodlemmaproof}).

\subsection{On the $2$-spec spaces spanned by rank $1$ and trace $0$ matrices}\label{generatedbyrank1}

\begin{lemme}\label{tracelemma}
Assume that $\K$ has characteristic not $2$ and that $n \geq 3$. Let $\calV$ be a $2$-spec subspace of $\Mat_n(\K)$, and $A$ and $B$ be two matrices of
$\calV$ such that $\rk A \leq 1$, $\rk B \leq 1$ and $\tr(A)=\tr(B)=0$. Then, $\tr(AB)=0$.
\end{lemme}

\begin{proof}
Let $\alpha \in \K \setminus \{0\}$.
As $\rk(A+\alpha\,B) \leq 2$ and $\tr(A+\alpha B)=0$, the characteristic polynomial of $A+\alpha\,B$ has the form
$$t^n+c_2(A+\alpha B)\,t^{n-2}.$$
Since $\K$ does not have characteristic $2$, we have
$$c_2(A+\alpha B)=\frac{\bigl(\tr(A+\alpha B)\bigr)^2-\tr\bigl((A+\alpha B)^2\bigr)}{2}=-\frac{\tr(A^2)+2\alpha \tr(AB)+\alpha^2\tr(B^2)}{2}\cdot$$
As $A$ and $B$ have trace zero and rank less than $2$, they are nilpotent and hence $\tr A^2=\tr B^2=0$.
We deduce that
$$c_2(A+\alpha B)=-\alpha \tr(AB).$$
If $\tr(AB) \neq 0$, then we \emph{choose} $\alpha:=\frac{1}{\tr(AB)}$, so that $A+\alpha\,B$ has characteristic polynomial
$t^n-t^{n-2}$, which has three different roots, namely $0$, $1$ and $-1$: this is forbidden since $A+\alpha\,B$ belongs to $\calV$.
Therefore, $\tr(AB)=0$.
\end{proof}

\begin{cor}\label{majdim1star}
Assume that $\K$ has characteristic not $2$ and that $n \geq 3$.
Let $\calV$ be a $2$-spec subspace of $\Mat_n(\K)$ which is spanned by matrices of rank $1$ and trace $0$.
Then, $\dim \calV \leq \dfrac{n^2}{2}\cdot$
\end{cor}

\begin{proof}
Indeed, Lemma \ref{tracelemma} shows that $\calV$ is totally singular for the non-degenerate symmetric bilinear form
$(A,B) \mapsto \tr(AB)$ on $\Mat_n(\K)$. Therefore, $\dim \calV \leq \frac{1}{2}\,\dim \Mat_n(\K)=\frac{n^2}{2}\cdot$
\end{proof}

Even though this corollary is a far cry from Theorem \ref{2specinequality}, it is a necessary step in our proof of it.

\subsection{On the coverings of a vector space by linear subspaces}\label{coveringsection}

\begin{Not}
Let $E$ be a vector space over $\K$. We denote by $E^\star$ its dual space, i.e.\ the space of all linear forms on $E$.
Given a linear subspace $F$ of $E$, we denote by $F^\bot \subset E^\star$ the space of linear forms on $E$ which vanish everywhere on $F$.
\end{Not}

\begin{lemme}\label{coveringlemma}
Let $p$ be a positive integer such that $\# \K>p$.
Let $E$ be an $n$-dimensional vector space over $\K$, and $(E_i)_{i \in I}$
be a family of $(n-1)p+1$ linear subspaces of $E$ in which exactly $p+1$ vector spaces have dimension $n-1$ and, for
all $k \in \lcro 1,n-2\rcro$, exactly $p$ vector spaces have dimension $k$.
Then, $E$ is not included in $\underset{i \in I}{\cup} E_i$.
\end{lemme}

Here is a straightforward corollary:

\begin{cor}\label{coveringlemma2}
Let $p$ be a positive integer such that $\# \K>p$.
Let $E$ be an $n$-dimensional vector space over $\K$, and $(E_i)_{i \in I}$
be a family of $(n-1)p$ linear subspaces of $E$ in which, for
all $k \in \lcro 1,n-1\rcro$, exactly $p$ vector spaces have dimension $k$.
Then, there exists a basis of $E$ in which no vector belongs to $\underset{i \in I}{\cup} E_i$.
\end{cor}

Indeed, if such a basis did not exist, then we would find a linear hyperplane $H$ of $E$ such that
$E \setminus \underset{i \in I}{\cup} E_i \subset H$, leading to $E=H \cup \underset{i \in I}{\cup} E_i$, which would contradict Lemma \ref{coveringlemma}.

Note that Lemma \ref{coveringlemma} is straightforward if $\K$ is infinite, for in that case
$E$ is never the union of finitely many of its proper linear subspaces. Thus, in the following two proofs of Lemma \ref{coveringlemma},
we assume that $\K$ is finite with cardinality $q$. The second one was communicated to us by Silv\`ere Gangloff.

\begin{proof}[Proof of Lemma \ref{coveringlemma} by induction on $n$]
The result is trivial for $n=1$. Given $n \geq 2$, we assume that it holds for $n-1$.
We note that the subspaces $E_i^\bot$ do not cover $E^\star$, for on one hand, we have
$\dim E_i^\bot=n-\dim E_i$ for all $i\in I$, so that
$$\sum_{i \in I} \# E_i^\bot=q+p \sum_{k=1}^{n-1} q^k\leq q+(q-1)\sum_{k=1}^{n-1} q^k=q^n=\# E^\star,$$
and, on the other hand, the linear subspaces $E_i^\bot$ all contain $0$ and there are several of them.
Thus, there is a hyperplane $H$ of $E$ which contains none of the $E_i's$, so that $\dim (E_i \cap H)=\dim E_i-1$
for all $i \in I$. By induction, some vector $x$ of $H$ belongs to none of the $H \cap E_i$ subspaces (note
that we may discard the $p$ spaces in this family that equal zero), and hence $x$ is a vector of $E \setminus \underset{i \in I}{\cup} E_i$.
\end{proof}

\begin{proof}[Direct proof of Lemma \ref{coveringlemma}]
Choose $k \in I$ such that $\dim E_k=n-1$. For all $i \in I \setminus \{k\}$,
we find $\dim (E_k \cap E_i) \geq \dim E_i-1$ since $E_k$ is a hyperplane of $E$, and hence
$$\# (E_i \setminus E_k) \leq q^{\dim E_i}-q^{\dim E_i-1}.$$
Therefore,
$$\sum_{i \in I \setminus \{k\}}\# (E_i \setminus E_k) \leq p\sum_{j=1}^{n-1} (q^j-q^{j-1})
\leq p\,(q^{n-1}-1).$$
It follows that
$$\# \underset{i \in I}{\bigcup} E_i \leq \# E_k+\sum_{i \in I \setminus \{k\}} \#(E_i \setminus E_k)
\leq q^{n-1}+p\,q^{n-1}-p \leq q^n-p<q^n=\# E,$$
and hence the $E_i$ subspaces do not cover $E$.
\end{proof}

\begin{Rem}
When $\K$ is finite with $\# \K=p+1$, Lemma \ref{coveringlemma} is optimal in the sense that, for every positive integer $n$, one can
construct an example where exactly one $1$-dimensional subspace of $E$ is not included in some $E_i$:
for example, in $E=\K^n$, one considers, for $(i,\lambda)\in \lcro 1,n-1\rcro \times \K$,
the subspace $E_{i,\lambda}:=\{(x_1,\dots,x_n) \in \K^n: \; x_1=\cdots=x_i \; \text{and} \; x_{i+1}=\lambda \,x_i\}$.
Set also $H:=\{0\} \times \K^{n-1} \subset \K^n$.
Then, one checks that the vectors of $E \setminus \Bigl(H \cup \underset{i \in \lcro 1,n-1\rcro}{\bigcup}\,
\underset{\lambda \in \K \setminus \{1\}}{\bigcup} E_{i,\lambda}\Bigr)$ are the vectors of the form
$(\alpha,\dots,\alpha)$ for some $\alpha \in \K \setminus \{0\}$.
\end{Rem}

\subsection{The proof of Lemma \ref{goodlemma}}\label{goodlemmaproof}

We use an induction on $n$. Assume first that $n=2$. Let $E$ be a $2$-dimensional vector space over $\K$, and
$\calV$ be a $1^\star$-spec linear subspace of $\calL(E)$.

Assume that we may find two linearly independent $\calV$-bad vectors $x$ and $y$.
Denote by $\calM$ the subspace of $\Mat_2(\K)$ associated with $\calV$ in the basis $(x,y)$.
Then, our assumptions show that $\calM$ contains the matrices $E_{1,2}$ and $E_{2,1}$,
and hence it contains their sum, which has two different non-zero eigenvalues in $\K$, namely $1$ and $-1$.
It follows that there is a $1$-dimensional subspace of $E$ which contains all the $\calV$-bad vectors.

As $E$ has $\# \K+1 \geq 4$ one-dimensional subspaces, we conclude that the set of $\calV$-good vectors is not included in the union of two such
subspaces. This proves the case $n=2$ in Lemma \ref{goodlemma}.

Let $n \geq 3$, and assume that the result of Lemma \ref{goodlemma} holds for the integer $n-1$.
Let $E$ be an $n$-dimensional vector space over $\K$, and $\calV$ be a $2$-spec linear subspace of $E$.
We use a \emph{reductio ad absurdum} by assuming that there exists a $2$-complex $(E_1,\dots,E_n)$ of $E$
whose union contains every $\calV$-good vector.

\begin{claim}\label{gooddualbasis}
There exists a basis $(\varphi_1,\dots,\varphi_n)$
of the dual space $E^\star$ such that no $E_i$ is included in some $\Ker \varphi_j$.
\end{claim}

\begin{proof}
This amounts to showing that there is a basis of $E^\star$ consisting of vectors none of which belongs
to some $E_i^\bot$. As the subspaces $E_1^\bot,\dots,E_n^\bot$ have respective
dimensions $n-1,n-1,n-2,n-2,\dots,n-\lfloor \frac{n+1}{2}\rfloor$, this follows from
Corollary \ref{coveringlemma2} applied to $p=2$.
\end{proof}

Given a (linear) hyperplane $H$ of $E$, we define the linear subspace
$$\calV_H:=\Bigl\{u \in \calV : \; \tr(u)=0 \quad \text{and} \quad \forall x \in H, \; u(x)=0\Bigr\}.$$

\begin{claim}\label{hyperplaneclaim}
Let $H$ be a hyperplane of $E$ which contains none of the spaces $E_1,\dots,E_n$.
Then, $\dim \calV_H>\lfloor \frac{n}{2}\rfloor$.
\end{claim}

\begin{proof}
Set $p:=\lfloor \frac{n}{2}\rfloor$. Denote by $\calW$ the linear subspace of $\calV$ consisting of its elements with image included in $H$,
and by $\calW'$ the space of endomorphisms of $H$ induced by the elements of $\calW$.
Then, $\calW'$ is a $1^\star$-spec subspace of $\calL(H)$ (indeed, given $u \in \calW$, we already know that $0$
is an eigenvalue of $u$, and hence any endomorphism induced by $u$ has at most one non-zero eigenvalue).
We contend that the set of vectors of $H$ that are both $\calV$-bad and $\calW'$-good
has a rank greater than $p$. To support this, we use a \emph{reductio ad absurdum}
by assuming that this is not the case.
Then, there is a $p$-dimensional linear subspace $G$ of $H$
which contains all the vectors of $H$ that are both $\calV$-bad and $\calW'$-good.
As $H$ contains none of the spaces $E_1,\dots,E_n$,
one checks that $E_1 \cap H=E_2 \cap H=\{0\}$ and
that $\bigl(E_3 \cap H,\dots,E_n \cap H,G\bigr)$ is a $2$-complex of $H$.
Applying the induction hypothesis to $\calW'$, we find that
at least one $\calW'$-good vector belongs to none of the spaces $G,E_1,E_2,\dots,E_n$.
Yet, such a vector must be $\calV$-good since it belongs to $H$ but not to $G$,
in contradiction with our initial assumption, which stated that
no vector outside of the union of $E_1,\dots,E_n$ were $\calV$-good.

We deduce that we may find $p+1$ linearly independent vectors $x_1,\dots,x_{p+1}$ of $H$
that are all $\calV$-bad and $\calW'$-good. Let us extend this family to a basis $(x_1,\dots,x_{n-1})$ of $H$,
and then to a basis $(x_1,\dots,x_n)$ of $E$. Consider the subspace $\calM$ of $\Mat_n(\K)$
associated with $\calV$ in that basis. Let $i \in \lcro 1,p+1\rcro$. As $x_i$ is $\calV$-bad,
there is a non-zero matrix $A \in \calM$ with trace zero and all rows zero except the $i$-th.

We write
$$A=\begin{bmatrix}
B & [?]_{(n-1) \times 1} \\
[0]_{1 \times (n-1)} & 0
\end{bmatrix}$$
and note that $B$ represents an element of $\calW'$ in the basis $(x_1,\dots,x_{n-1})$.
We also remark that $\tr B=0$ and that $B$ has all rows zero with the possible exception of the $i$-th. As $x_i$ is $\calW'$-good,
we deduce that $B=0$. It follows that $A$ is a non-zero scalar multiple of $E_{i,n}$.
Therefore $\calM$ contains $\Vect(E_{1,n},\dots,E_{p+1,n})$. As $(x_1,\dots,x_{n-1})$ is a basis of $H$,
all the matrices of $\Vect(E_{1,n},\dots,E_{p+1,n})$ represent elements of $\calV_H$ in the basis $(x_1,\dots,x_n)$.
This proves that $\dim \calV_H >p$.
\end{proof}

We are ready to conclude. Choose a basis $(\varphi_1,\dots,\varphi_n)$ of $E^\star$
with the property stated in Claim \ref{gooddualbasis}, and, for $i \in \lcro 1,n\rcro$, set $\calV_i:=\calV_{\Ker \varphi_i}$.
Denote by $(e_1,\dots,e_n)$ the pre-dual basis of $E$ for $(\varphi_1,\dots,\varphi_n)$.
Note that the linear subspaces $\calV_1,\dots,\calV_n$ of $\calL(E)$ are linearly independent: indeed,
in the basis $(e_1,\dots,e_n)$, the matrices associated with the elements of $\calV_i$ have all columns zero with the possible
exception of the $i$-th. Finally, setting
$$\calV':=\calV_1 \oplus \cdots \oplus \calV_n,$$
we deduce from Claim \ref{hyperplaneclaim} that
$$\dim \calV' \geq n\,\Bigl(1+\Bigl\lfloor \frac{n}{2}\Bigr\rfloor\Bigr)>\frac{n^2}{2}\cdot$$
This contradicts Corollary \ref{majdim1star} as $\calV'$ is included in the $2$-spec space $\calV$
and is spanned by matrices of rank $1$ and trace $0$.
We deduce that our initial assumption was false, i.e.\ the union of a $2$-complex of $E$ cannot contain
all the $\calV$-good vectors.

This finishes our proof of Lemma \ref{goodlemma} by induction. Proposition \ref{goodprop} follows directly from it.

\section{Upper bounds for the dimensions of $1^\star$-spec and $2$-spec spaces}\label{upperboundssection}

In the whole section, we assume that $\K$ has characteristic not $2$.
We prove Theorem \ref{1starspecinequality} by induction on $n$,
and then we deduce Theorem \ref{2specinequality} from it.

\subsection{Proof of Theorem \ref{1starspecinequality}}\label{proofof1starspecinequality}

The case $n=1$ is trivial.
Let $n \geq 2$. Assume that
every $1^\star$-spec subspace of $\Mat_{n-1}(\K)$ has its dimension less than or equal to $\dbinom{n}{2}+1$. \\
Let $\calV$ be a $1^\star$-spec linear subspace of $\Mat_n(\K)$.
By Proposition \ref{goodprop}, we may find a $\calV$-good vector $e_n$.
We complete $e_n$ in a basis $(e_1,\dots,e_n)$ of $\K^n$. Replacing $\calV$ with a similar space of matrices,
this shows that we may assume that $\calV$ contains no non-zero matrix with trace zero and all rows zero except the $n$-th one.
Then, we write every matrix of $\calV$ as
$$M=\begin{bmatrix}
K(M) & C(M) \\
[?]_{1 \times (n-1)} & a(M)
\end{bmatrix},$$
where $K(M)$ and $C(M)$ are respectively $(n-1) \times (n-1)$ and $(n-1) \times 1$ matrices, and $a(M)$ is a scalar.
Set
$$\calW:=\bigl\{M \in \calV : \; C(M)=0\bigr\},$$
and note that the rank theorem yields:
$$\dim \calV = \dim C(\calV)+\dim \calW\leq (n-1)+\dim \calW.$$
Note also that $K(\calW)$ is a $1^\star$-spec subspace of $\Mat_{n-1}(\K)$.
Denoting by $\calX$ the subspace of $\calW$ consisting of its matrices with all first $n-1$ rows equal to zero,
our initial assumption can be rephrased as follows:
$$\forall M \in \calX, \; a(M)=0 \Rightarrow M=0.$$
Now, we distinguish between two cases:
\begin{itemize}
\item Assume that $\calX=\{0\}$. \\
Then, $\dim \calW=\dim K(\calW)$. Applying the induction hypothesis
thus yields
$$\dim  \calV \leq (n-1)+\binom{n-1}{2}+1=\binom{n}{2}+1.$$

\item Assume that $\calX \neq \{0\}$, so that $\calX$ contains a matrix of the form
$$A=\begin{bmatrix}
[0]_{(n-1) \times (n-1)} & [0]_{(n-1) \times 1} \\
[?]_{1 \times (n-1)} & 1
\end{bmatrix}.$$
As $a(M)=0 \Rightarrow M=0$ for every $M \in \calX$, we have
$$\dim \calW = \dim K(\calW)+1.$$

By linearly combining $A$ with the matrices of $\calW$, we find that, for every $N \in K(\calW)$ and every $\lambda \in \K$, the space $\calV$
contains a matrix of the form
$$\begin{bmatrix}
N & [0]_{(n-1) \times 1} \\
[?]_{1 \times (n-1)} & \lambda
\end{bmatrix};$$
if $N$ has a non-zero eigenvalue $\alpha$, we may choose $\lambda$ in  $\K \setminus \{0,\alpha\}$,
which contradicts the assumption that $\calV$ be a $1^\star$-spec space.
Therefore, $K(\calW)$ is a $0^\star$-spec subspace of $\Mat_{n-1}(\K)$, and hence, by Theorem \ref{0starspectheorem},
$$\dim K(\calW) \leq \binom{n-1}{2}.$$
We conclude that
$$\dim \calV \leq  (n-1)+\binom{n-1}{2}+1=\binom{n}{2}+1.$$
\end{itemize}
This completes the proof of Theorem \ref{1starspecinequality}.

\subsection{Proof of Theorem \ref{2specinequality}}\label{2specinequalitysection}

Assume that $n \geq 3$. Let $\calV$ be a $2$-spec subspace of $\Mat_n(\K)$.
By Proposition \ref{goodprop}, we lose no generality in assuming that the last vector $e_n$ of the canonical basis of $\K^n$
if $\calV$-good. Then, as in the preceding section, we write every matrix of $\calV$ as
$$M=\begin{bmatrix}
K(M) & C(M) \\
[?]_{1 \times (n-1)} & a(M)
\end{bmatrix},$$
but now we set
$$\calW:=\bigl\{M \in \calV : \; C(M)=0 \; \text{and}\; a(M)=0\bigr\}.$$
Thus, $\calW$ is a $1^\star$-spec subspace of $\Mat_n(\K)$, and hence $K(\calW)$ is a $1^\star$-spec subspace of $\Mat_{n-1}(\K)$.
As $e_n$ is $\calV$-good, we find that $\dim K(\calW)=\dim \calW$, and hence the rank theorem yields
$$\dim \calV \leq n+\dim \calW=n+\dim K(\calW).$$
As Theorem \ref{1starspecinequality} entails that $\dim K(\calW) \leq \dbinom{n-1}{2}+1$, we conclude that
$$\dim \calV \leq n+\binom{n-1}{2}+1=\binom{n}{2}+2.$$
This completes the proof of Theorem \ref{2specinequality}.

\section{Uniqueness statements}\label{uniquesection}

The purpose of this section is to establish Proposition \ref{uniquenessprop} together with
the uniqueness statements in Theorems \ref{1starspecequality} to \ref{car3theo2}.

\subsection{Spaces of type $\calG_4(\K)$ or $\calG'_4(\K)$ versus the other ones}

\begin{prop}
\begin{enumerate}[(a)]
\item The spaces $\calG_4(\K)$ and $\calG'_4(\K)$ are not similar.
\item Given a non-empty and proper subset $I \subset \lcro 1,4\rcro$, the space
$\calV_I^{(2)}$ is similar neither to $\calG_4(\K)$ nor to $\calG'_4(\K)$.
\item If $\K$ has characteristic $3$ and $\calF$ is an exceptional $\overline{1}$-spec subspace of $\Mat_3(\K)$, then
$\calW_{0,\calF}^{(2)}$ is similar neither to $\calG_4(\K)$ nor to $\calG'_4(\K)$,
and $\calW_{1,\calF}^{(2)}$ is similar neither to $\calG_4(\K)$ nor to $\calG'_4(\K)$.
\end{enumerate}
\end{prop}

\begin{proof}
We note that $\calG'_4(\K)$ is not closed under multiplication, in contrast with $\calG_4(\K)$ and any space of the form $\calV^{(2)}_I$.
Therefore, $\calG'_4(\K)$ is not similar to $\calG_4(\K)$ nor to any space of the form $\calV^{(2)}_I$. \\
Moreover, $\calG_4(\K)$
is a Lie subalgebra of $\Mat_4(\K)$ and one computes that
$$\bigl\{A \oplus A \mid  A \in \frak{sl}_2(\K)\bigr\} \subset [\calG_4(\K),\calG_4(\K)].$$
On the other hand, for every non-empty and proper subset $I$ of $\lcro 1,n\rcro$, the subspace
$\calV^{(2)}_I$ is also a Lie subalgebra of $\Mat_4(\K)$, but its
derived algebra is included in $\NT_4(\K)$, which contains only singular matrices. As $[\calG_4(\K),\calG_4(\K)]$
contains non-singular matrices, we conclude that $\calG_4(\K)$ is similar to no space of the form $\calV^{(2)}_I$.

Now, let us assume that $\K$ has characteristic $3$ and let $\calF$ be an exceptional $\overline{1}$-spec subspace of $\Mat_3(\K)$.
One notes that each one of the spaces $\calW_{0,\calF}^{(2)}$ and $\calW_{1,\calF}^{(2)}$
contains either $E_{1,1}$ or $E_{4,4}$, which are matrices of rank $1$ and trace $1$.
However, neither $\calG_4(\K)$ nor $\calG'_4(\K)$ contains a matrix of rank $1$ and trace $1$
(for example, if such a matrix belonged to $\calG_4(\K)$, it would have the form
$M=\begin{bmatrix}
A & [?] \\
[0] & A
\end{bmatrix}$ and $A$ should be non-zero as $\tr(M) \neq 0$, so that $\rk M \geq 2\rk A\geq 2$).
It follows that neither $\calG_4(\K)$ nor $\calG'_4(\K)$ is similar to any one of the spaces
$\calW_{0,\calF}^{(2)}$ and $\calW_{1,\calF}^{(2)}$. This finishes the proof.
\end{proof}

\subsection{On the similarities between spaces of type $\calV_I^{(1^\star)}$ or $\calV_I^{(2)}$}

Here, we prove Proposition \ref{uniquenessprop}.
To achieve this, we need a preliminary result which is rather classical and which will be used in later sections:

\begin{lemme}\label{normalizerlemma}
Let $P \in \GL_n(\K)$ be such that $P\NT_n(\K)P^{-1}\subset \NT_n(\K)$. Then, $P$ is upper-triangular.
\end{lemme}

\begin{proof}
Note that $P\NT_n(\K)P^{-1}= \NT_n(\K)$ since the dimensions are equal.
Denote by $(e_1,\dots,e_n)$ the canonical basis of $\K^n$, and set $F_i:=\Vect(e_1,\dots,e_i)$ for every $i \in \lcro 1,n\rcro$. \\
Given a linear subspace $\calV$ of $\Mat_n(\K)$, one defines a sequence $(\calV(k))_{k \geq 0}$
of linear subspaces of $\K^n$ by $\calV(0)=\K^n$ and, for every non-negative integer $k$, $\calV(k+1)=\Vect\{MX \mid M \in \calV,\, X \in \calV(k)\}$.
In the special case $\calV=\NT_n(\K)$, that sequence is obviously $(F_n,F_{n-1},\dots,F_1,\{0\},\dots)$,
while, for $\calV=P \NT_n(\K) P^{-1}$, one checks by induction that it is $(PF_n,PF_{n-1},\dots,PF_1,\{0\},\dots)$.
We conclude that $PF_k=F_k$ for every $k \in \lcro 1,n\rcro$, which proves that $P$ is upper-triangular.
\end{proof}

\begin{proof}[Proof of Proposition \ref{uniquenessprop}]
Let $I$ and $J$ be non-empty subsets of $\lcro 1,n\rcro$.
Obviously $\calV_I^{(1^\star)}=\calV_J^{(1^\star)}$ if and only if $I=J$.
If in addition $I$ and $J$ are both proper subsets of $\lcro 1,n\rcro$, then one sees that
$\calV_I^{(2)}=\calV_J^{(2)}$ if and only if $I=J$ or $I=\lcro 1,n\rcro \setminus J$. \\
Assume that $P\calV_I^{(1^\star)}P^{-1}=\calV_J^{(1^\star)}$ for some $P \in \GL_n(\K)$.
The set of nilpotent matrices of $\calV_I^{(1^\star)}$ is $\NT_n(\K)$, and the same holds for $\calV_J^{(1^\star)}$.
As conjugating with a non-singular matrix preserves nilpotency, we deduce that $P \NT_n(\K)P^{-1}=\NT_n(\K)$, and
hence $P$ is upper-triangular. Therefore $\calV_I^{(1^\star)}=P\calV_I^{(1^\star)}P^{-1}=\calV_J^{(1^\star)}$, and hence $I=J$.
The same line of reasoning shows that $P\calV_I^{(2)}P^{-1}=\calV_J^{(2)}$ implies
$\calV_I^{(2)}=\calV_J^{(2)}$, which yields the second result in Proposition \ref{uniquenessprop}.
\end{proof}

Now, we have established all the uniqueness statements in Theorems \ref{1starspecequality} and \ref{2specequality}.
It remains to prove the remaining ones in Theorems \ref{car3theo1} and \ref{car3theo2}: that is the topic of the next section.

\subsection{The characteristic $3$ case: exceptional types versus spaces of type $\calV_I^{(1^\star)}$ or $\calV_I^{(2)}$}

Here, we assume that $\K$ has characteristic $3$. In order to examine the possible
similarities between the spaces cited in Theorems \ref{car3theo1} and \ref{car3theo2},
we shall examine the possible dimensions of $\calV x=\{u(x)\mid u \in \calV\}$, when $\calV$
is a linear subspace of $\Mat_n(\K)$ and $x$ is a non-zero vector of $\K^n$.
Our starting point is the following lemma:

\begin{lemme}\label{exceplemma1}
Let $\calF$ be an exceptional $\overline{1}$-spec subspace of $\Mat_3(\K)$.
Then, $\dim \calF x \geq 2$ for all $x \in \K^3 \setminus \{0\}$.
\end{lemme}

To prove it, we need a simple result which could have been featured in \cite{dSPsoleeigenvalue} but was deemed too
elementary at the time. Its proof is similar to that of Theorem 1.6 in \cite{dSPsoleeigenvalue}:

\begin{prop}\label{1specredn=2}
Let $\K$ be an arbitrary field of characteristic not $2$, and $\calV$ be a $\overline{1}$-spec subspace of $\Mat_2(\K)$ with dimension
$2$. Then $\calV \simeq \K I_2\oplus \NT_2(\K)$.
\end{prop}

\begin{proof}[Proof of Lemma \ref{exceplemma1}]
Assume on the contrary that some $x \in \K^3 \setminus \{0\}$ satisfies $\dim \calF x \leq 1$.
Without loss of generality, we may assume that $x=\begin{bmatrix}
1 & 0 & 0
\end{bmatrix}^T$.
As $I_3 \in \calF$, we deduce that $\calF x=\K x$, and therefore every matrix $M$ of $\calF$
may be written as
$$\begin{bmatrix}
b(M) & [?]_{1 \times 2} \\
[0]_{2 \times 1} & K'(M)
\end{bmatrix},$$
where $b(M) \in \K$ and $K'(M)$ is a $2 \times 2$ matrix.
Therefore, $K'(\calF)$ is a $\overline{1}$-spec subspace of $\Mat_2(\K)$.
One notes that $K'(M)=0 \Rightarrow b(M)=0$ as every matrix of $\calF$ has a sole eigenvalue in $\Kbar$.
Using the rank theorem, we deduce that $\dim \calF \leq \dim K'(\calF)+2$, whereas
Theorem \ref{1spectheorem} yields $\dim K'(\calF) \leq 2$. We deduce that $\dim K'(\calF)=2$.
Proposition \ref{1specredn=2} yields a matrix $P \in \GL_2(\K)$ such that $PK'(\calF)P^{-1}=\K I_2\oplus \NT_2(\K)$.
Conjugating $\calF$ with $1 \oplus P$, we see that no generality is lost in assuming that $K'(\calF)=\K I_2\oplus \NT_2(\K)$.
In that case, all the elements of $\calF$ are upper-triangular, and they must have equal diagonal entries since
$\calF$ is a $\overline{1}$-spec subspace. This yields $\calF \subset \K I_3 \oplus \NT_3(\K)$,
and therefore $\calF=\K I_3 \oplus \NT_3(\K)$ as the dimensions are equal. As this contradicts the assumption that
$\calF$ be exceptional, one concludes that $\dim \calF x \geq 2$ for every non-zero vector $x \in \K^3$.
\end{proof}

From there, we may show that the conditions (i) and (ii) in Theorem \ref{car3theo1} and Theorem \ref{car3theo2}
are incompatible. In what follows, we denote by $(e_1,\dots,e_n)$ the canonical basis of $\K^n$.

\begin{lemme}\label{transitivitylemma}
Assume that $n \geq 3$.
Let $I$ be a non-empty and proper subset of $\lcro 1,n\rcro$,
$\calF$ and $\calG$ be two exceptional $\overline{1}$-spec subspaces of $\Mat_3(\K)$, and
$p \in \lcro 0,n-3\rcro$.
Set $\calU:=\calV_I^{(2)}$, $\calV:=\calW_{p,\calF}^{(2)}$ and $\calW:=\calF \vee \calG$.
\begin{enumerate}[(a)]
\item Given $k \in \lcro 1,n\rcro$, one has
$\dim \calU x \leq k$ for all $x \in \Vect(e_1,\dots,e_k)$.
\item Let $x \in \K^n$. Then:
\begin{enumerate}[(i)]
\item $\dim \calV x \leq k$ if $x \in \Vect(e_1,\dots,e_k)$ for some $k \in \lcro 1,p\rcro$;
\item $p+2 \leq \dim \calV x \leq p+3$ if $x \in \Vect(e_1,\dots,e_{p+3}) \setminus \Vect(e_1,\dots,e_p)$;
\item $\dim \calV x \geq p+4$ if $x \in \K^n \setminus \Vect(e_1,\dots,e_{p+3})$.
\end{enumerate}
\item Assume that $n=6$. Then
$\dim \calW x \leq 3$ for every $x \in \Vect(e_1,e_2,e_3)$, whereas
$ \dim \calW x \geq 5$ for every $x \in E \setminus \Vect(e_1,e_2,e_3)$.
\end{enumerate}
\end{lemme}

\begin{proof}
We note that the overall shapes of the spaces $\calU$, $\calV$ and $\calW$
are unchanged should one conjugate them with an upper-triangular matrix.
Since any space of the form $\Vect(e_1,\dots,e_k)$ is stabilized by such a matrix,
and as, for every non-zero vector $x \in \K^n$, there exists an upper-triangular matrix $T$ such that $Tx \in \{e_1,\dots,e_n\}$,
we are reduced to the situation where $x$ belongs to $\{e_1,\dots,e_n\}$. However,
$\dim \calU x=i$ is obvious if $x=e_i$ for some $i \in \lcro 1,n\rcro$, so point (a) is proved.
For point (b), we see that $\dim \calV x=i$ if $x=e_i$ for some $i \in \lcro 1,n\rcro \setminus \{p+1,p+2,p+3\}$,
while Lemma \ref{exceplemma1} entails that $p+2 \leq \dim \calV x \leq p+3$ if $x \in \{e_{p+1},e_{p+2},e_{p+3}\}$.
Thus, point (b) is proved. Point (c) is obtained in a similar fashion.
\end{proof}

Now, we are ready to conclude:

\begin{prop}\label{uniqueness1starcar3}
Assume that $n \geq 3$.
Let $\calF$ be an exceptional $\overline{1}$-spec subspace of $\Mat_3(\K)$, and let
$p \in \lcro 0,n-3\rcro$. Then:
\begin{enumerate}[(a)]
\item If $n \neq 3$, there is no non-empty and proper subset $I$ of $\lcro 1,n\rcro$ such that $\calW_{p,\calF}^{(2)} \simeq \calV_I^{(2)}$.
\item There is no non-empty subset $I$ of $\lcro 1,n\rcro$ such that $\calW_{p,\calF}^{(1^\star)} \simeq \calV_I^{(1^\star)}$.
\end{enumerate}
\end{prop}

\begin{proof}
Set $\calW':=\calW_{p,\calF}^{(1^\star)}$ and $\calW:=\calW_{p,\calF}^{(2)}$.
\begin{enumerate}[(a)]
\item Assume that $n \neq 3$ and that there exists a non-empty and proper subset $I$ of $\lcro 1,n\rcro$ such that
$\calW \simeq \calV_I^{(2)}$. By point (a) of Lemma \ref{transitivitylemma},
this gives rise to a $(p+1)$-dimensional subspace $F$ of $\K^n$ such that $\dim \calW x \leq p+1$
for every $x \in F$. Point (b) of Lemma \ref{transitivitylemma} then shows that $F \subset \Vect(e_1,\dots,e_p)$,
which leads to $\dim F \leq p$. This contradiction yields point (a).
\item Assume that there exists a non-empty subset $I$ of $\lcro 1,n\rcro$  such that
$\calW' \simeq \calV_I^{(1^\star)}$. If $n=3$, then $\calW'=\calF$ and therefore $\calV_I^{(1^\star)}$ contains $I_3$.
This yields $I=\lcro 1,3\rcro$ and therefore $\calV_I^{(1^\star)}=\K I_3 \oplus \NT_3(\K)$, contradicting the assumption that $\calF$ be exceptional.
Therefore, $n>3$ and $\calW'$ does not contain $I_n$. We deduce that $I \subsetneq \lcro 1,n\rcro$, and therefore
$\calW=\K I_n\oplus \calW' \simeq \K I_n \oplus \calV_I^{(1^\star)}=\calV_I^{(2)}$. Finally, point (a) yields a contradiction.
\end{enumerate}
\end{proof}

\begin{prop}
Let $\calF$ and $\calG$ be two exceptional $\overline{1}$-spec subspaces of $\Mat_3(\K)$, and let
$(p,q) \in \lcro 0,n-3\rcro^2$.
\begin{enumerate}[(a)]
\item If $n>3$ and $\calW_{p,\calF}^{(2)} \simeq \calW_{q,\calG}^{(2)}$, then $p=q$ and $\calF \simeq \calG$.
\item If $\calW_{p,\calF}^{(1^\star)} \simeq \calW_{q,\calG}^{(1^\star)}$, then $p=q$ and $\calF \simeq \calG$.
\end{enumerate}
\end{prop}

\begin{proof}
For convenience, we set $\calW_1:=\calW_{p,\calF}^{(2)}$ and $\calW_2:=\calW_{q,\calG}^{(2)}$.
\begin{enumerate}[(a)]
\item Assume that $n>3$ and $\calW_1 \simeq \calW_2$.
Assume that $p<q$.
By point (b) of Lemma \ref{transitivitylemma} applied to $\calW_2$,
the assumption $\calW_1 \simeq \calW_2$ gives rise to a $(p+1)$-dimensional subspace $F$ of $\K^n$
such that $\dim \calW_1 x \leq p+1$ for every $x \in \K^n$. However,
$\dim \calW_1 x \geq p+2$ for every $x \in \K^n \setminus \Vect(e_1,\dots,e_p)$, so we would have $F \subset \Vect(e_1,\dots,e_p)$,
in contradiction with $\dim F=p+1$. This shows that $p \geq q$, and, symmetrically, one obtains $q \geq p$.
Therefore, $p=q$. \\
Choose $P \in \GL_n(\K)$ such that $P\calW_1 P^{-1}=\calW_2$. Then, $\Vect(e_1,\dots,e_p)$ is the set of all vectors
$x \in \K^n$ such that $\dim (\calW_1 x) \leq p$, and the same holds for $\calW_2$ instead of $\calW_1$.
On the other hand, $\Vect(e_1,\dots,e_{p+3})$ is the set of all vectors
$x \in \K^n$ such that $\dim (\calW_1 x) \leq p+3$, and the same holds for $\calW_2$ instead of $\calW_1$.
It follows that $P$ stabilizes $\Vect(e_1,\dots,e_p)$ and $\Vect(e_1,\dots,e_{p+3})$, leading to
$$P=\begin{bmatrix}
[?]_{p \times p} & [?]_{p \times 3} & [?]_{p \times (n-p-3)} \\
[0]_{3 \times p} & R & [?]_{3 \times (n-p-3)} \\
[0]_{(n-p-3) \times p} & [0]_{(n-p-3) \times 3} & [?]_{(n-p-3) \times (n-p-3)}
\end{bmatrix},$$
for some $R \in \GL_3(\K)$. One deduces that $R \calF R^{-1}=\calG$, and hence $\calF \simeq \calG$.
\item The result is trivial if $n=3$. If $n>3$, then we note that $\calW_1=\K I_n \oplus \calW_{p,\calF}^{(1^\star)} \simeq \K I_n\oplus
\calW_{p,\calG}^{(1^\star)}=\calW_2$, and we apply point (a).
\end{enumerate}
\end{proof}

We finish by tackling the special case $n=6$ for $\overline{2}$-spec subspaces.

\begin{prop}
Assume that $n=6$.
Let $\calF$ and $\calG$ be two exceptional $\overline{1}$-spec subspaces of $\Mat_3(\K)$.
Then:
\begin{enumerate}[(a)]
\item There is no non-empty and proper subset $I$ of $\lcro 1,n\rcro$ such that $\calF \vee \calG \simeq \calV_I^{(2)}$.
\item There is no integer $p \in \lcro 0,n-3\rcro$ and no exceptional $\overline{1}$-spec subspace $\calF'$ of $\Mat_3(\K)$
such that $\calF \vee \calG\simeq \calW_{p,\calF'}^{(2)}\; .$
\item Given exceptional $\overline{1}$-spec subspaces $\calF'$ and $\calG'$ of $\Mat_3(\K)$,
the similarity $\calF \vee \calG \simeq \calF' \vee \calG'$ implies $\calF \simeq \calF'$ and $\calG \simeq \calG'$.
\end{enumerate}
\end{prop}

\begin{proof}
By point (c) of Lemma \ref{transitivitylemma}, no $4$-dimensional subspace $F$ of $\K^6$
satisfies $\dim (\calF \vee \calG)x \leq 4$ for all $x \in F$ (as such a subspace would be included in $\Vect(e_1,e_2,e_3)$).
Using point (a) of Lemma \ref{transitivitylemma} with $k=4$, this proves point (a).

Assume that there exists an integer $p \in \lcro 0,3\rcro$ and an exceptional $\overline{1}$-spec subspace $\calF'$ of $\Mat_3(\K)$
such that $\calF \vee \calG\simeq \calW_{p,\calF'}^{(2)}\; .$
If $p>0$, then point (b) of Lemma \ref{transitivitylemma} would yield a non-zero vector $x \in \K^6$ such
that $\dim (\calF \vee \calG) x \leq 1$, contradicting point (c) of the same lemma.
Therefore, $p=0$ and one sees that $\dim \calW_{0,\calF}^{(2)} x\leq 4$ for all $x \in \Vect(e_1,e_2,e_3,e_4)$.
This would yield a $4$-dimensional subspace $F$ of $\K^6$ such that $\dim (\calF \vee \calG)x \leq 4$ for all $x \in F$.
Again, this would contradict point (c) in Lemma \ref{transitivitylemma}, whence point (b) is proved.

Let finally $\calF'$ and $\calG'$ be two exceptional $\overline{1}$-spec subspaces of $\Mat_3(\K)$,
and assume that there exists $P \in \GL_6(\K)$ such that $P(\calF \vee \calG)P^{-1}=\calF' \vee \calG'$.
Note that $\Vect(e_1,e_2,e_3)$ is the set of all vectors $x \in \K^6$ such that $\dim (\calF \vee \calG)x \leq 3$,
and the same holds for $\calF' \vee \calG'$ instead of $\calF \vee \calG$. We deduce that $P$ stabilizes
$\Vect(e_1,e_2,e_3)$ and may therefore be written as $P=\begin{bmatrix}
R_1 & [?]_{3 \times 3} \\
[0]_{3 \times 3} & R_2
\end{bmatrix}$, where $R_1$ and $R_2$ are matrices of $\GL_3(\K)$.
Therefore, $R_1 \calF R_1^{-1}=\calF'$ and $R_2 \calG R_2^{-1}=\calG'$, yielding point (c).
\end{proof}

\section{Some basic lemmas}\label{basiclemmasection}

\subsection{The case $n=2$}\label{n=2section}

The case $n=2$ in Theorems \ref{1starspecequality} and \ref{car3theo1}
is based upon the following lemma, which will be used frequently in the rest of the article:

\begin{lemme}\label{n=2lemma}
Let $\K$ be an arbitrary field of characteristic not $2$.
Let $A=\begin{bmatrix}
a & 0 \\
b & c
\end{bmatrix} \in \Mat_2(\K)$. Assume that the linear span of $A$ and $\begin{bmatrix}
0 & 1 \\
0 & 0
\end{bmatrix}$ is a $\overline{1}^\star$-spec subspace of $\Mat_2(\K)$. Then,
$b=0$, and either $a=c$ or $a=0$ or $c=0$.
\end{lemme}

\begin{proof}
For every $t \in \K$, the matrix
$\begin{bmatrix}
a & t \\
b & c
\end{bmatrix}$ has determinant $ac-bt$ and the discriminant of its characteristic polynomial is $(a-c)^2+4bt$.
At least one of those quantities is zero. If $b \neq 0$, then at most two elements $t$ of $\K$ satisfy either $ac-bt=0$ or $(a-c)^2+4bt=0$.
Therefore, $b=0$. Finally, if $a \neq 0$ and $c \neq 0$, then $a$ and $c$ are non-zero eigenvalues of $A$, and therefore $a=c$.
\end{proof}

Now, let $\calV$ be a $2$-dimensional $\overline{1}^\star$-spec subspace of $\Mat_2(\K)$, where $\K$ is a field
with characteristic not $2$.
Then, $\calV$ contains a non-zero matrix $B$ with trace $0$.
Its eigenvalues in $\overline{\K}$ are then opposite one of the other, and would therefore be different if non-zero.
Hence, $B$ is nilpotent. As $B$ is non-zero, we lose no generality in assuming that $B=\begin{bmatrix}
0 & 1 \\
0 & 0
\end{bmatrix}$. As $\calV$ has dimension $2$, it contains another matrix $A$ which is linearly independent from $B$, and, combining
this matrix with a scalar multiple of $B$, we lose no generality in assuming that $A$ has the form given in Lemma \ref{n=2lemma}.
We deduce that $A$ is a non-zero scalar multiple of either $I_2$, $E_{1,1}$ or $E_{2,2}$.
As $\calV$ has dimension $2$ and also contains $E_{1,2}$, we deduce that $\calV$ equals $\calV^{(1^\star)}_{\{1,2\}}$,
$\calV^{(1^\star)}_{\{1\}}$ or $\calV^{(1^\star)}_{\{2\}}$. This finishes the proof of Theorems \ref{1starspecequality} and
\ref{car3theo1} in the case $n=2$.

\subsection{Additional useful lemmas}

\begin{lemme}[Linear form lemma, type (I)]\label{linearformlemma}
Assume that $\K$ has more than $2$ elements.
Let $I$ be a non-empty subset of $\lcro 1,n\rcro$, and
$\varphi : \calV_I^{(1^\star)} \rightarrow \K$ be a linear form. Assume that,
for every $M \in \calV_I^{(1^\star)}$ with a non-zero eigenvalue $\lambda$ in $\K$,
one has $\varphi(M)\in \bigl\{0,\lambda\}$. Then, either $\varphi=0$ or
$\varphi$ assigns to every matrix of $\calV_I^{(1^\star)}$ its diagonal entry at the spots
of the form $(i,i)$ with $i \in I$.
\end{lemme}

\begin{proof}
We note that $\varphi$ induces an affine map $\psi$ from the affine subspace $\calD_I+\NT_n(\K)$
to $\K$, and $\psi$ cannot be onto as its only possible values are $0$ and $1$. It follows that
$\psi$ is constant, and hence $\varphi$ vanishes everywhere on $\NT_n(\K)$.
As $\varphi(\calD_I) \in \{0,1\}$, one concludes by using the linearity of $\varphi$.
\end{proof}

\begin{lemme}\label{lemmegeneralphietpsi}
Let $p$ be a positive integer, and $\varphi : \Mat_{1,p}(\K) \rightarrow \Mat_{1,p}(\K)$ and $\psi : \Mat_{p,1}(\K) \rightarrow \Mat_{p,1}(\K)$
be linear maps. Assume that, for every $(L,C) \in \Mat_{1,p}(\K) \times \Mat_{p,1}(\K)$,
the condition $LC=0$ implies $L\psi(C)=0$ and $\varphi(L)C=0$. Then, both $\varphi$ and $\psi$ are scalar multiples of the identity.
\end{lemme}

\begin{proof}
Let $L \in \Mat_{1,p}(\K)$. The assumptions show that the kernel of $\varphi(L)$ contains that of $L$, which yields that
$\varphi(L)$ is a scalar multiple of $L$. As $\varphi$ is linear, we deduce that it is a scalar multiple of the identity. The same line of reasoning
applies to $\psi$.
\end{proof}

\begin{lemme}\label{homotheticsprop}
Assume that $\K$ has characteristic not $2$.
Let $p$ be a positive integer, and let $f : \Mat_{1,p}(\K) \rightarrow \K$, $g : \Mat_{p,1}(\K) \rightarrow \K$,
$\varphi : \Mat_{1,p}(\K) \rightarrow \Mat_{1,p}(\K)$ and $\psi : \Mat_{p,1}(\K) \rightarrow \Mat_{p,1}(\K)$ be linear maps.
For $(L,C) \in \Mat_{1,p}(\K) \times \Mat_{p,1}(\K)$, set
$$A_L:=\begin{bmatrix}
0 & L & 0 \\
[0]_{p \times 1} & [0]_{p \times p} & [0]_{p \times 1} \\
f(L) & \varphi(L) & 0
\end{bmatrix} \quad \text{and} \quad
B_C:=\begin{bmatrix}
0 & [0]_{1 \times p} & 0 \\
\psi(C) & [0]_{p \times p} & C \\
g(C) & [0]_{1 \times p} & 0
\end{bmatrix}.$$
Assume:
\begin{enumerate}[(i)]
\item Either that every linear combination of $A_L$ and $B_C$ has at most one non-zero eigenvalue in $\overline{\K}$, for all $(L,C) \in \Mat_{1,p}(\K) \times \Mat_{p,1}(\K)$.
\item Or that $p \neq 2$ and that every linear combination of $A_L$ and $B_C$ has at most two eigenvalues in $\overline{\K}$, for all $(L,C) \in \Mat_{1,p}(\K) \times \Mat_{p,1}(\K)$.
\end{enumerate}
Then, there are scalars $\lambda$ and $\mu$ such that:
$$\forall (L,C) \in \Mat_{1,p}(\K) \times \Mat_{p,1}(\K), \; \varphi(L)=\lambda\, L \; \text{and} \; \psi(C)=\mu\, C.$$
\end{lemme}

\begin{proof}
The result is obvious if $p=1$. Assume now that $p \geq 2$. \\
Denote by $(e_1,\dots,e_{p+2})$ the canonical basis of $\K^{p+2}$.
By Lemma \ref{lemmegeneralphietpsi}, it suffices to show that $L\psi(C)=0$ and $\varphi(L)C=0$ for every
$(L,C) \in \Mat_{1,p}(\K) \times \Mat_{p,1}(\K)$ satisfying $LC=0$. \\
Let $(L,C) \in \Mat_{1,p}(\K) \times \Mat_{p,1}(\K)$ be such that $LC=0$, $L \neq 0$ and $C \neq 0$.
Set $x:=\begin{bmatrix}
0 \\
C \\
0
\end{bmatrix} \in \Mat_{p+2,1}(\K)$ and remark that $A_L x=\varphi(L)C\,e_{p+2}$, $A_L e_{p+2}=0$, $B_C x=0$ and $B_C e_{p+2}=x$.
In other words, $A_L$ and $B_C$ both stabilize the plane $\Vect(x,e_{p+2})$ and the matrices in $(x,e_{p+2})$
of their induced endomorphisms are
$$\begin{bmatrix}
0 & 0 \\
\varphi(L)C & 0
\end{bmatrix} \quad \text{and} \quad
\begin{bmatrix}
0 & 1 \\
0 & 0
\end{bmatrix} \; \text{, respectively.}$$
If condition (i) is satisfied, then every linear combination of these $2 \times 2$ matrices has at most one non-zero eigenvalue
in $\Kbar$.
This is also true if condition (ii) is satisfied: in that case indeed, we see that $0$ is an eigenvalue of
$\alpha A_L+\beta B_C$ for all $(\alpha,\beta)\in \K^2$, because $p \geq 3$ and $\rk(\alpha A_L+\beta B_C) \leq 4$
(note that, starting from $\alpha A_L+\beta B_C$, we find the zero matrix by deleting the first and last rows and then the first and last columns).

In any case, Lemma \ref{n=2lemma} shows that $\varphi(L)\,C=0$. \\
On the other hand, setting $y:=\begin{bmatrix}
0 \\
L^T \\
0
\end{bmatrix} \in \Mat_{p+2,1}(\K)$, we see that $A_L^T$ and $B_C^T$ all stabilize the plane $\Vect(y,e_1)$,
with induced endomorphisms represented in the basis $(y,e_1)$ by the matrices:
$$\begin{bmatrix}
0 & 1 \\
0 & 0
\end{bmatrix} \quad \text{and} \quad
\begin{bmatrix}
0 & 0 \\
L\psi(C) & 0
\end{bmatrix}, \; \text{respectively.}$$
Again, we deduce that $L\psi(C)=0$. This proves the claimed results.
\end{proof}

\begin{lemme}\label{lemmegeneralfetg}
Assume that $\K$ has more than $2$ elements.
Let $p \geq 1$ be an integer, and  $f : \Mat_{1,p}(\K) \rightarrow \K$ and $g : \Mat_{p,1}(\K) \rightarrow \K$
be two linear forms. Assume that $f(L)+g(C)=0$ for every $(L,C)\in \Mat_{1,p}(\K) \times \Mat_{p,1}(\K)$
for which $LC \neq 0$. Then, $f=0$ and $g=0$.
\end{lemme}

\begin{proof}
Let $L \in \Mat_{1,p}(\K) \setminus \{0\}$. Then, we may choose $C_0 \in \Mat_{p,1}(\K)$ such that $LC_0 \neq 0$.
Therefore $f(L)+s\,g(C_0)=f(L)+g(s C_0)=0$ for every $s \in \K \setminus \{0\}$.
As $\K$ has more than $2$ elements, this yields $f(L)=0$ (and $g(C_0)=0$).
We conclude that $f=0$, and the same line of reasoning yields $g=0$.
\end{proof}

\begin{lemme}\label{lemmeALetBC}
Assume that $\K$ has characteristic not $2$. Let $p$ be a positive integer, $(\lambda,\mu)\in \K^2$
and $f : \Mat_{1,p}(\K) \rightarrow \K$ and $g : \Mat_{p,1}(\K) \rightarrow \K$ be two linear forms such that, for every
$(L,C) \in \Mat_{1,p}(\K) \times \Mat_{p,1}(\K)$, the matrix
$$M_{L,C}=\begin{bmatrix}
0 & L & 0 \\
\mu\, C & 0 & C \\
f(L)+g(C) & \lambda\, L & 0
\end{bmatrix}$$
has at most two eigenvalues in $\Kbar$.
Then, $\lambda+\mu=0$. \\
If in addition $\K$ has characteristic not $3$, then $f=0$ and $g=0$.
\end{lemme}

\begin{proof}
Let $(L,C) \in \Mat_{1,p}(\K) \times \Mat_{p,1}(\K)$ be such that $C \neq 0$. Denote by $(e_1,\dots,e_{p+2})$ the canonical basis of $\K^{p+2}$.
Set $x:=\begin{bmatrix}
0 \\
C \\
0
\end{bmatrix}$. We note that $M_{L,C}$ stabilizes the $3$-dimensional subspace
$\Vect(e_1,x,e_{p+2})$, with induced endomorphism represented in the basis $(e_1,x,e_{p+2})$ by:
$$N=\begin{bmatrix}
0 & LC & 0 \\
\mu & 0 & 1 \\
f(L)+g(C) & \lambda LC & 0
\end{bmatrix}.$$
Therefore, $N$ has at most two eigenvalues in $\overline{\K}$.
However, the characteristic polynomial of $N$ is $t^3-(\lambda+\mu)LC\,t-(f(L)+g(C))LC$.
Therefore, its discriminant is zero, i.e.\ $4(\lambda+\mu)^3(LC)^3=27 (f(L)+g(C))^2(LC)^2$,
and hence
\begin{equation}\label{equadiscriminant}
LC\bigl(4(\lambda+\mu)^3 LC-27 (f(L)+g(C))^2\bigr)=0.
\end{equation}
If $\K$ has characteristic $3$, then one finds $\lambda+\mu=0$ by choosing $(L,C)$ so that $LC=1$. \\
Assume now that $\K$ has characteristic not $3$.
On the left hand-side of identity \eqref{equadiscriminant} is a homogeneous polynomial of degree $4$ in the variable $(L,C)$;
as $\K$, having characteristic neither $2$ nor $3$, has more than $3$ elements, the associated formal polynomial is zero.
However, $(L,C) \mapsto LC$ is non-zero, and hence:
$$\forall (L,C) \in \Mat_{1,p}(\K) \times \Mat_{p,1}(\K), \quad
4(\lambda+\mu)^3 LC=27 (f(L)+g(C))^2.$$
Fixing $C=0$ and varying $L$ yields $f=0$. Similarly, $g=0$. Choosing finally $(L,C)$ so that $LC=1$, one concludes that $\lambda+\mu=0$.
\end{proof}

\section{Large spaces of matrices with at most one non-zero eigenvalue (I)}\label{equalitysection1}

In this section, we assume that $\K$ has characteristic not 2, and we prove Theorems \ref{1starspecequality}
and \ref{car3theo1} in the very special situation where there exists a $\calV$-good vector $x$ together with
a matrix $A \in \calV$ with column space $\K x$ and $\tr(A) \neq 0$.

We shall first outline the general strategy that is to be applied
in the proofs of Theorems \ref{1starspecequality} to \ref{car3theo2}.
Then, we will turn to the specifics of the above case.

\subsection{The diagonal-compatibility method}\label{diagonalcompmethod}

Let $\calV$ be a $\overline{1}^\star$-spec linear subspace of $\Mat_n(\K)$ with dimension $\dbinom{n}{2}+1$.
We may as well view $\calV$ as a set of linear endomorphisms of $\K^n$. For to prove
Theorem \ref{1starspecequality}, it is essential to find a basis $(e'_1,\dots,e'_n)$ in
which all the elements of $\calV$ are represented by upper-triangular matrices.
We will obtain such a basis step-by-step and, to simplify things, we will replace $\calV$
with a succession of similar linear subspaces in order to ``purify" the form of the matrices of $\calV$, until we find
only upper-triangular matrices. Here is the standard sequence of choices in this method:

\begin{itemize}
\item The last vector $e'_n$ in chosen among the $\calV$-good vectors (this is surely a necessary condition!).
Then, we use an induction hypothesis to obtain a basis $(\overline{e'_1},\dots,\overline{e'_{n-1}})$ of the quotient space $\K^n/\K e'_n$
that is ``well-suited" to $\calV$.
\item At this point, each one of the vectors $e'_1,\dots,e'_{n-1}$ should be determined \emph{up to addition of a vector of $\K e'_n$.}
\item A reasonable choice of $e'_2,\dots,e'_{n-1}$ is then obtained by applying the induction hypothesis once more.
\item A reasonable choice of $e'_1$ comes last.
\end{itemize}

One of the key features, for which the method earns its name, is the use of two subspaces of $\calV$
that spring up at looking at things on the upper-left and lower-right corners, and the importance
of connecting the information between those two spaces by studying the diagonal ``middle" part (this is best exemplified in
\cite{dSPaffinenonsingular}). Another key feature is the focus on matrices of small rank, which help
refine the understanding of the structure of $\calV$.

In the study of $\overline{1}^\star$-spec spaces with maximal dimension, it will be convenient
to single out a special case:

\begin{Def}\label{propertyR}
Let $\calV$ be a $\overline{1}^\star$-spec linear subspace of $\Mat_n(\K)$. We say that $\calV$ has \textbf{property (R)}
when there exists a $\calV$-good vector $x$ and some $A \in \calV$ such that $\im A=\K x$ (so that $\tr A \neq 0$).
\end{Def}

We shall begin by proving the following special case of Theorems \ref{1starspecequality} and \ref{car3theo1}:

\begin{theo}\label{specialtheo}
Let $\K$ be a field with characteristic not $2$.
Let $\calV$ be a $\overline{1}^\star$-spec subspace of $\Mat_n(\K)$ with dimension $\dbinom{n}{2}+1$ and property (R).
Then, $\calV$ is similar to $\calV^{(1^\star)}_{\{n\}}$.
\end{theo}

This theorem can indeed be proved by induction with no reference to the more general Theorem \ref{1starspecequality}
(but a necessary reference to Gerstenhaber's theorem, as we shall see).
Moreover, the proof is substantially simpler than in the general case.

In the case $\calV$ has property (R), we choose $e'_n$ as a $\calV$-good vector given by property (R), and
then we choose some $D \in \calV$ with image $\K e'_n$. The space $\Ker D$ is then a natural candidate for
$\Vect(e'_1,\dots,e'_{n-1})$, which will simplify the proof.

\subsection{Setting things up, and setting the basis right}\label{setup1}

We aim at proving Theorem \ref{specialtheo} by induction on $n$.
For $n=2$, we already know that Theorem \ref{1starspecequality} holds (see Section \ref{n=2section}),
so all we need to prove for that case is that neither $\calV^{(1^\star)}_{\{1\}}$ nor
$\calV^{(1^\star)}_{\{1,2\}}$ has property (R). Denote by $(e_1,e_2)$ the canonical basis of $\K^2$.

\begin{itemize}
\item For $\calV=\calV^{(1^\star)}_{\{1\}}$, every non-zero matrix of $\calV^{(1^\star)}_{\{1\}}$ has rank $1$ and image $\K e_1$,
and $e_1$ is obviously not $\calV$-good, therefore $\calV$ does not have property (R).
\item For $\calV=\calV^{(1^\star)}_{\{1,2\}}$, the rank $1$ matrices of $\calV$ are the matrices of the form
$\begin{bmatrix}
0 & b \\
0 & 0
\end{bmatrix}$ with $b \in \K \setminus \{0\}$, and all of them have trace zero; therefore
$\calV$ does not have property (R).
\end{itemize}

This settles the case $n=2$ in Theorem \ref{specialtheo}.

\vskip 2mm
Now, let $n \geq 3$ be an integer such that Theorem \ref{specialtheo} holds for the integer $n-1$.
Denote by $(e_1,\dots,e_n)$ the canonical basis of $\K^n$.
Let $\calV$ be a $\overline{1}^\star$-spec linear subspace of $\Mat_n(\K)$ with dimension $\dbinom{n}{2}+1$
and property (R). By property (R), we have a $\calV$-good vector $x \in \K^n$ and
a matrix $D$ of $\calV$ with $\tr(D) \neq 0$ and $\im D=\K x$. \\
Multiplying $D$ with an appropriate scalar, we lose no generality in assuming that $\tr(D)=1$.
Note then that $D$ is diagonalisable, so that $\Ker D \oplus \im D=\K^n$.
Replacing $\calV$ with a similar space of matrices, we lose no generality in assuming that
$\Ker D=\Vect(e_1,\dots,e_{n-1})$ and $\im D=\Vect(e_n)$. In that case:
\begin{itemize}
\item[(A)] $\calV$ contains $E_{n,n}$, and $e_n$ is $\calV$-good.
\end{itemize}
We now use the same notations as in Section \ref{proofof1starspecinequality}: every matrix $M$ of $\calV$ is written as
$$M=\begin{bmatrix}
K(M) & C(M) \\
[?]_{1 \times (n-1)} & a(M)
\end{bmatrix},$$
where $K(M)$ and $C(M)$ are respectively $(n-1) \times (n-1)$ and $(n-1) \times 1$ matrices, and $a(M)$ is a scalar.
We also set
$$\calW:=\bigl\{M \in \calV : \; C(M)=0\bigr\} \quad \text{and} \quad \calV_{\ul}:=K(\calW)$$
(the subscript ``ul" stands for ``upper-left").
As $\calV$ contains $E_{n,n}$, a similar argument as in Section \ref{proofof1starspecinequality} shows that
no matrix of $\calV_{\ul}$ has a non-zero eigenvalue in
$\overline{\K}$ (if not, then we can use $E_{n,n}$ to create a matrix in $\calV$ with at least two non-zero eigenvalues in $\overline{\K}$).
Therefore, $\calV_{\ul}$ consists only of nilpotent matrices. By Gerstenhaber's theorem
(see \cite{Mathes,dSPGerstenhaberskew,Serezhkin}), we find
$\dim \calV_{\ul}\leq \dbinom{n-1}{2}$, and equality holds if and only if $\calV_{\ul} \simeq \NT_{n-1}(\K)$.

As $e_n$ is $\calV$-good, the rank theorem yields:
$$\dim \calV \leq \dim \calV_{\ul}+1+\dim C(\calV)\leq \binom{n-1}{2}+1+(n-1)=\binom{n}{2}+1.$$
As $\dim \calV=\dbinom{n}{2}+1$, we deduce that $\dim \calV_{\ul}=\dbinom{n-1}{2}$, which gives us:
\begin{itemize}
\item[(B)] There is a non-singular matrix $Q \in \GL_{n-1}(\K)$ such that $Q \calV_{\ul} Q^{-1}=\NT_{n-1}(\K)$.
\end{itemize}
Now, setting $P:=Q \oplus 1$, we replace $\calV$ with $P\,\calV\,P^{-1}$.
Then, property (A) is still satisfied in the new space $\calV$, while property (B) is improved as follows:
\begin{itemize}
\item[(B')] One has $\calV_{\ul}=\NT_{n-1}(\K)$.
\end{itemize}

From there, we intend to prove that $\calV=\calV_{\{n\}}^{(1^\star)}$.

\begin{claim}\label{e1tgood}
No non-zero matrix of $\calV$ has all last $n-1$ columns equal to $0$.
\end{claim}

\begin{proof}
Assume that such a matrix $A$ exists. Then, $A \in \calW$ and hence $K(A) \in \calV_{\ul}=\NT_{n-1}(\K)$.
As $K(A)$ has all columns zero starting from the second one, we deduce that $K(A)=0$.
Therefore, $A$ would be a non-zero scalar multiple of $E_{n,1}$, which would contradict the fact that $e_n$ is $\calV$-good.
\end{proof}

\subsection{Diagonal-compatibility, and matrices of special type}\label{corner1section}

Now, we write every matrix $M$ of $\calV$ as
$$M=\begin{bmatrix}
b(M) & R(M) \\
[?]_{(n-1) \times 1} & K'(M)
\end{bmatrix},$$
where $R(M)$ and $K'(M)$ are respectively $1 \times (n-1)$ and $(n-1) \times (n-1)$ matrices, and $b(M)$ is a scalar.
We set
$$\calW':=\bigl\{M \in \calV : \; R(M)=0\bigr\} \quad \text{and} \quad
\calV_{\lr}:=K'(\calW'),$$
(the subscript ``lr" stands for ``lower-right").

Using Claim \ref{e1tgood} and the same line of reasoning as in Section \ref{proofof1starspecinequality}, we find that
$\calV_{\lr}$ is a $\overline{1}^\star$-spec subspace of $\Mat_{n-1}(\K)$ and
$$\dim \calV \leq \dim \calV_{\lr}+(n-1) \leq \binom{n-1}{2}+1+(n-1)=\binom{n}{2}+1,$$
and hence $\dim \calV_{\lr}=\dbinom{n-1}{2}+1$.

\begin{Rem}\label{blockmatremark}
In what follows, every matrix of $\Mat_n(\K)$ will be written as a block matrix with the following shape:
$$M=\begin{bmatrix}
? & [?]_{1 \times (n-2)} & ? \\
[?]_{(n-2) \times 1} & [?]_{(n-2) \times (n-2)} & [?]_{(n-2) \times 1} \\
? & [?]_{1 \times (n-2)} & ?
\end{bmatrix}.$$
The question marks in the corners represent scalars.
\end{Rem}

We intend to use the induction hypothesis to obtain $\calV_{\lr}=\calV_{\{n-1\}}^{(1^\star)}$.
In order to do so, we need to consider some special matrices of $\calV$.
First of all, for every $U \in \NT_{n-2}(\K)$, property (B') yields that $\calV$ contains a matrix of the form
$$\begin{bmatrix}
0 & 0 & 0 \\
0 & U & 0 \\
? & ? & 0
\end{bmatrix},$$
and hence $\calV_{\lr}$ contains a matrix of the form
$$\begin{bmatrix}
U & 0 \\
? & 0
\end{bmatrix}.$$
On the other hand, property (A) yields that $E_{n-1,n-1} \in \calV_{\lr}$.
From there, we prove:

\begin{claim}\label{compatclaim1}
One has $\calV_{\lr}=\calV_{\{n-1\}}^{(1^\star)}$.
\end{claim}

\begin{proof}
Denote by $(f_1,\dots,f_{n-1})$ the canonical basis of $\K^{n-1}$.

First of all, we note that $f_{n-1}$ is $\calV_{\lr}$-good:
if indeed it were not, then $\calV$ would contain a non-zero matrix of the form
$M=\begin{bmatrix}
? & 0 & 0 \\
? & 0 & 0 \\
? & ? & 0
\end{bmatrix}$; as $K(\calW)=\NT_{n-1}(\K)$, we would have $K(M)=0$ and hence
$\im M=\K e_n$ with $\tr(M)=0$, which would contradict the fact that $e_n$ is $\calV$-good.
On the other hand, we know that $E_{n-1,n-1} \in \calV_\lr$, whence $\calV_{\lr}$ has property (R)
and $f_{n-1}$ is $\calV_\lr$-good. As $\calV_{\lr}$
is a $\overline{1}^\star$-spec subspace of $\Mat_{n-1}(\K)$ with dimension $\dbinom{n-1}{2}$, the induction hypothesis
yields a non-singular matrix $Q \in \GL_{n-1}(\K)$ such that $\calV_{\lr}=Q \calV_{\{n-1\}}^{(1^\star)}Q^{-1}$.
Set $\calA:=\calV_{\{n-1\}}^{(1^\star)}$ for convenience.
Firstly, we see that the $\calA$-bad vectors are precisely the vectors of $\Vect(f_1,\dots,f_{n-2})$.
However, $Q^{-1}E_{n-1,n-1}Q$ is a matrix of $\calA$ with rank $1$ and non-zero trace, and its
image must be spanned by an $\calA$-good vector since $f_{n-1}$ is $\calV_\lr$-good. Using the shape of $\calA$, this shows that
$\Ker(Q^{-1}E_{n-1,n-1}Q)=\Vect(f_1,\dots,f_{n-2})$, and hence $Q$ stabilizes $\Vect(f_1,\dots,f_{n-2})$.
Therefore, $Q=\begin{bmatrix}
Q_1 & [?]_{(n-2) \times 1} \\
[0]_{1 \times (n-2)} & ?
\end{bmatrix}$ for some $Q_1 \in \GL_{n-2}(\K)$.
It follows that every matrix of $\calV_{\lr}$ has the form
$\begin{bmatrix}
[?]_{(n-2) \times (n-2)} & [?]_{(n-2) \times 1} \\
[0]_{1 \times (n-2)} & ?
\end{bmatrix}$.

For every $U \in \NT_{n-2}(\K)$, we already knew that $\calV_{\lr}$
contained a matrix of the form $\begin{bmatrix}
U & 0 \\
? & 0
\end{bmatrix}$, and therefore $\calV_\lr$ contains $\begin{bmatrix}
U & 0 \\
0 & 0
\end{bmatrix}$. As $Q^{-1}\calV_{\lr} Q = \calA$, we deduce that
$Q_1^{-1}\NT_{n-2}(\K) Q_1 \subset \NT_{n-2}(\K)$, and hence Lemma \ref{normalizerlemma} shows that $Q_1^{-1}$ is upper-triangular.
Therefore, $Q$ is upper-triangular, which yields
$$\calV_{\lr}=Q \calA Q^{-1}=\calA=\calV_{\{n-1\}}^{(1^\star)}.$$
\end{proof}

\begin{claim}\label{linearformclaim1}
For every $M \in \calW'$ for which $K'(M)$ is nilpotent, one has $b(M)=0$.
\end{claim}

\begin{proof}
Claim \ref{e1tgood} yields a non-zero linear map $\beta : \calV_{\lr} \rightarrow \K$
such that $b(M)=\beta\bigl(K'(M)\bigr)$ for all $M \in \calW'$.
Lemma \ref{linearformlemma} then applies to $\beta$, which proves our claim.
\end{proof}

Now, we use the previous results to exhibit special ``elementary" matrices in $\calV$.
First of all, properties (A) and (B') yield:
\begin{itemize}
\item[(C)] There are linear maps $f : \Mat_{1,n-2}(\K) \rightarrow \Mat_{1,n-2}(\K)$ and $\varphi : \Mat_{1,n-2}(\K) \rightarrow \K$ such that, for
every $L \in \Mat_{1,n-2}(\K)$, the space $\calV$ contains
$$A_L:=\begin{bmatrix}
0 & L & 0 \\
0 & 0 & 0 \\
f(L) & \varphi(L) & 0
\end{bmatrix}.$$
\end{itemize}
On the other hand, combining property (B') with Claim \ref{compatclaim1} yields:
\begin{itemize}
\item[(D)] There is a linear form $h : \NT_{n-2}(\K) \rightarrow \K$ such that, for every $U \in \NT_{n-2}(\K)$, the space
$\calV$ contains the matrix
$$E_U:=\begin{bmatrix}
0 & 0 & 0 \\
0 & U & 0 \\
h(U) & 0 & 0
\end{bmatrix}.$$
\end{itemize}
Finally, combining Claims \ref{compatclaim1} and \ref{linearformclaim1} yields:
\begin{itemize}
\item[(E)] There are linear maps $g : \Mat_{n-2,1}(\K) \rightarrow \Mat_{n-2,1}(\K)$ and $\psi : \Mat_{n-2,1}(\K) \rightarrow \K$ such that, for
every $C \in \Mat_{n-2,1}(\K)$, the space $\calV$ contains the matrix
$$B_C:=\begin{bmatrix}
0 & 0 & 0 \\
\psi(C) & 0 & C \\
g(C) & 0 & 0
\end{bmatrix}.$$
\end{itemize}

The next step is to prove that all the maps $\varphi$, $\psi$, $f$ and $g$ vanish everywhere.

\subsection{The vanishing of $\varphi$, $\psi$, $f$ and $g$}

By Lemma \ref{homotheticsprop}, there are two scalars $\lambda$ and $\mu$ such that
$$\forall (L,C) \in \Mat_{1,n-2}(\K) \times \Mat_{n-2,1}(\K), \; \varphi(L)=\lambda\,L \quad \text{and} \quad  \psi(C)=\mu\,C.$$

\begin{claim}
One has $\lambda=\mu=0$, $f=0$ and $g=0$.
\end{claim}

\begin{proof}
Lemma \ref{lemmeALetBC} immediately yields $\mu=-\lambda$.
Let $(L,C)\in \Mat_{1,n-2}(\K) \times \Mat_{n-2,1}(\K)$ be such that $LC \neq 0$ (note that such a pair actually exists!).
Set $x:=\begin{bmatrix}
0 \\
C \\
0
\end{bmatrix}$. We note that the three matrices $A_L$, $B_C$ and $E_{n,n}$ all stabilize the $3$-dimensional subspace
$\Vect(e_1,x,e_n)$, with induced endomorphisms represented in the basis $(e_1+\lambda e_n,x,e_n)$ by
$$\begin{bmatrix}
0 & LC & 0 \\
0 & 0 & 0 \\
f(L) & 0 & 0
\end{bmatrix}, \;
\begin{bmatrix}
0 & 0 & 0 \\
0 & 0 & 1 \\
g(C) & 0 & 0
\end{bmatrix}
\; \quad \text{and} \quad
\begin{bmatrix}
0 & 0 & 0 \\
0 & 0 & 0 \\
\lambda & 0 & 1
\end{bmatrix}, \; \text{respectively.}$$
Set $\delta:=f(L)+g(C)$.
We deduce that, for every $\alpha \in \K$, the matrix
$$M_\alpha:=\begin{bmatrix}
0 & LC & 0 \\
0 & 0 & 1 \\
\delta+\lambda\alpha & 0 & \alpha
\end{bmatrix}$$
has at most one non-zero eigenvalue in $\Kbar$.
Yet, the characteristic polynomial of $M_\alpha$ is
$$\chi_\alpha(t)=t^3-\alpha\, t^2-LC\,(\delta+\lambda \alpha).$$
Note that if $\delta+\lambda \alpha \neq 0$, then all the roots of $\chi_\alpha(t)$ are non-zero and hence $\chi_\alpha(t)$ must have a sole
root in $\overline{\K}$.
Assume that $\lambda \neq 0$ or $\delta \neq 0$. Then, we may choose $\alpha \in \K$ such that $\delta+\lambda \alpha\neq 0$ and $\alpha \neq 0$
(since $\K$ has more than $2$ elements). Therefore $\chi_\alpha$ does not have a sole root in $\overline{\K}$:
indeed, if $\K$ has characteristic $3$, then this would imply that the coefficient of $t^2$ is $0$,
and if not, then this would imply that the sole root is $0$ as the coefficient of $t$ is $0$.
We deduce that $\lambda=0$ and $f(L)+g(C)=\delta=0$. From Lemma \ref{lemmegeneralfetg}, we conclude that $f=0$ and $g=0$.
\end{proof}

\subsection{The presence of $E_{1,n}$, and the vanishing of $h$}

We have just proved the following result:

\begin{prop}\label{sumupprop1}
Let $\calV'$ be a $\overline{1}^\star$-spec subspace of $\Mat_n(\K)$
which satisfies properties (A) and (B') and has dimension $\dbinom{n}{2}+1$.
Then, for every $(L,C) \in \Mat_{1,n-2}(\K) \times \Mat_{n-2,1}(\K)$, the space
$\calV'$ contains the matrices
$$\begin{bmatrix}
0 & L & 0 \\
0 & 0 & 0 \\
0 & 0 & 0
\end{bmatrix} \quad \text{and} \quad
\begin{bmatrix}
0 & 0 & 0 \\
0 & 0 & C \\
0 & 0 & 0
\end{bmatrix}.$$
\end{prop}

Set $P_1:=\begin{bmatrix}
1 & 1 \\
0 & 1
\end{bmatrix}\oplus I_{n-2}$ and note that $\calV':=P_1^{-1} \calV P_1$ satisfies properties (A) and (B').
We deduce that $\calV'$ contains $E_{2,n}$, and therefore $\calV$ contains $E_{2,n}+E_{1,n}$.
As we already know that the vector space $\calV$ contains $E_{2,n}$, we deduce:

\begin{claim}
The space $\calV$ contains $E_{1,n}$.
\end{claim}

We use this to prove:

\begin{claim}\label{hclaim1}
One has $h=0$.
\end{claim}

\begin{proof}
Let $U \in \NT_{n-2}(\K)$. We note that both $E_U$ and $E_{1,n}$ stabilize the subspace $\Vect(e_1,e_n)$
and that their induced endomorphisms are represented in the basis $(e_1,e_n)$ by the matrices
$\begin{bmatrix}
0 & 0 \\
h(U) & 0
\end{bmatrix}$ and
$\begin{bmatrix}
0 & 1 \\
0 & 0
\end{bmatrix}$, respectively.
Using Lemma \ref{n=2lemma}, one finds $h(U)=0$.
\end{proof}

\subsection{Conclusion}

We know that $\calV$ contains $E_{1,n}$, $E_{n,n}$, and, for every $(L,C,U)\in \Mat_{1,n-2}(\K) \times \Mat_{n-2,1}(\K) \times \NT_{n-2}(\K)$,
that it contains the matrices
$$\begin{bmatrix}
0 & L & 0 \\
0 & 0 & 0 \\
0 & 0 & 0
\end{bmatrix}, \; \begin{bmatrix}
0 & 0 & 0 \\
0 & 0 & C \\
0 & 0 & 0
\end{bmatrix} \quad \text{and} \quad \begin{bmatrix}
0 & 0 & 0 \\
0 & U & 0 \\
0 & 0 & 0
\end{bmatrix}.$$
By adding these matrices, we obtain the inclusion $\calV_{\{n\}}^{(1^\star)} \subset \calV$, and the equality of dimensions yields
$\calV=\calV_{\{n\}}^{(1^\star)}$.
This finishes our proof of Theorem \ref{specialtheo} by induction on $n$.

\section{Large spaces of matrices with at most one non-zero eigenvalue (II): characteristic not $3$}\label{equalitysection2}

In this section, we assume that $\K$ has characteristic neither $2$ nor $3$, and we prove Theorem \ref{1starspecequality}
by induction on $n$. The case $n=2$ has already been dealt with in Section \ref{n=2section}.

\subsection{Setting the induction up}\label{setup2}

Let $n \geq 3$, and assume that the result of Theorem \ref{1starspecequality} holds for the integer $n-1$.
Let $\calV$ be a $\overline{1}^\star$-spec subspace of $\Mat_n(\K)$ with dimension $\dbinom{n}{2}+1$.
There are two cases when we can conclude right away:
\begin{itemize}
\item If $\calV$ has property (R), then Theorem \ref{specialtheo} yields $\calV \simeq \calV_{\{n\}}^{(1^\star)}$.
\item If $\calV^T$ has property (R), then Theorem \ref{specialtheo} yields $\calV^T \simeq \calV_{\{n\}}^{(1^\star)}$.
A simple reversal of the order of the basis vectors shows that $\bigl(\calV_{\{n\}}^{(1^\star)}\bigr)^T
\simeq \calV_{\{1\}}^{(1^\star)}$, and hence $\calV \simeq \calV_{\{1\}}^{(1^\star)}$.
\end{itemize}
Thus, in the rest of the proof, we assume:
\begin{center}
Neither $\calV$ nor $\calV^T$ has property (R).
\end{center}
This means that if $x$ is a $\calV$-good vector, then no matrix of $\calV$ has $\K x$ as its image
(and the same holds for $\calV^T$ instead of $\calV$).

\vskip 2mm
We denote by $(e_1,\dots,e_n)$ the canonical basis of $\K^n$.
As in Section \ref{setup1}, we lose no generality in assuming that $e_n$ is $\calV$-good.
Then, we write every matrix of $\calV$ as
$$M=\begin{bmatrix}
K(M) & C(M) \\
[?]_{1 \times (n-1)} & a(M)
\end{bmatrix}$$
and we set
$$\calW:=\bigl\{M \in \calV : \; C(M)=0\bigr\} \quad \text{and} \quad
\calV_{\ul}:=K(\calW),$$
so that $\calV_{\ul}$ is a $\overline{1}^\star$-spec subspace of $\Mat_{n-1}(\K)$.
As $\calV$ contains no non-zero matrix $M$ satisfying $K(M)=C(M)=0$, the rank theorem yields
$$\dim \calV=\dim \calV_\ul+\dim C(\calV) \leq \binom{n-1}{2}+1+(n-1)=\binom{n}{2}+1,$$
whence $\dim \calV_\ul=\dbinom{n-1}{2}+1$. Using the induction hypothesis, we deduce:
\begin{itemize}
\item[(A)] There exists a (unique) non-empty subset $I$ of $\lcro 1,n-1\rcro$ such that $\calV_{\ul} \simeq \calV_I^{(1^\star)}$.
\end{itemize}
As in Section \ref{setup1}, we lose no generality in assuming that:
\begin{itemize}
\item[(A')] One has $\calV_\ul=\calV_I^{(1^\star)}$.
\end{itemize}
Since no non-zero matrix of $\calV$ has $\K e_n$ as its image, we have $K(M)=0 \Rightarrow a(M)=0$ for all $M \in \calW$,
which yields a linear form $\alpha : \calV_{\ul} \rightarrow \K$ such that
$$\forall M \in \calW, \; a(M)=\alpha(K(M)).$$

Next, we write every matrix of $\calW$ as
$$M=\begin{bmatrix}
b(M) & R(M) \\
[?]_{(n-1) \times 1} & K'(M)
\end{bmatrix},$$
where $b(M)$ is a scalar, and $R(M)$ and $K'(M)$ are respectively $1 \times (n-1)$ and $(n-1) \times (n-1)$ matrices.
We set
$$\calW':=\bigl\{M \in \calV : \; R(M)=0\bigr\} \quad \text{and} \quad
\calV_{\lr}:=K'(\calW').$$

\begin{claim}\label{e1VTgoodclaim1}
The vector $e_1$ is $\calV^T$-good.
\end{claim}

\begin{proof}
Let $M \in \calV$ be with all columns zero starting from the second one, and $\tr M=0$.
Then, $M \in \calW$ and hence $K(M) \in \calV_\ul$. As $\calV_{\ul}=\calV_I^{(1^\star)}$,
we find that $K(M)=0$. Finally, as $e_n$ is $\calV$-good, we conclude that $M=0$.
\end{proof}

As $\calV^T$ does not have property (R), we deduce that no matrix of $\calV$
has all columns zero starting from the second one. As in Section \ref{corner1section}, we deduce
that $\calV_{\lr}$ is a $\overline{1}^\star$-spec subspace of $\Mat_{n-1}(\K)$ with dimension $\dbinom{n-1}{2}+1$.
Moreover, there exists a linear form $\beta : \calV_{\lr} \rightarrow \K$ such that
$$\forall M \in \calW', \; b(M)=\beta(K'(M)).$$

Using the induction hypothesis, we find the following property:
\begin{itemize}
\item[(B)] There exists a (unique) non-empty subset $J$ of $\lcro 1,n-1\rcro$ such that $\calV_{\lr} \simeq \calV_J^{(1^\star)}$.
\end{itemize}

\begin{claim}\label{fn-1goodclaim}
The last vector of the canonical basis of $\K^{n-1}$ is $\calV_\lr$-good.
\end{claim}

\begin{proof}
Let $M \in \calW'$ be such that $K'(M)$ has all first $(n-2)$ rows zero, and with $\tr K'(M)=0$.
Then, $K'(M)$ is nilpotent and hence Lemma \ref{linearformclaim1}
yields $b(M)=\beta(K'(M))=0$. Thus, $M$ belongs to $\calW$ and $K(M)$ has trace zero and all columns zero starting from the second one.
As $\calV_{\ul}=\calV_I^{(1^\star)}$, we deduce that $K(M)=0$, and therefore the first $n-1$ rows of $M$ are zero.
As $e_n$ is $\calV$-good, we conclude that $M=0$. This proves the claimed statement.
\end{proof}

As $\alpha(\calD_I) \in \{0,1\}$, we set:
$$\widetilde{I}:=\begin{cases}
I & \text{if $\alpha(\calD_I)=0$} \\
I \cup \{n\} & \text{if $\alpha(\calD_I)=1$,}
\end{cases}$$
so that $\calV$ contains a matrix of the form
$$\calD_{\widetilde{I}}+\begin{bmatrix}
[0]_{(n-1) \times (n-1)} & [0]_{(n-1) \times 1} \\
[?]_{1 \times (n-1)} & 0
\end{bmatrix}.$$
Note that Lemma \ref{linearformlemma} yields that $\beta$ vanishes at every nilpotent matrix of $\calV_\lr$,
and hence $\widetilde{I} \neq \{1\}$.
Thus,
$$J':=(\widetilde{I} \setminus \{1\})-1$$
is non-empty, and $\calV_\lr$ contains a matrix of the form
$$\calD_{J'}+\begin{bmatrix}
[0]_{(n-2) \times (n-2)} & [0]_{(n-2) \times 1} \\
[?]_{1 \times (n-2)} & 0
\end{bmatrix}.$$

\subsection{Diagonal-compatibility, and an additional change of basis}

The following result will help us perform one additional change of basis so as to simplify the form of $\calV_{\lr}$.

\begin{claim}\label{corner2claim}
One has $J'=J$,
and there exists a non-singular matrix of the form
$Q=\begin{bmatrix}
I_{n-2} & [0]_{(n-2) \times 1} \\
[?]_{1 \times (n-2)} & 1
\end{bmatrix} \in \GL_{n-1}(\K)$ such that $\calV_{\lr}=Q\,\calV_J^{(1^\star)}\,Q^{-1}$.
\end{claim}

\begin{proof}[Proof of Claim \ref{corner2claim}]
By property (B), we already have
a matrix $Q \in \GL_{n-1}(\K)$ such that
$\calV_{\lr}=Q\,\calV_J^{(1^\star)}\,Q^{-1}$. Set $\calA:=\calV_J^{(1^\star)}$.
Note that for any \emph{upper-triangular} and non-singular matrix $Q'\in \GL_{n-1}(\K)$, we have $\calA=Q'\,\calA\,(Q')^{-1}$.

Denote by $(f_1,\dots,f_{n-1})$ the canonical basis of $\K^{n-1}$.
On one hand, Claim \ref{fn-1goodclaim} tells us that $f_{n-1}$ is $\calV_{\lr}$-good.
On the other hand, we see that no vector of $\Vect(f_1,\dots,f_{n-2})$ is $\calA$-good,
and hence equality $Q^{-1}\calV_{\lr}Q=\calA$ shows that $Q^{-1}f_{n-1} \not\in \Vect(f_1,\dots,f_{n-2})$.
Therefore, we may find a non-singular upper-triangular matrix $Q'\in \GL_{n-1}(\K)$ such that
$Q^{-1} f_{n-1} =Q' f_{n-1}$. As $Q'$ is upper-triangular, we may now replace
$Q$ with $QQ'$, in which case we have the additional property $Qf_{n-1}=f_{n-1}$.

Now, we write $Q=\begin{bmatrix}
R & [0]_{(n-2) \times 1} \\
L_1 & 1
\end{bmatrix}$ with $R \in \GL_{n-2}(\K)$ and $L_1 \in \Mat_{1,n-2}(\K)$.
Then,
$$Q=Q_1Q_2 \quad \text{where} \; Q_1:=\begin{bmatrix}
I_{n-2} & [0]_{(n-2) \times 1} \\
L_1R^{-1} & 1
\end{bmatrix} \quad \text{and} \quad
Q_2:=R \oplus 1.$$
In order to conclude, it remains to prove that $J'=J$ and that $\calE:=Q_2 \calA Q_2^{-1}$ equals $\calV_J^{(1^\star)}$.
Let $U \in \NT_{n-2}(\K)$. By property (A'), $\calV_{\ul}$ contains $\begin{bmatrix}
0 & [0]_{1 \times (n-2)} \\
[0]_{(n-2) \times 1} & U
\end{bmatrix}$.
It follows that $\calV_{\lr}$ contains a matrix of the form $\begin{bmatrix}
U & [0]_{(n-2) \times 1} \\
[?]_{1 \times (n-2)} & ?
\end{bmatrix}$, and hence this is also the case of $\calE=Q_1^{-1} \calV_\lr Q_1$.
One concludes that $\calA$ contains a matrix of the form
$$\begin{bmatrix}
R^{-1}UR & [0]_{(n-2) \times 1} \\
[?]_{1 \times (n-2)} & ?
\end{bmatrix}$$
for every $U \in \NT_{n-2}(\K)$. From the definition of $\calA$, it follows that
$R^{-1} U R$ is upper-triangular for all $U \in \NT_{n-2}(\K)$, and hence even strictly upper-triangular as it is nilpotent.
Then, one deduces from Lemma \ref{normalizerlemma} that $R^{-1}$ is upper-triangular,
and hence $Q_2$ is upper-triangular. It follows that $\calE=Q_2 \calA Q_2^{-1}=\calV_J^{(1^\star)}$.

Finally, we know that $\calV_\lr$ contains a matrix of the form
$\calD_{J'}+\begin{bmatrix}
[0]_{(n-2) \times (n-2)} & [0]_{(n-2) \times 1} \\
[?]_{1 \times (n-2)} & 0
\end{bmatrix}$, and hence this is also the case of $\calE=Q_1^{-1} \calV_\lr Q_1$.
As $\calE=\calV_J^{(1^\star)}$ and $J'$ is non-empty, this yields $J=J'$.
One concludes that
$$\calV_\lr=Q_1 \calE Q_1^{-1}=Q_1 \calV_{J}^{(1^\star)} Q_1^{-1}.$$
\end{proof}

Now, we fix some $Q$ given by Claim \ref{corner2claim}, set $P':=1 \oplus Q$ and replace $\calV$ with $(P')^{-1}\calV P'$.
One checks that $e_n$ is still $\calV$-good and that $\calV_{\ul}$ has been unchanged in the process,
while (B) is now replaced with the more precise property:
\begin{itemize}
\item[(B')] One has $\calV_{\lr}=\calV_J^{(1^\star)}$.
\end{itemize}

\subsection{Special matrices in $\calV$}\label{specialmat2}

From now on, we write every matrix of $\Mat_n(\K)$ as a block matrix of the following shape:
$$M=\begin{bmatrix}
? & [?]_{1 \times (n-2)} & ? \\
[?]_{(n-2) \times 1} & [?]_{(n-2) \times (n-2)} & [?]_{(n-2) \times 1} \\
? & [?]_{1 \times (n-2)} & ?
\end{bmatrix}.$$

Remember the linear forms $\alpha$ and $\beta$ defined in Section \ref{setup2}.
Applying Lemma \ref{linearformlemma} to both of them, we find that $\alpha$ and $\beta$
vanish at every nilpotent matrix (of $\calV_{\ul}$ and $\calV_{\lr}$, respectively).
Using properties (A') and (B'), we find linear maps $\varphi : \Mat_{1,n-2}(\K) \rightarrow \Mat_{1,n-2}(\K)$,
$\psi : \Mat_{n-2,1}(\K) \rightarrow \Mat_{n-2,1}(\K)$, $f : \Mat_{1,n-2}(\K) \rightarrow \K$,
$g : \Mat_{n-2,1}(\K) \rightarrow \K$ and $h : \NT_{n-2}(\K) \rightarrow \K$ such that, for every
$(L,C,U) \in \Mat_{1,n-2}(\K) \times \Mat_{n-2,1}(\K) \times \NT_{n-2}(\K)$, the space $\calV$
 contains the matrices
$$A_L:=\begin{bmatrix}
0 & L & 0 \\
0 & 0 & 0 \\
f(L) & \varphi(L) & 0
\end{bmatrix}
\; , \; B_C:=\begin{bmatrix}
0 & 0 & 0 \\
\psi(C) & 0 & C \\
g(C) & 0 & 0
\end{bmatrix} \; \text{and} \;
E_U:=\begin{bmatrix}
0 & 0 & 0 \\
0 & U & 0 \\
h(U) & 0 & 0
\end{bmatrix}.$$
Remember that $\calV$ contains a matrix of the form $$\calD_{\widetilde{I}}+\begin{bmatrix}
0 & 0 & 0 \\
0 & 0 & 0 \\
? & ? & 0
\end{bmatrix}.$$
As such a matrix belongs to $\calW'$, equality $\calV_\lr=\calV_J^{(1^\star)}$ yields:

\begin{itemize}
\item[(C)] The space $\calV$ contains $\calD_{\widetilde{I}}+d\,E_{n,1}$ for some $d \in \K$.
\end{itemize}
From there, our aim is to prove that $\calV \simeq \calV_{\widetilde{I}}^{(1^\star)}$. This will require one last change of basis.

\subsection{Analyzing $\varphi$ and $\psi$, and performing one last change of basis}

Lemma \ref{homotheticsprop} yields:
\begin{itemize}
\item[(D)] There are scalars $\lambda$ and $\mu$ such that:
$$\forall (L,C) \in \Mat_{1,n-2}(\K) \times \Mat_{n-2,1}(\K), \; \varphi(L)=\lambda\, L \; \text{and} \; \psi(C)=\mu\, C.$$
\end{itemize}

From there, Lemma \ref{lemmeALetBC} readily yields:

\begin{claim}
One has $\lambda+\mu=0$, $f=0$ and $g=0$.
\end{claim}

Now, let us set
$$P'':=\begin{bmatrix}
1 & 0 & 0 \\
0 & I_{n-2} & 0 \\
\lambda & 0 & 1
\end{bmatrix}$$
and replace $\calV$ with $(P'')^{-1} \calV P''$.
One checks that properties (A') and (B') are unchanged by this transformation (with the same sets $I$, $J$ and $\widetilde{I}$),
but now we have the improved property:

\begin{itemize}
\item[(E)] For every $(L,C) \in \Mat_{1,n-2}(\K) \times \Mat_{n-2,1}(\K)$, the space $\calV$
contains the matrices
$$A_L=\begin{bmatrix}
0 & L & 0 \\
0 & 0 & 0 \\
0 & 0 & 0
\end{bmatrix} \quad \text{and} \quad
B_C=\begin{bmatrix}
0 & 0 & 0 \\
0 & 0 & C \\
0 & 0 & 0
\end{bmatrix}.$$
\end{itemize}

We may sum up the previous reductions as follows:

\begin{prop}\label{sumupprop2}
Let $\calU$ be a $\overline{1}^\star$-spec linear subspace of $\Mat_n(\K)$ with dimension $\dbinom{n}{2}+1$.
Assume that:
\begin{enumerate}[(i)]
\item Neither $\calU$ nor $\calU^T$ has property (R);
\item The vector $e_n$ is $\calU$-good;
\item There exists a non-empty subset $I$ of $\lcro 1,n-1\rcro$ such that $\calU_{\ul}=\calV_I^{(1^\star)}$.
\end{enumerate}
Then, there exists a matrix $Q=\begin{bmatrix}
I_{n-1} & [0]_{(n-1) \times 1} \\
[?]_{1 \times (n-1)} & 1
\end{bmatrix}\in \GL_n(\K)$ such that, for every $(L,C) \in \Mat_{1,n-2}(\K) \times \Mat_{n-2,1}(\K)$, the space
$Q^{-1}\calU Q$ contains the matrices
$$\begin{bmatrix}
0 & L & 0 \\
0 & 0 & 0 \\
0 & 0 & 0
\end{bmatrix} \quad \text{and} \quad
\begin{bmatrix}
0 & 0 & 0 \\
0 & 0 & C \\
0 & 0 & 0
\end{bmatrix}.$$
\end{prop}

From there, we aim at proving that $\calV=\calV_{\widetilde{I}}^{(1^\star)}$.

\subsection{The presence of $E_{1,n}$ in $\calV$}\label{E1nsection1}

\begin{claim}\label{E1ninVclaim1}
The matrix $E_{1,n}$ belongs to $\calV$.
\end{claim}

\begin{proof}
Set $P_1:=\begin{bmatrix}
1 & 1 \\
0 & 1
\end{bmatrix} \oplus I_{n-2} \in \GL_n(\K)$ and remark that $P_1^{-1} \calV P_1$
satisfies the assumptions of Proposition \ref{sumupprop2} because $P_1$ is upper-triangular and $P_1 e_n=e_n$.

This yields a matrix of the form $P_2=\begin{bmatrix}
I_{n-1} & [0]_{(n-1) \times 1} \\
L_1 & 1
\end{bmatrix}$ such that
$P_2^{-1}P_1^{-1} \calV P_1P_2$ contains $E_{2,n}$.
It follows that we may find a triple $(a,b,L_2)\in \K^2 \times \Mat_{1,n-2}(\K)$ such that $\calV$ contains the matrix
$$H=(P_1P_2) E_{2,n} (P_1P_2)^{-1}=\begin{bmatrix}
-a & -L_2 & 1 \\
-a & -L_2 & 1 \\
[0]_{(n-3) \times 1} & [0]_{(n-3) \times (n-2)} & [0]_{(n-3) \times 1} \\
-ab & -b\,L_2 & b
\end{bmatrix}$$
and the first entry of $L_2$ is $b-a$. To be more specific,
$a$ and $b$ are the first two entries of $L_1$, and $L_1$ and $L_2$ share the same last $n-3$ entries.

Note that $H$ is similar to $E_{2,n}$ so it has rank $1$ and trace $0$.
Yet, for every $(L,C)\in \Mat_{1,n-2}(\K) \times \Mat_{n-2,1}(\K)$,
each one of the matrices $A_L$ and $B_C$ has trace $0$ and rank at most one,
therefore Lemma \ref{tracelemma} yields $\tr(A_LH)=0=\tr(B_CH)$.
Taking $L=\begin{bmatrix}
1 & 0 & \cdots & 0\end{bmatrix}$ and $C:=L^T$, we deduce that $a=0$ and $b(a-b)=0$,
which yields $a=b=0$. \\
Note that all the rows of $H$ are zero starting from the third one, and the first two entries in its first two rows are zero.
By linearly combining $H$ with well-chosen matrices of type $E_U$, $A_L$ and $B_C$, we deduce that
$\calV$ contains a matrix of the form
$$H'=\begin{bmatrix}
0 & [0]_{1 \times (n-2)} & 1 \\
[0]_{(n-2) \times 1} & [0] & [0]_{(n-2) \times 1} \\
\delta & [0]_{1 \times (n-2)} & 0
\end{bmatrix}.$$
The characteristic polynomial of $H'$ is $t^n-\delta\, t^{n-2}$, and hence $\delta=0$ as $H'$
has at most two eigenvalues in $\overline{\K}$. Therefore, $E_{1,n}=H'$ belongs to $\calV$.
\end{proof}

\subsection{The presence of $\calD_{\widetilde{I}}$ in $\calV$}

\begin{claim}\label{lastclaim2}
The space $\calV$ contains $\calD_{\widetilde{I}}$.
\end{claim}

\begin{proof}
We know that $\calV$ contains $\calD_{\widetilde{I}}+d\,E_{n,1}$ for some $d \in \K$.
The matrices $\calD_{\widetilde{I}}+d\,E_{n,1}$ and $E_{1,n}$ both stabilize the space $\Vect(e_1,e_n)$, and
their induced endomorphisms are represented by
$\begin{bmatrix}
? & 0 \\
d & ?
\end{bmatrix}$ and $\begin{bmatrix}
0 & 1 \\
0 & 0
\end{bmatrix}$ in the basis $(e_1,e_n)$, respectively. Then, Lemma \ref{n=2lemma} shows that $d=0$, and hence $\calV$ contains $\calD_{\widetilde{I}}$.
\end{proof}

\subsection{Conclusion}

Now that we have $E_{1,n}$ in $\calV$, we can apply the same proof as for Claim \ref{hclaim1} to obtain:

\begin{claim}\label{hclaim2}
The map $h$ vanishes everywhere on $\NT_{n-2}(\K)$.
\end{claim}

Let us conclude. For every $(L,C,U)\in \Mat_{1,n-2}(\K) \times \Mat_{n-2,1}(\K) \times \NT_{n-2}(\K)$, the space
$\calV$ contains the matrices
$$\begin{bmatrix}
0 & L & 0 \\
0 & 0 & 0 \\
0 & 0 & 0
\end{bmatrix} \quad , \quad
\begin{bmatrix}
0 & 0 & 0 \\
0 & 0 & C \\
0 & 0 & 0
\end{bmatrix}\quad \text{and} \quad
\begin{bmatrix}
0 & 0 & 0 \\
0 & U & 0 \\
0 & 0 & 0
\end{bmatrix}.$$
Moreover, we know that $\calV$ contains $E_{1,n}$ and $\calD_{\widetilde{I}}$.
Adding those matrices, we obtain $\calV_{\widetilde{I}}^{(1^\star)} \subset \calV$.
As those spaces have the same dimension, we conclude that $\calV=\calV_{\widetilde{I}}^{(1^\star)}$.
This finishes our proof of Theorem \ref{1starspecequality} by induction.

\section{Large spaces of matrices with at most two eigenvalues (I): characteristic not $3$}\label{equalitysection3}

In this section, $\K$ denotes an arbitrary field of characteristic neither $2$ nor $3$.
Our aim is to derive Theorem \ref{2specequality} from Theorem \ref{1starspecequality}.
The line of reasoning will be largely similar to the one in Section \ref{equalitysection2}, the main difference being
that the proof is not done by induction on $n$. Besides, the case $n=4$ is especially difficult and
requires a far more extensive inquiry.

\subsection{Setting things up}\label{setup3section}

Let $n \geq 3$ and $\calV$ be a $\overline{2}$-spec subspace of $\Mat_n(\K)$ with dimension $\dbinom{n}{2}+2$.
We first note that $I_n \in \calV$. Indeed, Theorem \ref{2specinequality} shows that $\calV$ is a maximal $\overline{2}$-spec
subspace of $\Mat_n(\K)$, while $\K I_n+\calV$ is obviously a $\overline{2}$-spec subspace of $\Mat_n(\K)$ which contains it.

Denote by $(e_1,\dots,e_n)$ the canonical basis of $\K^n$. Using Proposition \ref{goodprop}, we find that no generality is lost in assuming that
$e_n$ is $\calV$-good. As in Section \ref{2specinequalitysection},
any $M \in \calV$ may now be written as
$$M=\begin{bmatrix}
K(M) & C(M) \\
[?]_{1 \times (n-1)} & a(M)
\end{bmatrix},$$
and we define
$$\calW:=\bigl\{M \in \calV : \; C(M)=0 \; \text{and}\; a(M)=0\bigr\}
\quad \text{and} \quad
\calV_{\ul}:=K(\calW),$$
so that $\calV_{\ul}$ is a $\overline{1}^\star$-spec subspace of $\Mat_{n-1}(\K)$.
As $e_n$ is $\calV$-good, the rank theorem yields:
$$\dim \calV \leq \dim \calV_{\ul}+n \leq \binom{n-1}{2}+1+n=\binom{n}{2}+2,$$
and hence $\dim \calV_{\ul}=\dbinom{n-1}{2}+1$. Then, Theorem \ref{1starspecequality} yields a non-empty subset
$I$ of $\lcro 1,n-1\rcro$ such that:
\begin{itemize}
\item[(A)] $\calV_{\ul} \simeq \calV_I^{(1^\star)}$.
\end{itemize}
From there, we aim at proving that $\calV \simeq \calV_I^{(2)}$.
As in Section \ref{setup2}, a change of the first $n-1$ basis vectors shows that we lose no generality in assuming:
\begin{itemize}
\item[(A')] $\calV_{\ul}=\calV_I^{(1^\star)}$.
\end{itemize}
With (A'), we obtain that $e_1$ is $\calV^T$-good (the proof is the same as that of Claim \ref{e1VTgoodclaim1}).
With the above line of reasoning, it follows that, writing every $M \in \calV$ as
$$M=\begin{bmatrix}
b(M) & R(M) \\
[?]_{(n-1) \times 1} & K'(M)
\end{bmatrix},$$
and setting
$$\calW':=\bigl\{M \in \calV : \; R(M)=0 \; \text{and}\; b(M)=0\bigr\}
\quad \text{and} \quad \calV_{\lr}:=K'(\calW'),$$
one finds that $\calV_{\lr}$ is a $\overline{1}^\star$-spec subspace of $\Mat_{n-1}(\K)$ with dimension $\dbinom{n-1}{2}+1$.
Thus:
\begin{itemize}
\item[(B)] There exists a non-empty subset $J$ of $\lcro 1,n-1\rcro$ such that $\calV_{\lr} \simeq \calV_J^{(1^\star)}$.
\end{itemize}
If $1 \not\in I$, we note that $\calV_\ul$ contains a matrix of the form
$$\begin{bmatrix}
0 & [0]_{1 \times (n-2)}  \\
[0]_{(n-2) \times 1} & \calD_{I-1}
\end{bmatrix}.$$
If $1 \in I$, we see that
$\calV_\ul$ contains a matrix of the form
$$I_{n-1}-\begin{bmatrix}
1 & [0]_{1 \times (n-2)}  \\
[0]_{(n-2) \times 1} & \calD_{I-1}
\end{bmatrix}.$$
In any case, setting
$$J':=\begin{cases}
I-1 & \text{if $1 \not\in I$} \\
(\lcro 1,n\rcro \setminus I)-1 & \text{otherwise,}
\end{cases}$$
we see that $J'$ is non-empty and that $\calV_\lr$ contains a matrix of the form
$$\calD_{J'}+\begin{bmatrix}
[0]_{(n-2) \times (n-2)} & [0]_{(n-2) \times 1} \\
[?]_{1 \times (n-2)} & 0
\end{bmatrix}.$$

\subsection{Diagonal-compatibility, and an additional change of basis}

Now, we prove the diagonal-compatibility result:

\begin{claim}\label{corner3claim}
One has $J=J'$, and there exists a non-singular matrix of the form
$Q=\begin{bmatrix}
I_{n-2} & [0]_{(n-2) \times 1} \\
[?]_{1 \times (n-2)} & 1
\end{bmatrix}\in \GL_{n-1}(\K)$ such that $\calV_{\lr}=Q\,\calV_J^{(1^\star)}\,Q^{-1}$.
\end{claim}

\begin{proof}
Borrowing the line of reasoning from the proof of Claim \ref{corner2claim}, one sees that it suffices to establish the following
three points:
\begin{enumerate}[(i)]
\item The last vector of the canonical basis of $\K^{n-1}$ is $\calV_{\lr}$-good.
\item For every $U \in \NT_{n-2}(\K)$, the space $\calV_{\lr}$ contains a matrix of the form
$\begin{bmatrix}
U & [0]_{(n-2) \times 1} \\
[?]_{1 \times (n-2)} & 0
\end{bmatrix}$.
\item The space $\calV_\lr$ contains a matrix of the form $\calD_{J'}+\begin{bmatrix}
[0]_{(n-2) \times (n-2)} & [0]_{(n-2) \times 1} \\
[?]_{1 \times (n-2)} & 0
\end{bmatrix}$.
\end{enumerate}
For statement (i), we note that if $\calV_{\lr}$ contains a matrix $N$ with trace zero and all rows zero except the last one,
then $\calW \cap \calW'$ contains a matrix $M$ such that $K'(M)=N$, and $K(M)$ has all columns zero starting from the second one
and zero as its first row. Using $\calV_{\ul}=\calV_I^{(1^\star)}$, we deduce that $K(M)=0$,
and hence $M=0$ since $e_n$ is $\calV$-good. Therefore, $N=0$.\\
Statement (ii) is deduced immediately from $\calV_{\ul}=\calV_I^{(1^\star)}$.
Statement (iii) was proved earlier.
\end{proof}

With $Q$ given by Claim \ref{corner3claim}, we set $P:=1\oplus Q$ and replace $\calV$ with $P^{-1}\calV P$.
This does not modify the fact that $e_n$ is $\calV$-good and that $\calV_{\ul}=\calV_I^{(1^\star)}$, but now we have the improved property:
\begin{itemize}
\item[(B')] $\calV_{\lr}=\calV_J^{(1^\star)}$.
\end{itemize}

\subsection{Special matrices in $\calV$}

From now on, we write matrices of $\calV$ as block-matrices as specified in Remark \ref{blockmatremark}.

As in Section \ref{specialmat2}, properties (A') and (B') combined yield linear maps $f,g,h,\varphi,\psi$ such that,
for every $(L,C,U) \in \Mat_{1,n-2}(\K) \times \Mat_{n-2,1}(\K) \times \NT_{n-2}(\K)$, the space $\calV$ contains the matrices
$$A_L=\begin{bmatrix}
0 & L & 0 \\
0 & 0 & 0 \\
f(L) & \varphi(L) & 0
\end{bmatrix} \quad ; \quad
B_C=\begin{bmatrix}
0 & 0 & 0 \\
\psi(C) & 0 & C \\
g(C) & 0 & 0
\end{bmatrix} \quad \text{and} \quad
E_U=\begin{bmatrix}
0 & 0 & 0 \\
0 & U & 0 \\
h(U) & 0 & 0
\end{bmatrix}.$$
Moreover, we claim that:
\begin{itemize}
\item[(C)] $\calV$ contains $\calD_I+d\,E_{n,1}$ for some $d \in \K$.
\end{itemize}

Indeed, we know that $\calV_{\ul}$ contains the $(n-1) \times (n-1)$ matrix $\calD_I$,
so we may find some $M \in \calW$ with $K(M)=\calD_I$. Then, one of the matrices $I_n-M$ or $M$ must belong to
$\calW'$, whence property (B') yields (C) in either case.

\paragraph{}
Now, we must analyze the $\varphi$ and $\psi$ maps. Lemma \ref{homotheticsprop} yields:

\begin{claim}\label{homotheticsclaim82}
If $n \neq 4$, then there are scalars $\lambda$ and $\mu$ such that:
$$\forall (L,C) \in \Mat_{1,n-2}(\K) \times \Mat_{n-2,1}(\K), \; \varphi(L)=\lambda\, L \; \text{and} \; \psi(C)=\mu\, C.$$
\end{claim}

\subsection{Concluding the case $n \neq 4$}\label{not4section}

Assume that $n \neq 4$, and let $\lambda$ and $\mu$ be scalars given by Claim \ref{homotheticsclaim82}.
Then, Lemma \ref{lemmeALetBC} yields that $\lambda+\mu=0$, $f=0$ and $g=0$.
Setting
$$P'':=\begin{bmatrix}
1 & 0 & 0 \\
0 & I_{n-2} & 0 \\
\lambda & 0 & 1
\end{bmatrix}$$
and replacing $\calV$ with $(P'')^{-1} \calV P''$, we see that
the new space $\calV$ still satisfies the previous properties, but now it contains the matrices
$\begin{bmatrix}
0 & L & 0 \\
0 & 0 & 0 \\
0 & 0 & 0
\end{bmatrix}$ and $\begin{bmatrix}
0 & 0 & 0 \\
0 & 0 & C \\
0 & 0 & 0
\end{bmatrix}$ for all $(L,C) \in \Mat_{1,n-2}(\K) \times \Mat_{n-2,1}(\K)$.
We may sum up the previous reduction as follows:

\begin{prop}\label{sumupprop3}
Let $\calU$ be a $\overline{2}$-spec linear subspace of $\Mat_n(\K)$ with dimension $\dbinom{n}{2}+2$.
Assume that:
\begin{enumerate}[(i)]
\item The vector $e_n$ is $\calU$-good;
\item There exists a non-empty subset $I$ of $\lcro 1,n-1\rcro$ such that $\calU_{\ul}=\calV_I^{(1^\star)}$.
\end{enumerate}
Then, there exists a matrix $Q=\begin{bmatrix}
I_{n-1} & [0]_{(n-1) \times 1} \\
[?]_{1 \times (n-1)} & 1
\end{bmatrix}\in \GL_n(\K)$ such that, for every $(L,C) \in \Mat_{1,n-2}(\K) \times \Mat_{n-2,1}(\K)$, the space
$Q^{-1}\calU Q$ contains the matrices
$$\begin{bmatrix}
0 & L & 0 \\
0 & 0 & 0 \\
0 & 0 & 0
\end{bmatrix} \quad \text{and} \quad
\begin{bmatrix}
0 & 0 & 0 \\
0 & 0 & C \\
0 & 0 & 0
\end{bmatrix}.$$
\end{prop}

From there, we may follow the argumentation of Section \ref{E1nsection1} effortlessly to obtain
$E_{1,n} \in \calV$.

Remember that we know that $\calD_I+d\,E_{n,1}$ belongs to $\calV$
for some $d \in \K$.

\begin{claim}
One has $d=0$.
\end{claim}

\begin{proof}
If $\lcro 2,n-1\rcro\not\subset I$, then every linear combination of
$\calD_I+d\,E_{n,1}$ and $E_{1,n}$ is singular, so that it can have at most one non-zero eigenvalue.
With the same chain of arguments as in the proof of Claim \ref{lastclaim2}, one deduces that $d=0$. \\
If $\lcro 2,n-1\rcro \subset I$, then one notes that $\calV$ contains $I_n-(\calD_I+d\,E_{n,1})=\calD_{\lcro 1,n\rcro \setminus I}-d\,E_{n,1}$,
and now $\lcro 1,n\rcro \setminus I$ does not contain $\lcro 2,n-1\rcro$, so that $-d=0$.
\end{proof}

Finally, we have shown that $\calV$ contains $I_n$, $\calD_I$ and any strictly upper-triangular matrix.
We deduce that $\calV_I^{(2)} \subset \calV$, and the equality of dimensions yields $\calV=\calV_I^{(2)}$.
This proves Theorem \ref{2specequality} when $n \neq 4$.

\subsection{The case $n=4$}\label{n=4car<>3}

In this section, we assume that $n=4$.
In that situation, it may not be true that $\varphi$ and $\psi$ are scalar multiples of the identity.
We introduce the two matrices $A$ and $B$ of $\Mat_2(\K)$ such that
$$\forall (L,C) \in \Mat_{1,2}(\K) \times \Mat_{2,1}(\K), \; \varphi(L)=LA \; \text{and} \: \psi(C)=BC.$$

\begin{Rem}\label{changeofbasisrem}
In order to analyze the pair $(A,B)$, we have to understand how it is affected by simple modifications of $\calV$.
\begin{enumerate}[(i)]
\item Let $\lambda \in \K$, and set $P:=\begin{bmatrix}
1 & 0 & 0 & 0 \\
0 & 1 & 0 & 0 \\
0 & 0 & 1 & 0 \\
-\lambda & 0 & 0 & 1
\end{bmatrix}$. We note that $P\calV P^{-1}$ still satisfies conditions (A'), (B') and (C),
and for this new space the pair $(A,B)$ is replaced with $(A-\lambda I_2,B+\lambda I_2)$.

\item Let $T \in \GL_2(\K)$ be upper-triangular, and set $P:=1 \oplus T \oplus 1$.
Again, $P\calV P^{-1}$ still satisfies conditions (A'), (B') and (C),
the pair $(A,B)$ being replaced with $(TAT^{-1},TBT^{-1})$.
\end{enumerate}
\end{Rem}

Let $(L,C) \in \Mat_{1,2}(\K) \times \Mat_{2,1}(\K)$. One checks that the characteristic polynomial of the matrix
$$A_L+B_C=\begin{bmatrix}
0 & L & 0 \\
BC & [0]_{2 \times 2} & C \\
f(L)+g(C) & LA & 0
\end{bmatrix}$$
is
$$t^4-L(A+B)C\, t^2-(f(L)+g(C))\,LC\,t+\bigl((LAC)(LBC)-(LC)(LABC)\bigr).$$
Indeed, the computation of the coefficients of $t^3$, $t^2$ and $t$ is straightforward, while
$$\det(A_L+B_C)=\begin{vmatrix}
L \\
LA
\end{vmatrix}\times \begin{vmatrix}
BC & C
\end{vmatrix}=
\begin{vmatrix}
LBC & LC \\
LABC & LAC
\end{vmatrix}.$$
We want to make good use of the fact that $A_L+B_C$ has at most two eigenvalues in $\Kbar$. The following lemma
on polynomials of degree $4$ will be useful in that prospect:

\begin{lemme}\label{4by4carnot3lemma}
Let $M \in \Mat_4(\K)$. Assume that the characteristic polynomial of $M$ has the form
$\chi(t)=t^4+bt^2+ct+d$ and that $M$ has at most two eigenvalues in $\overline{\K}$. Then:
\begin{enumerate}[(a)]
\item $c=0$ implies $b^2=4d$.
\item $b=0$ implies $c=0$.
\end{enumerate}
\end{lemme}

\begin{proof}
We distinguish between three cases.
\begin{enumerate}[(i)]
\item Assume that $M$ has a triple eigenvalue $x$ together with a single eigenvalue $y$.
Then, $x$ is a multiple root of $\chi'(t)=4t^3+2bt+c$. If
$b=0$, then $\chi'(t)$ can only have a multiple root if $c=0$. Assume now that $c=0$. If $b \neq 0$, then
$\chi'(t)$ has three distinct roots as $4t^2+2b$ has two distinct non-zero roots. Thus, $b=0$ and $\chi(t)=t^4+d$
has four distinct roots if $d \neq 0$. Hence, $4d=0=b^2$.
\item Assume that $M$ has two distinct double eigenvalues $x$ and $y$.
Then $2x+2y=0$, and therefore $y=-x$. It follows that $\chi(t)=(t^2-x^2)^2=t^4-2x^2t^2+x^4$.
Thus, $b^2=4x^4=4d$ and $c=0$.
\item If $M$ has an eigenvalue $x$ with multiplicity $4$, we may set $y:=x$ and apply the same argument as in case (ii).
\end{enumerate}
This proves points (a) and (b) in either case.
\end{proof}

Now, we start analyzing the pair $(A,B)$.

\begin{claim}\label{AmoinsBscalairecar<>3}
There exists $\beta \in \K$ such that $A-B=\beta\,I_2$.
\end{claim}

\begin{proof}
Let $(L,C) \in \Mat_{1,2}(\K) \times \Mat_{2,1}(\K)$ be such that $LC=0$.
Then, the characteristic polynomial of $A_L+B_C$ is $t^4-(LAB+LBC)\, t^2+(LAC)(LBC)$.
Applying point (a) of Lemma \ref{4by4carnot3lemma}, we deduce that $(LAB-LBC)^2=0$, and hence $L(A-B)C=0$.
If we fix $C \in \Mat_{2,1}(\K)$, then we deduce from $\forall L \in \Mat_{1,2}(\K), \; LC=0 \Rightarrow L(A-B)C=0$
that $(A-B)C$ is a scalar multiple of $C$. The claimed statement follows.
\end{proof}

Using point (i) of Remark \ref{changeofbasisrem} with $\lambda:=\dfrac{\beta}{2}$, we find that no generality is lost
in assuming that $A=B$. Then, the characteristic polynomial of $A_L+B_C$ may be written as follows:
$$t^4-2\,LAC\,t^2-(f(L)+g(C))\,LC\,t+\bigl((LAC)^2-(LC)(LA^2C)\bigr).$$

\begin{center}
Until further notice, we assume that $A \neq 0$.
\end{center}

\begin{claim}
One has $f=0$ and $g=0$.
\end{claim}

\begin{proof}
We use a \emph{reductio ad absurdum}, assuming that $f \neq 0$ or $g \neq 0$.
Set $E:=\Mat_{1,2}(\K) \times \Mat_{2,1}(\K)$, and denote by $H$ the kernel of the non-zero linear form
$(L,C) \mapsto f(L)+g(C)$ on $E$.

First of all, we show that $A \not\in \K I_2$. Assume on the contrary that there exists a non-zero scalar $\alpha \in \K$ such that
$A=\alpha I_2$. As $H$ has dimension $3$, the non-degenerate quadratic form $(L,C) \mapsto LC$ does not vanish everywhere on $H$,
which yields some $(L,C) \in H$ such that $LC \neq 0$.
The characteristic polynomial of $A_L+B_C$, being $t^4-2\,\alpha\,LC\,t^2$ with $2\alpha\,LC \neq 0$,
has three distinct roots in $\overline{\K}$, which is forbidden.
Therefore, $A \not\in \K I_2$.

Now, let $(L,C) \in E$. Using Lemma \ref{4by4carnot3lemma}, one finds that
$$(LAC=0 \; \text{and}\; LC \neq 0) \, \Rightarrow f(L)+g(C)=0.$$
Let $C \in \K^2$ be such that $(C,AC)$ is a basis of $\K^2$.
Then, we may find some $L \in \Mat_{1,2}(\K)$ for which $LAC=0$ and $LC \neq 0$.
Moreover, for every $\lambda \in \K \setminus \{0\}$, the row matrix $\lambda L$
satisfies the same conditions. This yields
$$\forall \lambda \in \K \setminus \{0\}, \; \lambda f(L)+g(C)=0.$$
As $\K$ has more than two elements, it follows that $g(C)=0$.
Therefore, $g$ vanishes at every vector of $\K^2$ which is not an eigenvector for $A$.
However, since $A$ is not a scalar multiple of the identity, there are at most two distinct $1$-dimensional subspaces spanned by eigenvectors of $A$.
Therefore, the non-eigenvectors for $A$
span $\K^2$, which yields $g=0$. With the same line of reasoning applied to $A^T$, one finds that $f=0$.
This \emph{reductio ad absurdum} yields $f=0$ and $g=0$.
\end{proof}

\begin{claim}\label{A2egal0car<>3}
One has $A^2=0$.
\end{claim}

\begin{proof}
For every $(L,C) \in \Mat_{1,2}(\K) \times \Mat_{2,1}(\K)$, the characteristic polynomial of $A_L+B_C$
is $t^4-2\,LAC\,t^2+\bigl((LAC)^2-(LC)(LA^2C)\bigr)$, therefore point (a) in Lemma \ref{4by4carnot3lemma}
yields
$$(LC)\,(LA^2C)=0.$$
Fix a non-zero matrix $C \in \Mat_{2,1}(\K)$. Let $L \in \Mat_{1,2}(\K)$. If $LC \neq 0$, then $LA^2C=0$.
It follows that the linear form $L \mapsto LA^2C$ vanishes at every point of $\Mat_{1,2}(\K) \setminus \calD$,
where $\calD:=\{L \in \Mat_{1,2}(\K) : \; LC=0\}$. As $\calD$ is a proper linear subspace of $\Mat_{1,2}(\K)$, the set
$\Mat_{1,2}(\K) \setminus \calD$ spans $\Mat_{1,2}(\K)$ and therefore $LA^2C=0$ for all $L \in \Mat_{1,2}(\K)$.
We deduce that $A^2C=0$. As this holds for every non-zero $C \in \Mat_{2,1}(\K)$, the claimed statement is proved.
\end{proof}

\begin{Rem}
Here is a more simple proof: the map $(L,C) \mapsto (LC)\,(LA^2C)$ is a homogeneous polynomial of degree $4$
on $E:=\Mat_{1,2}(\K) \times \Mat_{2,1}(\K)$. As $\K$ has more than $3$ elements,
we deduce that one of the polynomial maps $(L,C) \mapsto LC$ and $(L,C) \mapsto LA^2C$ vanishes everywhere on $E$.
The first one being obviously non-zero, one deduces that the second one vanishes everywhere on $E$, which yields $A^2=0$.

However, in the prospect of the proof of Theorem \ref{car3theo2}, it will be useful to have a line of reasoning that
can be adapted to all fields of characteristic $3$, in particular to those with three elements.
\end{Rem}

As $A \neq 0$, we deduce that $A$ is a rank $1$ nilpotent $2 \times 2$ matrix.
The next step is to consider the $1$-dimensional subspace $\Ker A$.

\begin{claim}\label{KerAcar<>3}
The space $\Ker A$ is spanned by $\begin{bmatrix}
1 & 0
\end{bmatrix}^T$.
\end{claim}

\begin{proof}
Denote by $(e_1,e_2)$ the canonical basis of $\K^2$.
We use a \emph{reductio ad absurdum} by assuming that $\Ker A$ does not contain $e_1$.
Then, we may find a non-singular upper-triangular matrix $T \in \GL_2(\K)$ such that
$$TAT^{-1}=\begin{bmatrix}
0 & 0 \\
1 & 0
\end{bmatrix}.$$
Indeed, we may choose a non-zero vector $f_2 \in \Ker A$, so that $Ae_1=\alpha f_2$ for some $\alpha \in \K \setminus \{0\}$,
and we define $T$ as the matrix of coordinates of $(e_1,e_2)$ in the basis $(e_1,\alpha f_2)$. \\
Using point (ii) of Remark \ref{changeofbasisrem}, we see that no generality is lost in assuming that $A=\begin{bmatrix}
0 & 0 \\
1 & 0
\end{bmatrix}$. \\
Now, we have a scalar $\delta \in \K$ such that, for every $(\ell_1,\ell_2,c_1,c_2,d) \in \K^5$, the space $\calV$ contains the matrix
$$\begin{bmatrix}
0 & \ell_1 & \ell_2 & 0 \\
0 & 0 & d & c_1 \\
c_1 & 0 & 0 & c_2 \\
\delta d & \ell_2 & 0 & 0
\end{bmatrix}.$$
In particular, $\calV$ contains
$$\begin{bmatrix}
0 & 1 & 0 & 0 \\
0 & 0 & 1 & 0 \\
0 & 0 & 0 & 1 \\
\delta & 0 & 0 & 0
\end{bmatrix},$$
with characteristic polynomial $t^4-\delta$. If $\delta \neq 0$, this matrix would have four distinct eigenvalues in $\Kbar$.
Therefore, $\delta=0$. It follows that $\calV$ contains
$$\begin{bmatrix}
0 & 1 & 0 & 0 \\
0 & 0 & 1 & 1 \\
1 & 0 & 0 & 0 \\
0 & 0 & 0 & 0
\end{bmatrix},$$
 the characteristic polynomial of which is $t(t^3-1)$, with four distinct roots in $\overline{\K}$. This final contradiction shows that
 $e_1 \in \Ker A$.
\end{proof}

We deduce that $A=\begin{bmatrix}
0 & \alpha \\
0 & 0
\end{bmatrix}$ for some $\alpha \in \K \setminus \{0\}$.
Taking the case $A=0$ into account, we sum up some of our recent findings as follows:

\begin{prop}\label{sumupprop4}
Let $\calU$ be a $\overline{2}$-spec linear subspace of $\Mat_4(\K)$ with dimension $8$.
Assume that:
\begin{enumerate}[(i)]
\item The vector $e_4$ is $\calU$-good;
\item There exists a non-empty subset $I$ of $\lcro 1,3\rcro$ such that $\calU_{\ul}=\calV_I^{(1^\star)}$.
\end{enumerate}
Then, there exists a matrix $Q=\begin{bmatrix}
I_3 & [0]_{3 \times 1} \\
[?]_{1 \times 3} & 1
\end{bmatrix}\in \GL_4(\K)$ such that $Q^{-1}\calU Q$ contains $E_{1,3}$ and $E_{2,4}$.
\end{prop}

Now, let us come back to $\calV$ in its reduced form, i.e.\ $A=\begin{bmatrix}
0 & \alpha \\
0 & 0
\end{bmatrix}$ for some $\alpha \in \K$, \emph{possibly with $\alpha=0$.}
Assume first that $A=0$. Then, one can follow the chain of arguments from Section \ref{not4section}
to obtain $\calV=\calV_I^{(2)}$. Indeed, the only difference between the present situation and Section \ref{not4section}
is that the result of Proposition \ref{sumupprop4} is substantially weaker than the one of Proposition \ref{sumupprop3}.
However, one checks that the conclusion $E_{2,n} \in \calU$ is sufficient to adapt the line of reasoning
of the proof of Claim \ref{E1ninVclaim1}.
The case $A=0$ is thus settled.

\paragraph{}
Let us return to the case $A \neq 0$. Replacing $\calV$ with $D^{-1} \calV D$, where
$D:=1 \oplus \alpha \oplus I_2$, we lose no generality in assuming that $A=\begin{bmatrix}
0 & 1 \\
0 & 0
\end{bmatrix}$ (use point (ii) of Remark \ref{changeofbasisrem}). This means that, for every $(l_1,l_2,c_1,c_2)\in \K^4$,
the space $\calV$ contains the matrix
$$\begin{bmatrix}
0 & l_1 & l_2 & 0 \\
c_2 & 0 & 0 & c_1 \\
0 & 0 & 0 & c_2 \\
0 & 0 & l_1 & 0
\end{bmatrix}.$$

\begin{claim}
The matrix $E_{1,4}$ belongs to $\calV$.
\end{claim}

\begin{proof}
Set $P_1:=\begin{bmatrix}
1 & 1 \\
0 & 1
\end{bmatrix} \oplus I_2 \in \GL_4(\K)$ and remark that $P_1^{-1} \calV P_1$
satisfies the assumptions of Proposition \ref{sumupprop4}.
This yields a matrix $P_2=\begin{bmatrix}
1 & 0 & 0 & 0 \\
0 & 1 & 0 & 0 \\
0 & 0 & 1 & 0 \\
a & b & c & 1
\end{bmatrix}$
such that $P_2^{-1}P_1^{-1} \calV P_1P_2$ contains $E_{2,4}$ and $E_{1,3}$.
We deduce that $\calV$ contains both matrices
$$H=\begin{bmatrix}
-a & a-b & -c & 1 \\
-a & a-b & -c & 1 \\
0 & 0 & 0 & 0 \\
-ba & b(a-b) & -bc & b
\end{bmatrix} \quad \text{and} \quad
H'=\begin{bmatrix}
0 & 0 & 1 & 0 \\
0 & 0 & 0 & 0 \\
0 & 0 & 0 & 0 \\
0 & 0 & a & 0
\end{bmatrix}.$$
As $\calV$ contain $H'$ and $E_{1,3}$, it also contains $a\,E_{4,3}$, which yields $a=0$ since $e_4$ is $\calV$-good.
Moreover, $H$ and $E_{2,4}$ belong to $\calV$ and both have rank $1$ and trace $0$.
By Lemma \ref{tracelemma}, this yields $\tr(H E_{2,4})=0$, and hence $b=0$.
It follows that $\calV$ contains the matrix $\begin{bmatrix}
0 & 0 & -c & 1 \\
0 & 0 & -c & 1 \\
0 & 0 & 0 & 0 \\
0 & 0 & 0 & 0
\end{bmatrix}$: by linearly combining it with $E_{1,3}$, $E_{2,4}$ and a matrix of type $E_U$, we deduce that $\calV$ contains
a matrix of the form
$\begin{bmatrix}
0 & 0 & 0 & 1 \\
0 & 0 & 0 & 0 \\
0 & 0 & 0 & 0 \\
\beta & 0 & 0 & 0
\end{bmatrix}$, the characteristic polynomial of which is $t^2(t^2-\beta)$.
Therefore $\beta=0$, and hence $\calV$ contains $E_{1,4}$.
\end{proof}

From there, one follows the line of reasoning from Section \ref{not4section} to obtain
$h=0$ and $\calD_I \in \calV$. In particular, for every $(l_1,l_2,c_1,c_2,a,b)\in \K^6$, the space
$\calV$ contains the matrix
$$\begin{bmatrix}
0 & l_1 & l_2 & b \\
c_2 & 0 & a & c_1 \\
0 & 0 & 0 & c_2 \\
0 & 0 & l_1 & 0
\end{bmatrix}.$$
It remains to investigate the possible values of $I$.

\begin{claim}
Either $I=\{1,3\}$ or $I=\{2,3\}$.
\end{claim}

\begin{proof}
Let $I_0$ be an arbitrary non-empty and proper subset of $\lcro 1,4\rcro$ such that $\calV$ contains $\calD_{I_0}$.
We claim that $I_0$ contains at least one element of each sets $\lcro 1,2\rcro$ and $\lcro 3,4\rcro$.
Assume on the contrary that $I_0 \cap \lcro 3,4\rcro=\emptyset$.
Denote by $\alpha$ and $\beta$ the first two diagonal entries of $\calD_{I_0}$.
Then, $\calV$ contains
$$\begin{bmatrix}
\alpha & 1 & 0 & 0 \\
1 & \beta & 0 & 0 \\
0 & 0 & 0 & 1 \\
0 & 0 & 1 & 0
\end{bmatrix}.$$
As the eigenvalues of $\begin{bmatrix}
0 & 1 \\
1 & 0
\end{bmatrix}$ are $1$ and $-1$, it follows that every eigenvalue of $N=\begin{bmatrix}
\alpha & 1 \\
1 & \beta
\end{bmatrix}$ in $\overline{\K}$ belongs to $\{1,-1\}$.
However, the characteristic polynomial of $N$ is $t^2-(\alpha+\beta)t+(\alpha\beta-1)$,
and one checks that it equals none of the polynomials $(t-1)(t+1)$, $(t-1)^2$ and $(t+1)^2$.
This contradiction yields $I_0 \cap \lcro 3,4\rcro \neq \emptyset$.
With the same line of reasoning, one proves that $I_0 \cap \lcro 1,2\rcro \neq \emptyset$. \\
As $\calV$ contains $I_4$ and $\calD_{I_0}$, it also contains their difference $\calD_{\lcro 1,4\rcro \setminus I_0}$,
and therefore $\lcro 1,4\rcro \setminus I_0$ has at least one common point with $\lcro 1,2\rcro$ and $\lcro 3,4\rcro$.
It follows that $\#\bigl(I_0 \cap \lcro 1,2\rcro\bigr)=\#\bigl(I_0 \cap \lcro 3,4\rcro\bigr)=1$.

Returning to $I$, we knew from the start that $4 \not\in I$. It follows that $I \cap \lcro 3,4\rcro=\{3\}$,
and our claim ensues since $I \cap \lcro 1,2\rcro$ contains exactly one element.
\end{proof}

We may conclude, at last!
\begin{itemize}
\item Assume that $I=\{1,3\}$. By linearly combining the matrices $I_4$, $\calD_I$ and every matrix of the form
$$\begin{bmatrix}
0 & l_1 & l_2 & b \\
c_2 & 0 & a & c_1 \\
0 & 0 & 0 & c_2 \\
0 & 0 & l_1 & 0
\end{bmatrix} \quad \text{with $(l_1,l_2,c_1,c_2,a,b)\in \K^6$,}$$
we find the inclusion $\calG_4'(\K) \subset \calV$,
which entails that $\calV=\calG'_4(\K)$ as both spaces have dimension $8$.
\item Assume that $I=\{2,3\}$.
Set $K:=\begin{bmatrix}
0 & 1 \\
1 & 0
\end{bmatrix}$. Then we have shown that, for every $(B_1,B_2) \in \Mat_2(\K)^2$, the space $\calV$ contains
$\begin{bmatrix}
B_1 & B_2 \\
0 & KB_1K^{-1}
\end{bmatrix}$. Setting $P:=I_2 \oplus K$, we deduce that $P^{-1}\calV P$ contains $\calG_4(\K)$, and the equality
of dimensions yields $P^{-1}\calV P = \calG_4(\K)$.
\end{itemize}

In any case, we have shown that $\calV$ is similar either to $\calV_I^{(2)}$, or to $\calG_4(\K)$, or to $\calG'_4(\K)$.
This finishes our proof of Theorem \ref{2specequality}.

\section{Preliminary results for the characteristic 3 case}\label{basiclemmasectioncar3}

The proofs of Theorems \ref{car3theo1} and \ref{car3theo2}, which will be carried out in the next two sections, will require
a solid understanding of the structure of the exceptional $\overline{1}$-spec subspaces of $\Mat_3(\K)$,
together with a handful of lemmas that are specific to fields of characteristic $3$.
This section is devoted to considerations of this sort. Throughout the section, it is assumed that $\K$
is an arbitrary field of characteristic $3$.

\subsection{Basic considerations on the exceptional $\overline{1}$-spec subspaces of $\Mat_3(\K)$}

Here, we recall some results from Section 4 of \cite{dSPsoleeigenvalue}.

\begin{Not}
For $\delta \in \K$, we set
$$\calF_\delta:=\Vect\Biggl(I_3,
\begin{bmatrix}
0 & 1 & 0 \\
0 & 0 & 0 \\
1 & 0 & 0
\end{bmatrix},
\begin{bmatrix}
0 & 0 & 0 \\
0 & 0 & 1 \\
\delta & 0 & 0
\end{bmatrix},
\begin{bmatrix}
1 & 0 & 1 \\
-1 & 0 & 0 \\
-1 & -\delta & -1
\end{bmatrix}\Biggr) \subset \Mat_3(\K)$$
and
$$\calG_\delta:=\Vect\Biggl(I_3,
\begin{bmatrix}
0 & 1 & 0 \\
0 & 0 & 0 \\
1 & 0 & 0
\end{bmatrix},
\begin{bmatrix}
0 & 0 & 0 \\
0 & 0 & 1 \\
\delta & 0 & 0
\end{bmatrix},
\begin{bmatrix}
0 & 0 & 1 \\
-1 & 0 & 0 \\
0 & -\delta & 0
\end{bmatrix}\Biggr) \subset \Mat_3(\K).$$
\end{Not}

Recall the following result:

\begin{theo}[de Seguins Pazzis \cite{dSPsoleeigenvalue}]\label{dePazzisexcept3theo}
Assume that $\K$ has characteristic $3$, and let $\calV$ be an exceptional $\overline{1}$-spec subspace of $\Mat_3(\K)$.
Then, there exists $\delta \in \K$ such that $\calV \simeq \calF_\delta$ or $\calV \simeq \calG_\delta$.
\end{theo}

We introduce three conditions on exceptional $\overline{1}$-spec subspaces of $\Mat_3(\K)$:

\begin{Def}
Let $\calV$ be an exceptional $\overline{1}$-spec subspace of $\Mat_3(\K)$.
\begin{itemize}
\item We say that $\calV$ is \textbf{fully-reduced} if it \emph{equals} $\calF_\delta$ or $\calG_\delta$ for some $\delta \in \K$.
\item We say that $\calV$ is \textbf{semi-reduced} if it contains three matrices of the following forms:
$$\begin{bmatrix}
0 & 1 & 0 \\
0 & 0 & 0 \\
? & 0 & 0
\end{bmatrix} \quad , \quad
\begin{bmatrix}
0 & 0 & 0 \\
0 & 0 & 1 \\
? & 0 & 0
\end{bmatrix} \quad \text{and} \quad \begin{bmatrix}
? & 0 & 1 \\
? & ? & 0 \\
? & ? & ?
\end{bmatrix}.$$
Notice that, together with $I_3$, those matrices constitute a linearly independent $4$-tuple, which is therefore a basis of $\calV$.
\item We say that $\calV$ is \textbf{well-reduced} if it contains three matrices of the following forms:
$$\begin{bmatrix}
0 & 1 & 0 \\
0 & 0 & 0 \\
\delta & 0 & 0
\end{bmatrix} \quad , \quad
\begin{bmatrix}
0 & 0 & 0 \\
0 & 0 & 1 \\
\epsilon & 0 & 0
\end{bmatrix} \quad \text{and} \quad \begin{bmatrix}
? & 0 & 1 \\
? & ? & 0 \\
? & ? & ?
\end{bmatrix}, \quad \text{with $(\delta,\epsilon) \neq (0,0)$.}$$
\end{itemize}
\end{Def}

In particular, a fully-reduced space is always well-reduced, and a well-reduced space is always semi-reduced
(in each case, the converse implication does not hold in general).

\begin{Rem}\label{fullyreducedremark}
Let $\calV$ be a semi-reduced exceptional $\overline{1}$-spec subspace of $\Mat_3(\K)$.
One checks that $e_3$ is $\calV$-good and that $e_1$ is $\calV^T$-good.
Indeed, if a matrix $M \in \calV$ has its upper-right $2 \times 2$ sub-matrix equal to zero,
one writes $M$ as a linear combination of $I_3$ and three matrices of the above form, and one successively obtains that $M$ belongs to $\K\,I_3$
(by looking at the entries at the $(1,2)$, $(1,3)$ and $(2,3)$-spots), and then that $M=0$ by looking at the entry at the $(2,2)$-spot.
\end{Rem}

\subsection{Additional results on the exceptional $\overline{1}$-spec subspaces of $\Mat_3(\K)$}

The following lemma may be viewed as the characteristic $3$ version of Lemma \ref{linearformlemma}:

\begin{lemme}[Linear form lemma (type (II))]\label{linearformlemma2}
Assume that $n \geq 3$. Let $p \in \lcro 0,n-3\rcro$ and $\calF$ be an exceptional $\overline{1}$-spec subspace of $\Mat_3(\K)$.
Let $\varphi : \calW_{p,\calF}^{(1^\star)} \rightarrow \K$ be a linear form such that,
for every $M \in \calW_{p,\calF}^{(1^\star)}$ with a non-zero eigenvalue $\lambda$ in $\K$,
one has $\varphi(M)\in \bigl\{0,\lambda\}$. Then, $\varphi=0$.
\end{lemme}

Proving this requires an important result on the exceptional $\overline{1}$-spec subspaces of $\Mat_3(\K)$:

\begin{lemme}\label{spansingular}
Let $\calF$ be an exceptional $\overline{1}$-spec subspace of $\Mat_3(\K)$. Then, $\calF$
is spanned by its singular elements.
\end{lemme}

\begin{proof}[Proof of Lemma \ref{spansingular}]
Denote by $\calG$ the subspace of $\calF$ spanned by its singular elements.
For $\delta \in \K$, set
$$A:=\begin{bmatrix}
0 & 1 & 0 \\
0 & 0 & 0 \\
1 & 0 & 0
\end{bmatrix}, \; B_\delta:=\begin{bmatrix}
0 & 0 & 0 \\
0 & 0 & 1 \\
\delta & 0 & 0
\end{bmatrix}, \; C_\delta:=\begin{bmatrix}
1 & 0 & 1 \\
-1 & 0 & 0 \\
-1 & -\delta & -1
\end{bmatrix}
\; \text{and} \;
D_\delta:=\begin{bmatrix}
0 & 0 & 1 \\
-1 & 0 & 0 \\
0 & -\delta & 0
\end{bmatrix}.$$
By Theorem \ref{dePazzisexcept3theo}, we lose no generality in assuming that there exists $\delta \in \K$ such that either
$\calF=\calF_\delta=\Vect(I_3,A,B_\delta,C_\delta)$ or $\calF=\calG_\delta=\Vect(I_3,A,B_\delta,D_\delta)$.
Notice already that $A$ and $B_\delta$ are singular.
\begin{itemize}
\item Assume first that $\calF=\calF_\delta$.
For all $a \in \K$, one checks that $\det(aA+C_\delta)=\delta-a$.
Taking $a=\delta $, we find that $\delta A+C_\delta$ is singular and hence $C_\delta \in \calG$.
Taking $a=\delta-1$, we find that $\det(aA+C_\delta)=1$, and as $aA+C_\delta$ has a sole eigenvalue in $\Kbar$,
this eigenvalue is $1$, so that $I_3-(aA+C_\delta)$ is singular. It follows that $I_3 \in \calG$.
As all the vectors $I_3,A,B_\delta,C_\delta$ belong to $\calG$, we conclude that $\calG=\calF$.
\item Assume now that $\calF=\calG_\delta$.
One computes that, for all $(a,b)\in \K^2$,
$$\det(a A+b B_\delta+\,D_\delta)=(1+ab^2)\,\delta+a^2\,b.$$
In particular $-A+B_\delta+D_\delta$ has determinant $1$, and its sole eigenvalue must therefore be $1$;
similarly, one proves that the sole eigenvalue of $-A-B_\delta+D_\delta$ is $-1$.
It follows that $\calG$ contains $-A+B_\delta+D_\delta-I_3$ and  $-A-B_\delta+D_\delta+I_3$.
Therefore, $\calG$ contains $D_\delta-I_3$ and $D_\delta+I_3$, and hence it contains their sum and their difference.
Thus $\calG$ contains $D_\delta$ and $I_3$, and we conclude like in the first case.
\end{itemize}
\end{proof}

\begin{proof}[Proof of Lemma \ref{linearformlemma2}]
The singular elements of $\calF$ are all nilpotent as $\calF$ is a $\overline{1}$-spec subspace of $\Mat_3(\K)$.
Thus, one deduces from Lemma \ref{spansingular} that $\calW_{p,\calF}^{(1^\star)}$ is spanned by its nilpotent elements.
Let $M \in \calW_{p,\calF}^{(1^\star)}$ be a non-zero nilpotent element, and
set $D:=\calD_{\lcro p+1,p+3\rcro}$, which belongs to $\calW_{p,\calF}^{(1^\star)}$.
One notes that $1$ is an eigenvalue of $D+\alpha M$ for every $\alpha \in \K$.
Therefore, the restriction of $\varphi$ to the affine line $D+\K M$ can only take the two values $0$ and $1$.
It follows that this affine form is not onto, and hence it is constant. One deduce that $\varphi(M)=0$.
As $\varphi$ is linear and vanishes at every nilpotent matrix of $\calW_{p,\calF}^{(1^\star)}$, we conclude that $\varphi=0$.
\end{proof}

\begin{Rem}
Lemma \ref{linearformlemma2} yields a new proof that a $\overline{1}^\star$-spec subspace of type (II)
is never similar to a $\overline{1}^\star$-spec subspace of type (I) (with the terminology from Theorem \ref{car3theo1}).
\end{Rem}

Now, we recall the following result from \cite{dSPsoleeigenvalue} (it is not formally stated there as a lemma,
but it is proved in the course\footnote{Note that our notion of a $\calV$-good
vector coincides with that of \cite{dSPsoleeigenvalue} for $\overline{1}$-spec subspaces of $\Mat_3(\K)$: indeed,
no such subspace may contain a matrix with rank $1$ and non-zero trace, as such a matrix would
have two different eigenvalues, namely its trace and $0$.} of Section 4.2 of \cite{dSPsoleeigenvalue}):

\begin{lemme}\label{goodcompletionlemma}
Let $E$ be a $3$-dimensional vector space over $\K$, and $\calV$ be an exceptional $\overline{1}$-spec subspace of
$\calL(E)$. Let $x$ be a $\calV$-good vector. Then, $x$ may be completed in a basis $(f_1,f_2,x)$ such that
$\Mat_{(f_1,f_2,x)}(\calV)$ is semi-reduced\footnote{In \cite{dSPsoleeigenvalue} (below Definition 4.13), we wrongfully asserted
that $x$ could be completed in a basis $(f_1,f_2,x)$ such that $\Mat_{(f_1,f_2,x)}(\calV)$ be \emph{fully-reduced}. This fails
for some vectors $x$ (even if one were to replace ``fully-reduced" with ``well-reduced") due to the fact that $\calV$ may be represented by $\calF_0$ or
$\calG_0$. The rest of Section 4 of \cite{dSPsoleeigenvalue} remains valid, nevertheless.}.
\end{lemme}

Here are more precise statements, the first of which is proved in Section 4.2 of \cite{dSPsoleeigenvalue}.
In both lemmas, we denote by $(e_1,e_2,e_3)$ the canonical basis of $\K^3$.

\begin{lemme}\label{goodcompletionlemma2}
Let $\calF$ be an exceptional $\overline{1}$-spec subspace of $\Mat_3(\K)$. Assume that $e_3$ is $\calV$-good and
that $\calV$ contains a matrix of the form
$\begin{bmatrix}
0 & 1 & 0 \\
0 & 0 & 0 \\
? & ? & 0
\end{bmatrix}$. Then, there exists a pair $(a,b)\in \K^2$ such that, for
$P:=\begin{bmatrix}
1 & 0 & 0 \\
0 & 1 & 0 \\
a & b & 1
\end{bmatrix}$, the space $P^{-1} \calV P$ is semi-reduced. \\
If in addition $\calV$ contains a rank $2$ matrix of the form $\begin{bmatrix}
0 & 1 & 0 \\
0 & 0 & 0 \\
? & ? & 0
\end{bmatrix}$, then the space $P^{-1} \calV P$ is well-reduced.
\end{lemme}

\begin{lemme}\label{changeofgoodvectorlemma1}
Let $\calV$ be a fully-reduced exceptional $\overline{1}$-spec subspace of $\Mat_3(\K)$.
Let $\alpha \in \K \setminus \{0\}$. Then:
\begin{enumerate}[(i)]
\item $e_2+\alpha e_3$ is $\calV$-good;
\item For every $M=(m_{i,j}) \in \calV \setminus \{0\}$
such that $M(e_2+\alpha e_3)=0$, one has $m_{1,3} \neq 0$;
\item One may choose $\alpha$ such that there exists a matrix $P \in \GL_3(\K)$ with third column $\begin{bmatrix}
0 & 1 & \alpha \\
\end{bmatrix}^T$ and for which $P^{-1} \calV P$ is well-reduced.
\end{enumerate}
\end{lemme}

\begin{proof}

Assume first that $\calV=\calF_\delta$ for some $\delta \in \K$.
Let $$M=\begin{bmatrix}
x+z & y & z \\
-z & x & t \\
y+\delta t-z & -\delta z & x-z
\end{bmatrix} \in \calV.$$
\begin{enumerate}[(i)]
\item Assume first that $\im M \subset \K(e_2+\alpha e_3)$.
The first row of $M$ must be zero, whence $y=z=0$.
The second and third rows $L_2(M)$ and $L_3(M)$ must satisfy $L_3(M)=\alpha\,L_2(M)$, and hence
$t=\frac{x-z}{\alpha}$ and $x=-\frac{ \delta z}{\alpha}=0$, which yields $t=0$, and hence $M=0$.

\item Assume now that $M \neq 0$ and $M(e_2+\alpha e_3)=0$.
Assume furthermore that $z=0$. The last two columns of $M$ must satisfy $C_2(M)=-\alpha C_3(M)$, and hence $y=-\alpha z=0$,
$x=-\alpha t$ and $-\delta z=\alpha(z-x)$, which yields $x=y=z=t=0$.

Thus, points (i) and (ii) are proved.

\item Using Lemma \ref{goodcompletionlemma2}, one sees that proving point (iii) amounts to showing that, with a good choice of $\alpha$,
some matrix $M \in \calV \setminus \{0\}$ such that $M(e_2+\alpha e_3)=0$ has rank greater than $1$: indeed,
with such a matrix $M$, we would find a matrix $P \in \GL_3(\K)$ with third column $\begin{bmatrix}
0 & 1 & \alpha \\
\end{bmatrix}^T$ and for which $P^{-1}MP=\begin{bmatrix}
N & [0]_{2 \times 1} \\
L & 0
\end{bmatrix}$ for some $N \in \Mat_2(\K)$ and some $L \in \Mat_{2,1}(\K)$, and hence $N$ is nilpotent
and non-zero; then, some $Q \in \GL_2(\K)$ satisfies $Q^{-1}NQ=\begin{bmatrix}
0 & 1 \\
0 & 0
\end{bmatrix}$ and hence, with $P':=Q \oplus 1$, the space $(PP')^{-1} \calV (PP')$ contains a rank $2$ matrix of the form $\begin{bmatrix}
0 & 1 & 0 \\
0 & 0 & 0 \\
? & ? & 0
\end{bmatrix}$; by Lemma \ref{goodcompletionlemma2}, this yields a matrix of the form $P''=\begin{bmatrix}
1 & 0 & 0 \\
0 & 1 & 0 \\
? & ? & 1
\end{bmatrix}$ such that $(PP'P'')^{-1} \calV (PP'P'')$ is semi-reduced, and one notes that
$PP'P''$ has third column $\begin{bmatrix}
0 & 1 & \alpha \\
\end{bmatrix}^T$.

As $M \mapsto M(e_2+\alpha e_3)$ maps linearly $\calV$ into a space of dimension $3$,
we see that there exists a non-zero $M \in \calV$ with $M(e_2+\alpha e_3)=0$. As earlier, one may write
$$M=\begin{bmatrix}
x+z & y & z \\
-z & x & t \\
y+\delta t-z & -\delta z & x-z
\end{bmatrix}.$$
With $M(e_2+\alpha e_3)=0$, one finds:
$$y=-\alpha z, \; x=\frac{\alpha+\delta}{\alpha} z, \; t=-\frac{\alpha+\delta}{\alpha^2}z,$$
which yields
$$M=z\,\begin{bmatrix}
\frac{\delta-\alpha}{\alpha} & -\alpha & ? \\
-1 & \frac{\delta+\alpha}{\alpha} & ? \\
? & ? & ?
\end{bmatrix}.$$
If $\rk M \leq 1$, then we deduce that $\begin{vmatrix}
\frac{\delta-\alpha}{\alpha} & -\alpha  \\
-1 & \frac{\delta+\alpha}{\alpha}
\end{vmatrix}=0$, i.e.\
\begin{equation}\label{absurdequation}
\alpha^3+\alpha^2=\delta^2.
\end{equation}
One notes that $\alpha$ may be chosen in $\{1,-1\}$ and such that \eqref{absurdequation} fails,
as $\delta^2$ cannot equal both $-1$ and $0$. This yields point (iii).
\end{enumerate}

Now, assume that $\calV=\calG_\delta$ for some $\delta \in \K$.
Let $$M=\begin{bmatrix}
x & y & z \\
-z & x & t \\
y+\delta t-z & -\delta z & x
\end{bmatrix} \in \calV.$$

\begin{enumerate}[(i)]
\item Assume that $\im M \subset \K(e_2+\alpha e_3)$.
Then, the first row of $M$ must be zero, which yields $x=y=z=0$. As the third row must be the product of the second one with
$\alpha$, one deduces that $t=\frac{x}{\alpha}=0$, and hence $M=0$.

\item Now, assume that $M(e_2+\alpha e_3)=0$ and  $M \neq 0$.
Assume furthermore that $z=0$.
As the second column is the product of the third one with $-\alpha$,
one finds $y=-\alpha z=0$, $x=-\alpha t$ and $-\delta z=-\alpha x$, which yields $x=0$ and $t=0$.
Therefore, $M=0$, which contradicts our assumptions. Thus, $z \neq 0$.

This proves points (i) and (ii).

\item Finally, let $M \in \calV$ be such that $M(e_2+\alpha e_3)=0$.
Using the relationship between the last two columns of $M$, one finds some $z \in \K$ for which
$$M=z\,\begin{bmatrix}
\frac{\delta}{\alpha} & -\alpha & ? \\
-1 & \frac{\delta}{\alpha} & ? \\
? & ? & ?
\end{bmatrix}.$$
As in the first case above, if $\rk M=1$, then one finds $\alpha^3=\delta^2$ by writing that the $2 \times 2$ upper-left determinant is zero.
However, $\delta^2$ cannot equal both $1$ and $-1$. This yields point (iii).
\end{enumerate}
\end{proof}

In the same spirit, we have the following result:

\begin{lemme}\label{changeofgoodvectorlemma2}
Let $\calV$ be a fully-reduced exceptional $\overline{1}$-spec subspace of $\Mat_3(\K)$.
Then, $e_1$ is $\calV$-good.
\end{lemme}

\begin{proof}
The proof is similar to the one of point (i) of Lemma \ref{changeofgoodvectorlemma1}.
We leave the details to the reader.
\end{proof}

We finish with a lemma on the nilpotent matrices in a space of type $\calW_{p,\calF}^{(1^\star)}$:

\begin{lemme}\label{nilpotenthyperplane}
Let $\calF$ be an exceptional $\overline{1}$-spec subspace of $\Mat_3(\K)$, and $p \in \lcro 0,n-3\rcro$.
Then, there is no linear hyperplane of  $\calW_{p,\calF}^{(1^\star)} \subset \Mat_n(\K)$ in which all the matrices are nilpotent.
\end{lemme}

\begin{proof}
Suppose, by a \emph{reductio ad absurdum}, that some linear hyperplane $H$ of $\calW_{p,\calF}^{(1^\star)}$ consists entirely of nilpotent matrices.
For $M \in \calW_{p,\calF}^{(1^\star)}$, denote by $\Delta(M)$ the $3 \times 3$ sub-matrix obtained by selecting the row and column indices
in $\{p+1,p+2,p+3\}$.
Therefore, $\Delta(H)$ is a linear subspace of codimension $0$ or $1$ in $\calF$, and it consists solely of nilpotent matrices.
As $\dim \Delta(H) \geq 3$, Gerstenhaber's theorem yields some $P \in \GL_3(\K)$ for which
$P \Delta(H) P^{-1} =\NT_3(\K)$. Then, $I_3 \in \calF \setminus \Delta(H)$, and hence
$\calF=\K I_3\oplus \Delta(H)=P\bigl(\K I_3 \oplus \NT_3(\K)\bigr)P^{-1}$, contradicting the assumption that $\calF$ be
exceptional.
\end{proof}

\subsection{Additional lemmas for fields of characteristic $3$}

The next lemma may be seen as an equivalent of Lemma \ref{lemmeALetBC} for fields of characteristic $3$.

\begin{lemme}\label{lemmeALetBCcar3}
Let $(a,b,c)\in \K^3$ and $(d_1,d_2,d_3) \in \{0,1\}^3$.
Set
$$A:=\begin{bmatrix}
0 & 1 & 0 \\
0 & 0 & 0 \\
a & 0 & 0
\end{bmatrix}, \; B:=\begin{bmatrix}
0 & 0 & 0 \\
0 & 0 & 1 \\
b & 0 & 0
\end{bmatrix} \quad \text{and} \quad D:=\begin{bmatrix}
d_1 & 0 & 0 \\
0 & d_2 & 0 \\
c & 0 & d_3
\end{bmatrix}.$$
Then, the conclusion $a=b=c=0$ follows from each one of the following assumptions:
\begin{enumerate}[(a)]
\item The triple $(d_1,d_2,d_3)$ is different from $(0,0,0)$ and $(1,1,1)$, and
$\Vect(A,B,D)$ is a $\overline{2}$-spec subspace of $\Mat_3(\K)$.
\item One has $\Sp_{\overline{\K}}(M)\subset \{0,1\}$ for all $M \in D+\Vect(A,B)$.
\end{enumerate}
\end{lemme}

\begin{proof}
We start with a preliminary computation.
Let $(x,y)\in \K^2$. Set $M:=D+xA+yB$.
The characteristic polynomial of $M$ is
$$p(t)=(t-d_1)(t-d_2)(t-d_3)-xy(ax+by+c).$$
Thus, $p(t)$ has the same differential as $(t-d_1)(t-d_2)(t-d_3)$.
\begin{itemize}
\item Assume that $d_1$, $d_2$ and $d_3$ are not equal and that every matrix of $D+\Vect(A,B)$ has at most two eigenvalues in $\Kbar$.
\begin{itemize}
\item Assume that two of the $d_i$'s equal $0$. Fixing $(x,y)\in \K^2$ and taking $M$ and $p(t)$ as above,
we note that $p'(t)=-2t$ has zero as its sole root, and therefore the only possible multiple eigenvalue of $M$
is $0$. It follows that $xy(ax+by+c)=0$.
Therefore,
$$\forall (x,y)\in (\K \setminus \{0\})^2, \; ax+by+c=0.$$
As $\K \setminus \{0\}$ has more than one element, this yields $a=b=c=0$.

\item If two of the $d_i$'s equal $1$, we note that every matrix of $(I_3-D)+\Vect(A,B)$
has at most two eigenvalues in $\Kbar$. Applying the previous situation in this new context,
we find that $a=b=-c=0$.
\end{itemize}
In particular, this proves that condition (a) implies $a=b=c=0$.

\item Assume now that every $M \in D+\Vect(A,B)$ satisfies $\Sp_{\overline{\K}}(M) \subset \{0,1\}$.
Fix $(x,y) \in \K^2$ and take $M$ and $p(t)$ as above.
\begin{itemize}
\item If the $d_i$'s are not equal, then the first part of the proof yields $a=b=c=0$.
\item Assume that $d_1=d_2=d_3=0$. Then, $p(t)=t^3-xy(ax+by+c)$.
As every eigenvalue of $M$ must belong to $\{0,1\}$, we deduce that $xy(ax+by+c) \in \{0,1\}$.
It follows that
\begin{equation}\label{idpolynom1}
\forall (x,y)\in \K^2, \; xy(ax+by+c)\bigl(xy(ax+by+c)-1\bigr)=0.
\end{equation}
Assume that $\# \K>3$. Then, $\# \K>4$ as $\K$ has characteristic $3$.
Identity \eqref{idpolynom1}, which is polynomial in $x$ and $y$, with all partial degrees less than or equal to $4$,
then yields that either $\forall (x,y)\in \K^2, \; ax+by+c=0$, or $\forall (x,y)\in \K^2, \; xy=0$, or
$\forall (x,y)\in \K^2, \; xy(ax+by+c)=1$. As the latter cases cannot hold (take $(x,y)=(1,1)$ and $(x,y)=(0,0)$, respectively),
one deduces that $a=b=c=0$. \\
Assume finally that $\# \K=3$, so that $\forall x \in \K, \; x^3=x$.
Then, developing identity \eqref{idpolynom1} yields:
$$\forall (x,y)\in \K^2, \;
c^2 x^2y^2+(a^2-ac-b)xy^2+(b^2-bc-a)x^2y-(ab+c) xy=0.$$
As $\# \K>2$, this leads to $c^2=0$, $a^2-ac-b=0$ and $ab+c=0$.
Thus, $a^2=b$ and $a^3=0$, which yields $a=b=c=0$, as claimed.
\item If $d_1=d_2=d_3=1$, then we note that, replacing $D$ with $D'=I_3-D$, every $M \in D'+\Vect(A,B)$
satisfies $\Sp_{\Kbar}(M) \subset \{0,1\}$, and hence $a=b=-c=0$.
\end{itemize}
We conclude that condition (b) implies $a=b=c=0$.
\end{itemize}
\end{proof}

We finish with the equivalent of Lemma \ref{4by4carnot3lemma} for fields of characteristic $3$:

\begin{lemme}\label{degree4car3}
Let $(a,b,c)\in \K^3$. Assume that the polynomial
$p(t)=t^4+at^2+bt+c$ has at most two roots in $\Kbar$.
Then, there exists $u \in \K$ such that either
$p(t)=t^4+u\,t$ or $p(t)=t^4-2u t^2+u^2$.
If in addition $p(t)$ has at most one non-zero root in $\Kbar$, then
$p(t)=t^4+v\,t$ for some $v \in \K$.
\end{lemme}

\begin{proof}
Let $x$ be a multiple root of $p(t)$ in $\Kbar$.
\begin{itemize}
\item Assume that $x$ has multiplicity greater than $2$. Then, the roots of $p(t)$ are $x,x,x,y$ for some $y \in \Kbar$,
and $y=0$ as the coefficient of $p(t)$ alongside $t^3$ is zero. Therefore, $p(t)=(t-x)^3\,t=t^4-x^3t$.
\item Assume that $x$ has multiplicity $2$. Then, the roots of $p(t)$ are $x,x,y,y$ for some $y \in \Kbar$.
Using again the coefficient of $t^3$, one finds $y=-x$, and therefore $p(t)=(t^2-x^2)^2=t^4-2x^2t^2+x^4$.
\end{itemize}
Finally, we note that if $p(t)=t^4-2u t^2+u^2$ for some non-zero $u \in \K$, then
$p(t)$ has two non-zero roots in $\Kbar$ (namely, the square roots of $u$).
Thus, the second statement follows from the first one.
\end{proof}

\section{Large spaces of matrices with at most one non-zero eigenvalue (III): characteristic $3$}\label{equalitysection4}

Throughout the section, we assume that $\K$ has characteristic $3$ and we aim at proving Theorem \ref{car3theo1}.
The strategy is globally similar to that of Section \ref{equalitysection2}, but the larger variety of solutions
adds a lot to the complexity of the induction process, and the weaker result in Lemma \ref{lemmeALetBC}
also requires different arguments in the proof (in particular, Lemma \ref{lemmeALetBCcar3} will be involved).

\subsection{Setting things up, and diagonal-compatibility}\label{setupcar31star}

Our aim is to prove Theorem \ref{car3theo1} by induction on $n$. The case $n=2$ has been dealt with in
Section \ref{n=2section}. Let $n \geq 3$, and assume that the result of Theorem \ref{car3theo1}
holds for the integer $n-1$. Let $\calV$ be a $\overline{1}^\star$-spec subspace of $\Mat_n(\K)$
with dimension $\dbinom{n}{2}+1$. If either $\calV$ or $\calV^T$ has property (R) (see Definition \ref{propertyR}),
then Theorem \ref{specialtheo} yields that either $\calV \simeq \calV_{\{n\}}^{(1^\star)}$ or $\calV \simeq \calV_{\{1\}}^{(1^\star)}$.
In those cases, we are done. Therefore, in the rest of the proof, we assume that:
\begin{center}
Neither $\calV$ nor $\calV^T$ has property (R).
\end{center}
In other words, if a non-zero vector $x \in \K^n$ is $\calV$-good, then no matrix of $\calV$ has $\K x$
as its column space, and if it is $\calV^T$-good, then no matrix of $\calV$ has $\K x^T$
as its row space.

Denote by $(e_1,\dots,e_n)$ the canonical basis of $\K^n$.
By Proposition \ref{goodprop}, we lose no generality in assuming that $e_n$ is $\calV$-good.
From there, we define the spaces $\calW$ and $\calV_\ul$, together with the maps $K$, $C$ and $a$,
as in Section \ref{setup2}.

\begin{Rem}[On the non-vanishing of $a$]\label{nonvanisharemark}
The linear form $a$ on $\calW$ is non-zero if and only if $\calV$ contains a matrix $M$ such that $Me_n=e_n$.
Therefore, for any $P \in \GL_n(\K)$ satisfying $Pe_n=e_n$, the linear form corresponding to $a$ for the
similar subspace $P\calV P^{-1}$ is non-zero if and only if $a$ is non-zero.
\end{Rem}

With the same line of reasoning as in Section \ref{setup2}, one finds that
$\calV_\ul$ is a $\overline{1}^\star$-spec subspace of $\Mat_{n-1}(\K)$ with dimension $\dbinom{n-1}{2}+1$.
Applying the induction hypothesis yields:
\begin{itemize}
\item[(A)] The $\overline{1}^\star$-spec subspace $\calV_\ul$ of $\Mat_{n-1}(\K)$ has type (I) or (II).
\end{itemize}
Replacing $\calV$ with a well-chosen similar subspace, as in Section \ref{setup2}, we lose no generality in further assuming that:
\begin{itemize}
\item[(A')] Either $\calV_\ul=\calV_I^{(1^\star)}$ for some non-empty subset $I$ of $\lcro 1,n-1\rcro$,
or $\calV_\ul=\calW_{p,\calF}^{(1^\star)}$ for some $p \in \lcro 0,n-4\rcro$ and some
exceptional $\overline{1}$-spec subspace $\calF$ of $\Mat_3(\K)$. Moreover, if $p=0$, one may assume that $\calF$ is
semi-reduced (and one may even assume that $\calF$ is fully-reduced).
\end{itemize}
In the second case, one sets $I:=\lcro p+1,p+3\rcro$ and, in any case, one notes that $\calD_I \in \calV_\ul$.

As $e_n$ is $\calV$-good and $\calV$ does not have property (R), we also find, as in Section \ref{setup2},
a linear form $\alpha : \calV_\ul \rightarrow \K$ such that
$$\forall M \in \calW, \; a(M)=\alpha(K(M)).$$

Using Remark \ref{fullyreducedremark}, one notes that the first vector of the canonical basis of $\K^{n-1}$
is $\calV_\ul^T$-good. As $e_n$ is $\calV$-good, one deduces, with the line of reasoning from the proof of Claim \ref{e1VTgoodclaim1}, that
\begin{center}
$e_1$ is $\calV^T$-good.
\end{center}
Then, as $\calV^T$ does not have property (R), one may define the maps
$K'$, $R$ and $b$ together with the spaces $\calW'$ and $\calV_\lr$ as in Section \ref{setup2},
and one finds that $\calV_\lr$ is a $\overline{1}^\star$-spec subspace of $\Mat_{n-1}(\K)$ with dimension $\dbinom{n-1}{2}+1$.
Therefore:
\begin{itemize}
\item[(B)] The $\overline{1}^\star$-spec subspace $\calV_\lr$ of $\Mat_{n-1}(\K)$ has type (I) or (II).
\end{itemize}
As $\calV^T$ does not have property (R), we obtain a linear form
$\beta : \calV_\lr \rightarrow \K$ such that
$$\forall M \in \calW', \; b(M)=\beta(K'(M)).$$
By Lemmas \ref{linearformlemma} and \ref{linearformlemma2}, $\beta$
vanishes at every nilpotent matrix of $\calV_\lr$.
With the same proof as Claim \ref{fn-1goodclaim}, one uses the fact that $e_n$ is  $\calV$-good to obtain:
\begin{center}
The last vector of the canonical basis of $\K^{n-1}$ is $\calV_\lr$-good.
\end{center}

Now, we need to understand when $b$ is the zero linear form. It turns out that this solely depends on the type of the $\calV_\ul$ space!

\begin{claim}
The linear form $b$ is non-zero if and only if $1 \in I$.
\end{claim}

\begin{proof}
The converse implication is obvious as $\calV_\ul$ contains $\calD_I$.
For the direct implication, note first that Lemmas \ref{linearformlemma} and \ref{linearformlemma2}
yield that if there exists some $N \in \calV_\lr$ with a non-zero eigenvalue and $\beta(N)=0$, then
$\beta$ vanishes everywhere on $\calV_\lr$ and therefore $b=0$. By property (A'), it suffices to examine the case where $1 \not\in I$.
Then, $\calV$ contains a matrix of the form
$$M=\begin{bmatrix}
0 & [0] & 0 \\
[0] & \calD_{I-1} & [0] \\
? & [?] & ?
\end{bmatrix}, \quad \text{with $I-1 \neq \emptyset$.}$$
It follows that $K'(M)$ is a matrix of $\calV_\lr$ for which $1$ is an eigenvalue and $\beta(K'(M))=b(M)=0$.
Therefore, $b=0$.
\end{proof}

In order to simplify the discussion, it is useful to limit the number of cases.
We eliminate one tedious special case by using a \emph{``transpose and conjugate"} argument.
Assume that $b=0$. Denote by $K_n$ the permutation matrix of $\GL_n(\K)$
associated with $\sigma_n : i \mapsto n+1-i$, and set $\calV':=K_n \calV^T K_n^{-1}$, which is a
$\overline{1}^\star$-spec subspace of $\Mat_n(\K)$ with dimension $\dbinom{n}{2}+1$.
Notice that neither $\calV'$ nor $(\calV')^T$ has property (R),
and that $e_n$ is $\calV'$-good since $e_1$ is $\calV^T$-good. From there, we could use the same reduction of $\calV'_\ul$
to arrive at the current point of discussion for $\calV'$. However, in that case, Remark \ref{nonvanisharemark}
would show that the linear form $a'$ attached to $\calV'$ (as $a$ was attached to $\calV$) is zero
(because $b$ is zero). We also note that if the conclusion of Theorem \ref{car3theo1} holds for $\calV'$, then
it holds for $\calV$. Indeed, for every non-empty subset $I$ of $\lcro 1,n\rcro$, one has
$\bigl(\calV_I^{(1^\star)}\bigr)^T \simeq \calV_{\sigma_n(I)}^{(1^\star)}$, while, for every $p \in \lcro 0,n-3\rcro$
and every exceptional $\overline{1}$-spec subspace $\calF$ of $\Mat_3(\K)$, one has
$\bigl(\calW_{p,\calF}^{(1^\star)}\bigr)^T \simeq \calW_{n-3-p,\calF^T}^{(1^\star)}$ and $\calF^T$
is an exceptional $\overline{1}$-spec subspace of $\Mat_3(\K)$.
We conclude that no generality is lost in assuming:
\begin{itemize}
\item[(C)] If $a \neq 0$ then $b \neq 0$.
\end{itemize}
This immediately helps us discard one tedious special case for $\calV_\lr$:

\begin{claim}\label{lowrightdiscard}
There is no exceptional $\overline{1}$-spec subspace $\calG$ of $\Mat_3(\K)$ for which
$\calV_\lr \simeq \calW_{n-4,\calG}^{(1^\star)}$.
\end{claim}

\begin{proof}
Assume on the contrary that such a subspace $\calG$ exists.
There are two striking properties of $\calE:=\calW_{n-4,\calG}^{(1^\star)} \subset \Mat_{n-1}(\K)$, where we denote
by $(f_1,\dots,f_{n-1})$ the canonical basis of $\K^{n-1}$:
\begin{itemize}
\item No $\calE$-good vector belongs to $\Vect(f_1,\dots,f_{n-4})$.
\item For every vector $x \in \K^{n-1} \setminus \Vect(f_1,\dots,f_{n-4})$, there is some $N \in \calE$ with $Nx=x$.
This follows from the fact that, for every $C \in \Mat_{n-4,3}(\K)$, the space $\calE$ contains the matrix
$\begin{bmatrix}
[0]_{(n-4) \times (n-4)} & C \\
[0]_{3 \times (n-4)} & I_3
\end{bmatrix}$, as $I_3 \in \calG$.
\end{itemize}
Remembering that $f_{n-1}$ is $\calV_\lr$-good, we deduce from those two facts that $\calV_\lr$
must contains a matrix $N$ with $Nf_{n-1}=f_{n-1}$. This yields $a \neq 0$ (see Remark \ref{nonvanisharemark}).
Yet, since $\calV_\lr \simeq \calE$, Lemma \ref{linearformlemma2} yields $b=0$, contradicting property (C).
\end{proof}

In the rest of the proof, we shall write every matrix of $\calV$ with the block decomposition specified in Remark \ref{blockmatremark}.

Now, we are ready to reduce $\calV$ one step further and to recognize the possible similarity classes of $\calV_\lr$
with respect to that of $\calV_\ul$. In this prospect, we write every matrix of $\calW \cap \calW'$ as
$$M=\begin{bmatrix}
b(M) & 0 & 0 \\
? & H(M) & 0 \\
?  & ? & a(M)
\end{bmatrix}$$
and we set
$$\calV_\md:=H(\calW \cap \calW') \subset \Mat_{n-2}(\K),$$
where the subscript ``m" stands for ``middle".
Note that the space $\calV_\md$ depends only on $\calV_\ul$.
From the definition of a semi-reduced exceptional $\overline{1}$-spec subspace of $\Mat_3(\K)$, one finds
the possible shapes of $\calV_\md$, depending on $\calV_\ul$:

\begin{center}
\begin{tabular}{| c | c |}
\hline
If $\calV_\ul=\cdots$ & then $\calV_\md=\cdots$ \\
\hline
\hline
$\calV_{\{1\}}^{(1^\star)}$ & $\NT_{n-2}(\K)$ \\
\hline
$\calV_I^{(1^\star)}$, where $I \subset \lcro 1,n-1\rcro$, $I \neq \emptyset$ and $I \neq \{1\}$ & $\calV_{(I \setminus \{1\})-1}^{(1^\star)}$ \\
\hline
$\calW_{0,\calF}^{(1^\star)}$, for a semi-reduced  & $\calV_{\{1,2\}}^{(1^\star)}$ \\
exceptional $\overline{1}$-spec subspace $\calF$ of $\Mat_3(\K)$ & \\
\hline
$\calW_{p,\calF}^{(1^\star)}$, where $p \in \lcro 1,n-4\rcro$ and $\calF$ is an  & $\calW_{p-1,\calF}^{(1^\star)}$ \\
exceptional $\overline{1}$-spec subspace of $\Mat_3(\K)$ & \\
\hline
\end{tabular}
\end{center}

Finally, we set
$$\widetilde{I}:=\begin{cases}
I & \text{if $a=0$} \\
I \cup \{n\} & \text{otherwise,}
\end{cases}$$
so that $\calV$ contains a matrix of the form $\calD_{\widetilde{I}}+\begin{bmatrix}
[0]_{(n-1) \times (n-1)} & [0]_{(n-1) \times 1} \\
[?]_{1 \times (n-1)} & 0
\end{bmatrix}$.
As in Section \ref{setup2}, one uses the fact that $\beta$ vanishes at every nilpotent matrix of $\calV_\lr$ - a consequence of Lemmas
\ref{linearformlemma} and \ref{linearformlemma2} - to
obtain that $J:=\bigl(\widetilde{I} \setminus \{1\}\bigr)-1$ is non-empty, and one finds that
$\calV_\lr$ contains a matrix of the form
$$\calD_J +\begin{bmatrix}
[0]_{(n-2) \times (n-2)} & [0]_{(n-2) \times 1} \\
[?]_{1 \times (n-2)} & 0
\end{bmatrix}.$$

\begin{claim}[Compatibility claim]\label{cornerclaim1car3}
There exists a matrix of the form
$Q=\begin{bmatrix}
I_{n-2} & [0]_{(n-2) \times 1} \\
[?]_{1 \times (n-2)} & 1
\end{bmatrix} \in \GL_{n-1}(\K)$ such that the following implications hold:
\begin{center}
\begin{tabular}{| c | c |}
\hline
If $\calV_\ul=\cdots$ & then $\calV_\lr=Q\,\calE\,Q^{-1}$, where $\calE=\cdots$ \\
\hline
\hline
$\calV_I^{(1^\star)}$ for some non-empty $I \subset \lcro 1,n\rcro$ & $\calV_J^{(1^\star)}$, where $J=(\widetilde{I}\setminus \{1\})-1$ \\
\hline
$\calW_{0,\calF}^{(1^\star)}$ for some semi-reduced exceptional & $\calV_{\{1,2\}}^{(1^\star)}$ \\
$\overline{1}$-spec subspace $\calF$ of $\Mat_3(\K)$ &   \\
\hline
$\calW_{p,\calF}^{(1^\star)}$ for some exceptional &  \\
$\overline{1}$-spec subspace $\calF$ of $\Mat_3(\K)$ & $\calW_{p-1,\calF}^{(1^\star)}$ \\
and some $p \in \lcro 1,n-4\rcro$ & \\
\hline
\end{tabular}
\end{center}
\end{claim}

\begin{proof}
If there exists a non-empty subset $J'$ of $\lcro 1,n-1\rcro$ such that $\calV_\lr \simeq \calV_{J'}^{(1^\star)}$,
then we set $\calA:=\calV_{J'}^{(1^\star)}$. Otherwise, there exists a (unique) $q \in \lcro 0,n-4\rcro$ and an
exceptional $\overline{1}$-spec subspace $\calG$ of $\Mat_3(\K)$ such that $\calV_\lr \simeq \calW_{q,\calG}^{(1^\star)}$,
and in fact $q<n-4$ (see Claim \ref{lowrightdiscard}); in that case, we set $\calA:=\calW_{q,\calG}^{(1^\star)}$.

In any case, we have a non-singular matrix $Q \in \GL_{n-1}(\K)$ such that $\calV_{\lr}=Q\,\calA\,Q^{-1}$.
Denote by $(f_1,\dots,f_{n-1})$ the canonical basis of $\K^{n-1}$, and recall that $f_{n-1}$ is $\calV_\lr$-good.
One sees that the $\calA$-good vectors are the ones of $\K^{n-1} \setminus \Vect(f_1,\dots,f_{n-2})$.
Therefore, $Q^{-1}f_{n-1} \not\in \Vect(f_1,\dots,f_{n-2})$, which yields
a matrix of the form $Q'=\begin{bmatrix}
I_{n-2} & [?] \\
[0] & ?
\end{bmatrix} \in \GL_{n-1}(\K)$ such that
$Q^{-1} f_{n-1} =Q' f_{n-1}$. One checks that $Q' \calA (Q')^{-1}=\calA$, whence we may now replace
$Q$ with $QQ'$, in which case we have the additional property $Qf_{n-1}=f_{n-1}$.

Now, we can write $Q=\begin{bmatrix}
R & [0]_{(n-2) \times 1} \\
L_1 & 1
\end{bmatrix}$ with $R \in \GL_{n-2}(\K)$ and $L_1 \in \Mat_{1,n-2}(\K)$.
Then,
$$Q=Q_1Q_2, \quad \text{where $Q_1=\begin{bmatrix}
I_{n-2} & [0]_{(n-2) \times 1} \\
L_1R^{-1} & 1
\end{bmatrix}$ and $Q_2:=R \oplus 1$.}$$
From there, we set $\calE:=Q_2 \calA Q_2^{-1}$ and we wish to prove that $\calE$ has the shape stated in our claim,
depending on the shape of $\calV_\ul$.

Let $U \in \calV_\md$.
Then, we know that $\calV_\lr$ contains a matrix of the form
$\begin{bmatrix}
U & [0]_{(n-2) \times 1} \\
[?]_{1 \times (n-2)} & ?
\end{bmatrix}$. It follows that $\calA$ contains a matrix of the form
$$Q^{-1} \begin{bmatrix}
U & [0]_{(n-2) \times 1} \\
[?]_{1 \times (n-2)} & ?
\end{bmatrix} Q=\begin{bmatrix}
R^{-1} U R & [0]_{(n-2) \times 1} \\
[?]_{1 \times (n-2)} & ?
\end{bmatrix}.$$

Finally, we define a new $\overline{1}^\star$-spec subspace $\calB$ of $\Mat_{n-2}(\K)$ as follows:
\begin{center}
\begin{tabular}{| c | c |}
\hline
If $\calA=\cdots$ & then $\calB:=\cdots$ \\
\hline
\hline
$\calV_{\{n-1\}}^{(1^\star)}$ & $\NT_{n-2}(\K)$ \\
\hline
$\calV_{J'}^{(1^\star)}$ for some non-empty subset $J' \subset \lcro 1,n-1\rcro$ & $\calV_{J' \setminus \{n-1\}}^{(1^\star)}$ \\
with $J' \neq \{n-1\}$ & \\
\hline
$\calW_{q,\calG}^{(1^\star)}$ for some exceptional &  \\
$\overline{1}$-spec subspace $\calG$ of $\Mat_3(\K)$ & $\calW_{q,\calG}^{(1^\star)}$  \\
and some $q \in \lcro 0,n-5\rcro$ & \\
\hline
\end{tabular}
\end{center}

In any case, one finds $R^{-1} \calV_\md R \subset \calB$.
\begin{itemize}
\item If $\calV_\md$ contains $\NT_{n-2}(\K)$, then one deduces from Lemma \ref{nilpotenthyperplane}
that $\calB$ is not a $\overline{1}^\star$-spec subspace of $\Mat_{n-2}(\K)$ of type (II).
\item If $\calV_\md$ has dimension $\dbinom{n-2}{2}+1$, then we find
$\calB=R^{-1} \calV_\md R$ as $\dim \calB \leq \dbinom{n-2}{2}+1$ in any case.
\end{itemize}

For to conclude, we distinguish between two cases:
\begin{itemize}
\item Assume that $\calV_\ul$ has type (I) or equals $\calW_{0,\calF}^{(1^\star)}$ for some semi-reduced
exceptional $\overline{1}$-spec subspace $\calF$ of $\Mat_3(\K)$. Then, $\NT_{n-2}(\K) \subset \calV_\md$ and hence
$\calB$ cannot have type (II). Therefore, $\calA$ has type (I), and hence $\calV_\lr$ has type (I). As we know that
$\calV_\lr$ contains a matrix of the form $\calD_J +\begin{bmatrix}
[0]_{(n-2) \times (n-2)} & [0]_{(n-2) \times 1} \\
[?]_{1 \times (n-2)} & 0
\end{bmatrix}$ where $J \subset \lcro 1,n-1\rcro$ is non-empty, we may follow the line of reasoning featured in the proof of Claim
\ref{corner2claim} to obtain $\calE=\calV_J^{(1^\star)}$.

\item Now, assume that $\calV_\ul=\calW_{p,\calF}^{(1^\star)}$ for some $p \in \lcro 1,n-4\rcro$ and some exceptional
$\overline{1}^\star$-spec subspace $\calF$ of $\Mat_3(\K)$. Then, $\calV_\md$ is a $\overline{1}^\star$-spec subspace of $\Mat_{n-2}(\K)$
of type (II). As $\dim \calV_\md=\dbinom{n-2}{2}+1$, one deduces that $\calB=R^{-1} \calV_\md R$, and hence $\calA$
has type (II), so that
$$\calE=Q_2(\calB \vee \{0\})Q_2^{-1}=(R\calB R^{-1}) \vee \{0\}=\calV_\md \vee \{0\}=\calW_{p-1,\calF}^{(1^\star)} \vee \{0\}.$$
\end{itemize}
In any case, this completes the proof since $\calV_\lr=Q_1 \calE Q_1^{-1}$.
\end{proof}

Fix $Q$ given by Claim \ref{cornerclaim1car3}, set $P:=1\oplus Q$ and replace $\calV$ with
$P^{-1}\calV P$. For the new space $\calV$, notice that (C) still holds and that
we have the following improvement of properties (A') and (B):

\begin{itemize}
\item[\textbf{Case 1.}] Either there exists a non-empty subset $I$ of $\lcro 1,n-1\rcro$ such that
$\calV_\ul=\calV_I^{(1^\star)}$ and $\calV_\lr=\calV_J^{(1^\star)}$, where $J=(\widetilde{I} \setminus \{1\})-1$;
\item[\textbf{Case 2.}] Or there exists an exceptional $\overline{1}$-spec subspace $\calF$ of $\Mat_3(\K)$ and an integer
$p \in \lcro 1,n-4\rcro$ such that
$\calV_\ul=\calW_{p,\calF}^{(1^\star)}$ and $\calV_\lr=\calW_{p-1,\calF}^{(1^\star)}$;
\item[\textbf{Case 3.}] Or there exists a semi-reduced exceptional $\overline{1}$-spec subspace $\calF$ of $\Mat_3(\K)$
such that
$\calV_\ul=\calW_{0,\calF}^{(1^\star)}$ and $\calV_\lr=\calV_{\{1,2\}}^{(1^\star)}$.
\end{itemize}

In Case 1, we aim at showing that $\calV \simeq \calV_{\widetilde{I}}^{(1^\star)}$,
except when $n=3$ and $\widetilde{I}=\lcro 1,3\rcro$, in which case one cannot rule out the possibility that
$\calV$ be an exceptional $\overline{1}$-spec subspace of $\Mat_3(\K)$; in Case 2,
we want to prove that $\calV \simeq \calW_{p,\calF}^{(1^\star)}$; in Case 3,
we want to establish that $\calV \simeq \calW_{0,\calF}^{(1^\star)}$.
Except for Case $1$ with $n=3$ and $\widetilde{I}=\lcro 1,3\rcro$ (Section \ref{specialcasen=3section}),
the first two cases may be dealt with singlehandedly (Section \ref{1starcase1and2car3section}).
Case 3 is far more tedious and will be tackled in Section \ref{1starcase3car3section}.

By the ``transpose and conjugate" technique which we have already used earlier in the section,
one may discard the case when $\widetilde{I}=\{1,2,n\}$ and $n \geq 4$: indeed,
it may be reduced to Case 1 with $\widetilde{I}=\{1,n-1,n\}$ (and $n \geq 4$).

We finish with general considerations that apply to all the above three cases.
Remember that we have two linear forms $\alpha : \calV_\ul \rightarrow \K$ and $\beta : \calV_\lr \rightarrow \K$
such that
$$\forall M \in \calW, \; a(M)=\alpha(K(M)) \quad \text{and} \quad \forall M \in \calW', \; b(M)=\beta(K'(M)).$$
Also, remember that $\calV_\ul$ contains the $(n-1) \times (n-1)$ matrix $\calD_I$.
Applying Lemma \ref{linearformlemma} or Lemma \ref{linearformlemma2} to $\alpha$ (whether Case 1 holds or Cases 2 or 3 hold),
and using the definition of $\widetilde{I}$ and the shape of $\calV_\lr$, we obtain:
\begin{itemize}
\item[(D)] The space $\calV$ contains $\calD_{\widetilde{I}}+d\,E_{n,1}$ for some $d \in \K$.
\end{itemize}

\subsection{Case $1$ with $n=3$ and $\widetilde{I}=\lcro 1,3\rcro$}\label{specialcasen=3section}

Here, we assume that $n=3$ and $\widetilde{I}=\lcro 1,3\rcro$, so that
$\calV_\ul=\calV_{\{1,2\}}^{(1^\star)}=\calV_\lr$.
In that case, our aim is to prove that $\calV$ is actually a $\overline{1}$-spec subspace of $\Mat_3(\K)$.
First of all, we use the respective shapes of $\calV_\ul$ and $\calV_\lr$ to obtain three special matrices of $\calV$.

By property (D), we already know that $\calV$ contains a matrix of the form
$$D=\begin{bmatrix}
1 & 0 & 0 \\
0 & 1 & 0 \\
? & 0 & 1
\end{bmatrix}.$$
On the other hand, by Lemma \ref{linearformlemma}, $\alpha$ and $\beta$ vanish everywhere on $\NT_2(\K)$.
It follows that $\calV$ contains two matrices
$$A=\begin{bmatrix}
0 & 1 & 0 \\
0 & 0 & 0 \\
? & ? & 0
\end{bmatrix} \quad \text{and} \quad B=\begin{bmatrix}
0 & 0 & 0 \\
? & 0 & 1 \\
? & 0 & 0
\end{bmatrix}.$$
Note that $A$, $B$ and $D$ are linearly independent, and that all have trace zero.

We may sum up one of the above results as follows:

\begin{prop}\label{sumuppropcasen=3}
Let $\calV'$ be a $\overline{1}^\star$-spec subspace of $\Mat_3(\K)$ with dimension $4$.
Assume that $e_3$ is $\calV'$-good, that $\calV'_\ul=\calV_{\{1,2\}}^{(1^\star)}$, and that the linear form
$a'$ attached to $\calV'$ is non-zero.
Then, there exists a matrix of the form
$Q=\begin{bmatrix}
1 & 0 & 0 \\
0 & 1 & 0 \\
0 & ? & 1
\end{bmatrix}$ such that the subspace $Q^{-1}\calV' Q$ contains a matrix of the form
$\begin{bmatrix}
0 & 0 & 0 \\
? & 0 & 1 \\
? & 0 & 0
\end{bmatrix}$.
\end{prop}

Set $P:=\begin{bmatrix}
1 & 1 & 0 \\
0 & 1 & 0 \\
0 & 0 & 1
\end{bmatrix}$ and note that $\calV':=P^{-1} \calV P$ satisfies the assumptions of Proposition \ref{sumuppropcasen=3}
(see Remark \ref{nonvanisharemark}). With $Q$ given by Proposition \ref{sumuppropcasen=3},
we deduce that $\calV$ contains a matrix of the form
$$H=PQ \begin{bmatrix}
0 & 0 & 0 \\
? & 0 & 1 \\
? & 0 & 0
\end{bmatrix}Q^{-1}P^{-1}=
\begin{bmatrix}
? & ? & 1 \\
? & ? & 1 \\
? & ? & ?
\end{bmatrix}.$$
Note that $H$, being similar to a trace zero matrix, has trace zero. On the other hand, it is obvious that $H$ is not contained in $\Vect(A,B,D)$
(judging from its entry at the $(1,3)$-spot).
As $\dim \calV=4$, we deduce that $(A,B,D,H)$ is a basis of $\calV$. This yields:

\begin{claim}\label{everytrace0}
Every matrix of $\calV$ has trace zero.
\end{claim}

From there, we obtain:

\begin{claim}
Every matrix of $\calV$ has a sole eigenvalue in $\Kbar$.
\end{claim}

\begin{proof}
Let $M \in \calV$ be with several eigenvalues in $\Kbar$. Then, $0$ must be one of them
and we denote by $x$ the other one.
If $x$ has multiplicity $p$ in the characteristic polynomial of $M$, then $\tr M=p.x \neq 0$
as $p \in \{1,2\}$ and $x \neq 0$. This contradiction proves our claim.
\end{proof}

Therefore, $\calV$ is a $4$-dimensional $\overline{1}$-spec subspace of $\Mat_3(\K)$.
By the very definition of an exceptional $\overline{1}$-spec subspace of $\Mat_3(\K)$,
we deduce that either $\calV \simeq \K I_3 \oplus \NT_3(\K) = \calV_{\{1,2,3\}}^{(1^\star)}$
or $\calV$ is an exceptional $\overline{1}$-spec subspace of $\Mat_3(\K)$, in which case it has type (II).
This completes Case $1$ when $n=3$ and $\widetilde{I}=\lcro 1,3\rcro$.

\subsection{Cases $1$ and $2$}\label{1starcase1and2car3section}

In this section, we assume that $\calV$ falls into Case $1$ or Case $2$.
Using the results of Section \ref{specialcasen=3section}, we may entirely discard the situation
when we have Case 1 with $\widetilde{I}=\{1,2,n\}$
(see the end of Section \ref{setupcar31star}).

We know that, for every $L \in \Mat_{1,n-2}(\K)$, the space
$\calV_\ul$ contains $\begin{bmatrix}
0 & L \\
[0]_{(n-2) \times 1} & [0]_{(n-2) \times (n-2)}
\end{bmatrix}$ and, for every $C \in \Mat_{n-2,1}(\K)$, the space
$\calV_\lr$ contains $\begin{bmatrix}
[0]_{(n-2) \times (n-2)} & C \\
[0]_{1 \times (n-2)} & 0
\end{bmatrix}$. Both kinds of matrices are nilpotent, therefore Lemmas \ref{linearformlemma}
and \ref{linearformlemma2} show that $\alpha$ (respectively, $\beta$) vanishes at them.
This gives rise to two endomorphisms $\varphi : \Mat_{1,n-2}(\K) \rightarrow  \Mat_{1,n-2}(\K)$ and
$\psi : \Mat_{n-2,1}(\K) \rightarrow  \Mat_{n-2,1}(\K)$ together with two linear forms
$f : \Mat_{1,n-2}(\K) \rightarrow \K$ and $g : \Mat_{n-2,1}(\K) \rightarrow \K$ such that, for all
$(L,C) \in \Mat_{1,n-2}(\K) \times \Mat_{n-2,1}(\K)$, the space $\calV$ contains the matrices
$$A_L=\begin{bmatrix}
0 & L & 0 \\
0 & 0 & 0 \\
f(L) & \varphi(L) & 0
\end{bmatrix} \quad \text{and} \quad
B_C=\begin{bmatrix}
0 & 0 & 0 \\
\psi(C) & 0 & 0 \\
g(C) & 0 & 0
\end{bmatrix}.$$
Set
$$\calV'_m:=\begin{cases}
\NT_{n-2}(\K) & \text{in Case 1} \\
\calV_\md & \text{in Case 2.}
\end{cases}$$
In any case, using the respective shapes of $\calV_\ul$ and $\calV_\lr$ together with Lemmas \ref{linearformlemma}
and \ref{linearformlemma2}, we find a linear form $h : \calV'_m \rightarrow \K$ such that,
for every $U \in \calV'_m$, the space $\calV$ contains the matrix
$$E_U=\begin{bmatrix}
0 & 0 & 0 \\
0 & U & 0 \\
h(U) & 0 & 0
\end{bmatrix}.$$
Lemma \ref{homotheticsprop} yields scalars $\lambda$ and $\mu$ such that
$$\forall (L,C) \in \Mat_{1,n-2}(\K) \times \Mat_{n-2,1}(\K), \quad \varphi(L)=\lambda\,L \quad \text{and} \quad
\psi(C)=\mu\,C.$$
From Lemma \ref{lemmeALetBC}, one deduces that $\lambda+\mu=0$.
Setting $P:=\begin{bmatrix}
1 & 0 & 0 \\
0 & I_{n-2} & 0 \\
\lambda & 0 & 1
\end{bmatrix}$ and replacing $\calV$ with $P^{-1} \calV P$, one finds that
$\calV_\ul$ and $\calV_\lr$ are left entirely unchanged (while properties (C) and (D) still hold), but now we have,
for all $(L,C) \in \Mat_{1,n-2}(\K) \times \Mat_{n-2,1}(\K)$,
$$A_L=\begin{bmatrix}
0 & L & 0 \\
0 & 0 & 0 \\
f(L) & 0 & 0
\end{bmatrix} \quad \text{and} \quad
B_C=\begin{bmatrix}
0 & 0 & 0 \\
0 & 0 & C \\
g(L) & 0 & 0
\end{bmatrix}.$$
Now, we analyze $f$ and $g$.

\begin{claim}\label{1starspeccar3basicclaim1}
The space $\calV$ contains $\calD_{\widetilde{I}}$. Moreover, $f=0$ and $g$ vanishes at the column matrix $\begin{bmatrix}
1 & 0 & \cdots & 0
\end{bmatrix}^T$.
\end{claim}

\begin{proof}
Recall that we have a scalar $d$ such that $\calV$ contains $D=\calD_{\widetilde{I}}+d\,E_{n,1}$.
Set $L_0:=\begin{bmatrix}
1 & 0 & \cdots & 0
\end{bmatrix} \in \Mat_{1,n-2}(\K)$.
To begin with, we aim at proving that $d=0$ and that
\begin{equation}\label{vanishingonfirst}
f(L_0)=g(L_0^T)=0.
\end{equation}
Denote by $d_1,\dots,d_n$ the diagonal entries of $\calD_{\widetilde{I}}$.
Let $(x,y,z)\in \K^3$. We remark that $x\,D+y\,A_{L_0}+z\,B_{L_0^T}$
stabilizes $\Vect(e_1,e_2,e_n)$, and the matrix of its induced endomorphism in the basis $(e_1,e_2,e_n)$
is
$$x\,\begin{bmatrix}
d_1 & 0 & 0 \\
0 & d_2 & 0 \\
d & 0 & d_n
\end{bmatrix}+y\,\begin{bmatrix}
0 & 1 & 0 \\
0 & 0 & 0 \\
f(L_0) & 0 & 0
\end{bmatrix}+z\,\begin{bmatrix}
0 & 0 & 0 \\
0 & 0 & 1 \\
g(L_0^T) & 0 & 0
\end{bmatrix}.$$
This matrix has at most one non-zero eigenvalue in $\Kbar$.
If $d_1$, $d_2$ and $d_n$ are unequal, then Lemma \ref{lemmeALetBCcar3} yields $d=0$ together with equality \eqref{vanishingonfirst}.

Now, assume that $d_1=d_2=d_n$.
Fix $(x,y)\in \K^2$. Note that $e_3,\dots,e_{n-1}$ are all eigenvectors of $D+x\,A_{L_0}+y\,B_{L_0^T}$, with respective
eigenvalues $d_3,\dots,d_{n-1}$. Yet, $d_1=d_2=d_n$ and $\widetilde{I}$ equals neither $\emptyset$ nor $\{1,2,n\}$.
Therefore, one of the scalars $d_3,\dots,d_{n-1}$ equals $1$,
and hence the only possible eigenvalues of
$D+x\,A_{L_0}+y\,B_{L_0^T}$ in $\Kbar$ are $0$ and $1$. Applying Lemma \ref{lemmeALetBCcar3} to the induced
endomorphisms of $D$, $A_{L_0}$ and $B_{L_0^T}$ on $\Vect(e_1,e_2,e_n)$, we obtain $d=0$ and equality \eqref{vanishingonfirst}.

It remains to show that $f$ vanishes everywhere on $\Mat_{1,n-2}(\K)$. To prove this, we use an invariance argument.
Let $U \in \GL_{n-2}(\K)$ be an arbitrary non-singular \emph{upper-triangular} matrix, and
set $Q:=1 \oplus U \oplus 1 \in \GL_n(\K)$.
Replacing $\calV$ with $\calV':=Q \calV Q^{-1}$ leaves our previous assumptions wholly unchanged. Indeed:
\begin{itemize}
\item If $\calV_\ul=\calV_I^{(1^\star)}$, then $\calV'_\ul=(1 \oplus U)\, \calV_I^{(1^\star)}\,(1 \oplus U)^{-1}=\calV_I^{(1^\star)}$, and
a similar result holds for $\calV'_\lr$.
\item If $\calV_\ul=\calW_{p,\calF}^{(1^\star)}$ for some $p \in \lcro 1,n-4\rcro$ and some exceptional $\overline{1}$-spec subspace $\calF$ of
$\Mat_3(\K)$, then $\calV'_\ul=\calW_{p,\calF'}^{(1^\star)}$ for some exceptional $\overline{1}$-spec subspace $\calF'$ of
$\Mat_3(\K)$, and a similar result holds for $\calV'_\lr$.
\item As $Qe_n=e_n$, the non-vanishing of $a$ is equivalent to the non-vanishing of the linear form $a'$ associated with $\calV'$.
The same holds for $b$ and $b'$ because $Q^Te_1=e_1$.
\end{itemize}
Finally, for every $(L,C)\in \Mat_{1,n-2}(\K) \times \Mat_{n-2,1}(\K)$, the space $\calV'$ contains the matrices
$$\begin{bmatrix}
0 & LU^{-1} & 0 \\
0 & 0 & 0 \\
f(L) & 0 & 0
\end{bmatrix} \quad \text{and} \quad
\begin{bmatrix}
0 & 0 & 0 \\
0 & 0 & UC \\
g(C) & 0 & 0
\end{bmatrix}.$$
Setting $f' : L \mapsto f(LU)$ and $g' : C \mapsto g(U^{-1}C)$,
the above line of reasoning applies to $\calV'$ and yields
$$f'(L_0)=0.$$
Taking all $U$'s of the form $\begin{bmatrix}
1 & [?]_{1 \times (n-3)} \\
[0]_{(n-3) \times 1} & I_{n-3}
\end{bmatrix}$, one deduces that $f$ vanishes at every $1 \times (n-2)$ row matrix
with first entry $1$. The linear span of those matrices being $\Mat_{1,n-2}(\K)$, we conclude that $f=0$.
\end{proof}

Before moving forward, we summarize a small portion of our recent results.

\begin{prop}\label{sumuppropcase1and2}
Let $\calV'$ be a $\overline{1}^\star$-spec subspace of $\Mat_n(\K)$ with dimension $\dbinom{n}{2}+1$.
Assume that $\calV'$ has properties (A') and (B)
and denote by $\widetilde{I'}$ the associated subset of $\lcro 1,n\rcro$.
Assume that either the linear form $a'$ associated with
$\calV'$ is zero, or $\calV'_\ul$ has type (I) and $1 \in \widetilde{I'}$.
Assume furthermore that $\widetilde{I'} \neq \{ 1,2,n\}$.
Then, there exists a matrix $Q=\begin{bmatrix}
I_{n-1} & [0]_{(n-1) \times 1} \\
[?]_{1 \times (n-1)} & 1
\end{bmatrix} \in \GL_n(\K)$ such that
$Q^{-1}\calV' Q$ contains $E_{2,n}$.
\end{prop}

\begin{claim}\label{car3case1and2containsE1n}
The space $\calV$ contains $E_{1,n}$.
\end{claim}

\begin{proof}
We use the invariance argument already featured in the proof of Claim \ref{E1ninVclaim1}.
Set $P_1:=\begin{bmatrix}
1 & 1 \\
0 & 1
\end{bmatrix} \oplus I_{n-2}$ and consider the space $\calV':=P_1^{-1} \calV P_1$.
Note that $P_1e_n=e_n$, so that $e_n$ is $\calV'$-good. Moreover, neither $\calV'$ nor $(\calV')^T$
has property (R). As $P_1$ is upper-triangular, one finds that the overall shape
of $\calV'_\ul$ is the same one as that of $\calV_\ul$. Moreover, the linear form $a'$ associated with $\calV'$
is non-zero if and only if $a$ is non-zero. Finally, if $a$ is non-zero, then the non-vanishing of
$b$ depends solely on the type of the overall shape of $\calV'_\ul$, and therefore $\calV'$ has property (C).
It follows that the set $\widetilde{I}$ attached to $\calV'$ is the same one as the one attached to $\calV$.
As $\calV$ satisfies all the assumptions of Proposition \ref{sumuppropcase1and2}, we deduce that
$\calV'$ also does.

This yields a matrix of the form $P_2=\begin{bmatrix}
I_{n-1} & [0]_{(n-1) \times 1} \\
[?]_{1 \times (n-1)} & 1
\end{bmatrix} \in \GL_n(\K)$ such that
$P_2^{-1}P_1^{-1} \calV P_1P_2$ contains $E_{2,n}$, giving rise
to a triple $(a,b,L_2)\in \K^2 \times \Mat_{1,n-2}(\K)$ such that $\calV$ contains the matrix
$$H=P_1P_2E_{2,n} (P_1P_2)^{-1}=\begin{bmatrix}
-a & -L_2 & 1 \\
-a & -L_2 & 1 \\
[0]_{(n-3) \times 1} & [0]_{(n-3) \times (n-2)} & [0]_{(n-3) \times 1} \\
-ab & -b\,L_2 & b
\end{bmatrix},$$
and the first entry of the row matrix $L_2$ is $b-a$.
The matrix $H$, being similar to $E_{2,n}$, has rank $1$ and trace $0$.
This is also the case of $E_{2,n}$ and $E_{1,2}$, which belong to $\calV$. Applying
Lemma \ref{tracelemma}, we deduce that $\tr(E_{2,n}H)=\tr(E_{1,2}H)=0$, i.e.\
$b(a-b)=0$ and $a=0$. Therefore, $a=b=0$.

By linearly combining $H$ with matrices of type $A_L$ and $E_{2,n}$, we deduce that
$\calV$ contains a matrix of the form
$$H'=\begin{bmatrix}
0 & 0 & 1 \\
0 & T & 0 \\
0 & 0 & 0
\end{bmatrix},$$
where $T$ is a strictly upper-triangular matrix.

Now, we contend that $\calV$ contains $E_{1,n}+e\,E_{n,1}$ for some $e \in \K$.
To support this, it suffices to show that $T \in \calV_\md'$, as then
$H'-E_T$ would belong to $E_{1,n}+\K\,E_{n,1}$.
If $\calV_\md'=\NT_{n-2}(\K)$, then we readily have $T \in \calV'_\md$.
If we now assume that $\calV_\md' \neq \NT_{n-2}(\K)$ (which implies $n \geq 5$), we find that $\calV_\md'$ is a $\overline{1}^\star$-spec subspace of
$\Mat_{n-2}(\K)$ with the maximal dimension $\dbinom{n-2}{2}+1$.
Let $U \in \calV_\md'$ and $x \in \K$. Since the matrix $x\,H'+E_U$ has at most one non-zero eigenvalue in $\K$
and since it stabilizes the subspace $\Vect(e_2,\dots,e_{n-1})$, one finds that
$x\,T+U$ has at most one non-zero eigenvalue in $\overline{\K}$. In other words,
$\K T+\calV_\md'$ is a $\overline{1}^\star$-spec subspace of $\Mat_{n-2}(\K)$. However, Theorem \ref{1starspecinequality} shows that
$\calV_\md'$ is maximal among such spaces, whence $T \in \calV_\md'$ as claimed.

Therefore, we have some $e \in \K$ such that $\calV$ contains $E_{1,n}+e\,E_{n,1}$,
and in fact $e=0$ as the square roots of $-e$ in $\overline{\K}$ are eigenvalues of $E_{1,n}+e\,E_{n,1}$.
Finally, $E_{1,n} \in \calV$.
\end{proof}

\begin{claim}\label{vanishgandhcar3case1and2}
One has $g=0$ and $h=0$.
\end{claim}

\begin{proof}
Let $U \in \calV'_\md$ and $(x,y) \in \K^2$. Note that $x\,E_{1,n}+y\,E_U$ belongs to $\calV$, to the effect that it has
at most one non-zero eigenvalue in $\overline{\K}$. The matrices $E_{1,n}$ and $E_U$ all stabilize $\Vect(e_1,e_n)$,
with induced endomorphisms represented in the basis $(e_1,e_n)$ by $\begin{bmatrix}
0 & 1 \\
0 & 0
\end{bmatrix}$ and $\begin{bmatrix}
0 & 0 \\
h(U) & 0
\end{bmatrix}$. Then, Lemma \ref{n=2lemma} yields $h(U)=0$.

Let $C \in \Mat_{n-2,1}(\K)$. We note that $E_{1,n}$ and $B_C$ vanish everywhere on $\Vect(e_2,\dots,e_{n-1})$ and
induce endomorphisms of the quotient space $\K^n/\Vect(e_2,\dots,e_{n-1})$ represented by $
\begin{bmatrix}
0 & 1 \\
0 & 0
\end{bmatrix}$ and $\begin{bmatrix}
0 & 0 \\
g(C) & 0
\end{bmatrix}$ in the basis $(\overline{e_1},\overline{e_n})$, respectively. Again, this yields $g(C)=0$.
\end{proof}

From there, we can conclude. We know that, for every $(L,C,U) \in \Mat_{1,n-2}(\K) \times \Mat_{n-2,1}(\K) \times \calV'_\md$,
the space $\calV$ contains the matrix
$$\begin{bmatrix}
0 & L & 0 \\
0 & U & C \\
0 & 0 & 0
\end{bmatrix}.$$
We also know that $\calV$ contains $\calD_{\widetilde{I}}$ and that it contains $E_{1,n}$. By linearly combining such matrices,
we deduce that:
\begin{itemize}
\item Either $\calV_\ul$ has type (I) and therefore $\calV_{\widetilde{I}}^{(1^\star)} \subset \calV$;
\item Or $\calV_\ul=\calW_{p,\calF}^{(1^\star)}$ for some $p \in \lcro 1,n-4\rcro$ and some exceptional $\overline{1}$-spec subspace $\calF$ of
$\Mat_3(\K)$, and therefore $\calW_{p,\calF}^{(1^\star)} \subset \calV$.
\end{itemize}
In either case, the equality of dimensions shows that $\calV$ has type (I) or type (II), which concludes our proof in Cases 1 and 2.

\subsection{Case $3$}\label{1starcase3car3section}

Here, we assume that there is a semi-reduced exceptional $\overline{1}$-spec subspace $\calF$ of $\Mat_3(\K)$
such that $\calV_\ul=\calW_{0,\calF}^{(1^\star)}$. In that case, we have $\calV_\lr=\calV_{\{1,2\}}^{(1^\star)}$,
$a=0$ and $\widetilde{I}=\lcro 1,3\rcro$. Our goal is to prove that $\calV \simeq \calW_{0,\calF}^{(1^\star)}$.
Our strategy of proof is roughly similar to that in the two previous cases, but
things are more complicated.

Using equality $a=0$, we find linear maps $\Delta_1 : \calF \rightarrow \Mat_{1,3}(\K)$ and
$\Delta_2 : \calF \rightarrow \Mat_{1,n-4}(\K)$ such that, for every $N \in \calF$, the space $\calV$
contains
$$M_N:=\begin{bmatrix}
N & [0]_{3 \times (n-4)} & [0]_{3 \times 1} \\
[0]_{(n-4) \times 3} & [0]_{(n-4) \times (n-4)} & [0]_{(n-4) \times 1} \\
\Delta_1(N) & \Delta_2(N) & 0
\end{bmatrix}.$$
Note that $\Delta_1$ and $\Delta_2$ are uniquely defined because $e_n$ is $\calV$-good.
Remember that $\calV$ contains a matrix of $\calD_{\lcro 1,3\rcro}+\K E_{n,1}$, which shows that
$\Delta_1(I_3)=\begin{bmatrix}
d & 0 & 0
\end{bmatrix}$ for some $d \in \K$.

Our first aim is to show that, by using an additional reduction, one may
actually assume that $\Delta_1=0$ and $\Delta_2=0$.
We start by analyzing the map $\Delta_2$.

\begin{claim}\label{Delta2nulclaim1}
One has $\Delta_2=0$.
\end{claim}

\begin{proof}
Let $N \in \calF$ and $C_2 \in  \Mat_{n-4,1}(\K) \setminus \{0\}$.
Using $\calV_\lr=\calV_{\{1,2\}}^{(1^\star)}$, we see that
$\calV$ contains a matrix of the form
$$B=\begin{bmatrix}
[?]_{3 \times 1} & [0]_{3 \times (n-2)} & [0]_{3 \times 1} \\
[?]_{(n-4) \times 1} & [0]_{(n-4) \times (n-2)} & C_2 \\
? & [0]_{1 \times (n-2)} & 0
\end{bmatrix}.$$
Defining $x \in \K^n$ as the last column of $B$, we note that $M_N$ and $B$ stabilize the plane $\Vect(x,e_n)$
and that the endomorphisms of $\Vect(x,e_n)$ they induce are represented in the basis $(x,e_n)$
by $\begin{bmatrix}
0 & 0 \\
\Delta_2(N)C_2 & 0
\end{bmatrix}$ and $\begin{bmatrix}
0 & 1 \\
0 & 0
\end{bmatrix}$, respectively. Using Lemma \ref{n=2lemma}, one deduces that $\Delta_2(N)C_2=0$.
Varying $N$ and $C_2$ yields $\Delta_2=0$.
\end{proof}

As $\calF$ is semi-reduced, there is a pair $(\delta,\epsilon) \in \K^2$ such that $\calF$ contains the matrices
$$A=\begin{bmatrix}
0 & 1 & 0 \\
0 & 0 & 0 \\
\delta & 0 & 0
\end{bmatrix} \quad \text{and} \quad
B=\begin{bmatrix}
0 & 0 & 0 \\
0 & 0 & 1 \\
\epsilon & 0 & 0
\end{bmatrix}.$$
As $M_B$ belongs to $\calW'$, equality $\calV_\lr=\calV_{\{1,2\}}^{(1^\star)}$ yields
$$\Delta_1(B)=\begin{bmatrix}
? & 0 & 0
\end{bmatrix}.$$

\begin{claim}\label{delta1seminulenA}
One has $\Delta_1(I_3)=\begin{bmatrix}
d & 0 & 0
\end{bmatrix}$, $\Delta_1(A)=\begin{bmatrix}
0 & d & 0
\end{bmatrix}$ and $\Delta_1(B)=0$.
\end{claim}

\begin{proof}
Judging from the shape of $\calV_\lr$ and the fact that $\beta$ vanishes at every nilpotent matrix of $\calV_\lr$,
we know that $\calV$ contains a matrix of the form
$$H=\begin{bmatrix}
0 & [0] & 0 \\
[?] & [0] & C_0 \\
? & [0] & 0
\end{bmatrix}, \quad \text{where $C_0=\begin{bmatrix}
0 & 1 & 0 & \cdots & 0 \end{bmatrix}^T$.}$$
Let us write $\Delta_1(A)=\begin{bmatrix}
a & b & c
\end{bmatrix}$ and $\Delta_1(B)=\begin{bmatrix}
e & 0 & 0
\end{bmatrix}$.

Note that all four matrices $M_A$, $M_B$, $H$ and $\calD_{\lcro 1,3\rcro}+d\,E_{n,1}$ vanish everywhere on $\Vect(e_4,\dots,e_{n-1})$.
The matrices of their induced endomorphisms in the basis $(\overline{e_1},\overline{e_2},\overline{e_3},\overline{e_n})$
of $\K^n/\Vect(e_4,\dots,e_{n-1})$ are, for some $(f,g,h) \in \K^3$,
$$\begin{bmatrix}
0 & 1 & 0 & 0 \\
0 & 0 & 0 & 0 \\
\delta & 0 & 0 & 0 \\
a & b & c & 0
\end{bmatrix}, \quad
\begin{bmatrix}
0 & 0 & 0 & 0 \\
0 & 0 & 1 & 0 \\
\epsilon & 0 & 0 & 0 \\
e & 0 & 0 & 0
\end{bmatrix} , \quad
\begin{bmatrix}
0 & 0 & 0 & 0 \\
f & 0 & 0 & 0 \\
g & 0 & 0 & 1 \\
h & 0 & 0 & 0
\end{bmatrix}\quad \text{and} \quad
\begin{bmatrix}
1 & 0 & 0 & 0 \\
0 & 1 & 0 & 0 \\
0 & 0 & 1 & 0 \\
d & 0 & 0 & 0
\end{bmatrix},$$
respectively.
For every $(x,y) \in \K^2$, it follows that the matrix
$$\begin{bmatrix}
0 & x & 0 & 0 \\
f & 0 & y & 0 \\
\delta x+\epsilon y+g & 0 & 0 & 1 \\
ax+ey+h & bx & cx & 0
\end{bmatrix}$$
has at most one non-zero eigenvalue in $\Kbar$; therefore, its characteristic polynomial
$$p(t)=t^4-x(f+c)t^2-xy(\delta x+\epsilon y+g+b)\,t+x(cfx-y(ax+ey+h))$$
has at most one non-zero root in $\Kbar$.
In the special case when $y=0$ and $x=1$, this polynomial factorizes as $(t^2-f)(t^2-c)$, which has
at least three roots in $\Kbar$ if $f \neq c$. It follows that $f=c$.

Next, Lemma \ref{degree4car3} yields that $2xf=0$ and
$x(cfx-y(ax+ey+h))=0$ for all $(x,y)\in \K^2$. The first identity yields $f=c=0$.
Then, the second one yields $ax+ey+h=0$ for all $(x,y) \in (\K \setminus \{0\})^2$.
As $\K \setminus \{0\}$ has more than two elements, one deduces that $a=e=h=0$.
This yields the claimed results on $\Delta_1(B)$.

Finally, let us prove that $d=b$. By adding the four $4 \times 4$ matrices above, we obtain that
$$\begin{bmatrix}
1 & 1 & 0 & 0 \\
0 & 1 & 1 & 0 \\
\delta+\epsilon+g & 0 & 1 & 1 \\
d &  b & 0  & 0 \\
\end{bmatrix}$$
has at most one non-zero eigenvalue in $\Kbar$. As its trace is $0$ and its determinant is $b-d$, one deduces
from Lemma \ref{degree4car3} that $b=d$. Therefore, $\Delta_1(A)=\begin{bmatrix}
0 & d & 0
\end{bmatrix}$.
\end{proof}

Choosing $P:=I_n+d E_{n,1}$, we replace $\calV$ with $P^{-1} \calV P$.
The new space $\calV$ satisfies all the previous assumptions, with the same $\calV_\ul$ and $\calV_\lr$ spaces,
but now we have $\Delta_1(I_3)=\Delta_1(A)=\Delta_1(B)=0$. In this situation, we aim at proving that $\Delta_1$ vanishes
everywhere on $\calF$.

First, let us sum up our findings:

\begin{prop}\label{sumupprop5}
Let $\calU$ be a $\overline{1}^\star$-spec subspace of $\Mat_n(\K)$ with dimension $\dbinom{n}{2}+1$.
Assume that:
\begin{enumerate}
\item[(i)] Neither $\calU$ nor $\calU^T$ has property (R);
\item[(ii)] $e_n$ is $\calU$-good;
\item[(iii)] $\calU_\ul=\calW_{0,\calG}^{(1^\star)}$ for some semi-reduced exceptional $\overline{1}$-spec subspace $\calG$ of $\Mat_3(\K)$.
\end{enumerate}
Then, there are scalars $\delta$ and $\epsilon$ together with a matrix $Q=\begin{bmatrix}
I_{n-1} & [0]_{(n-1) \times 1} \\
[?]_{1 \times (n-1)} & 1
\end{bmatrix} \in \GL_n(\K)$ such that
$Q^{-1}\calU Q$ contains the matrices
$$I_3 \oplus 0_{n-3}, \quad \begin{bmatrix}
0 & 1 & 0 \\
0 & 0 & 0 \\
\delta & 0 & 0
\end{bmatrix} \oplus 0_{n-3} \quad \text{and} \quad
\begin{bmatrix}
0 & 0 & 0 \\
0 & 0 & 1 \\
\epsilon & 0 & 0
\end{bmatrix} \oplus 0_{n-3}.$$
\end{prop}

From now on, we shall assume that $\calF$ is fully-reduced. To see why this is not a restrictive assumption, note that we can always
find some $R \in \GL_3(\K)$ such that $R \calF R^{-1}$ is fully-reduced, and
replace $\calV$ with $\calV':=\widetilde{R} \calV \widetilde{R}^{-1}$ for $\widetilde{R}:=R\oplus I_{n-3}$;
then, one sees that the situation is globally unchanged, with $\calV'_\ul=\calW_{0,R \calF R^{-1}}^{(1^\star)}$ as the new upper space.
With that new assumption, we can move forward:

\begin{claim}\label{claim1011}
The map $\Delta_1$ vanishes everywhere on $\calF$.
\end{claim}

\begin{proof}
We use an invariance argument. Denote by $(f_1,f_2,f_3)$ the canonical basis of $\K^3$.
By Lemma \ref{changeofgoodvectorlemma1}, the vector
$f_2-f_3$ is $\calF$-good. Lemma \ref{goodcompletionlemma}
yields a matrix $R \in \GL_3(\K)$ such that $R^{-1}\calF R$ is semi-reduced and
$R f_3=f_2-f_3$. Set $\widetilde{R}:=R \oplus I_{n-3} \in \GL_n(\K)$ and note that
$\widetilde{R}^{-1} \calV \widetilde{R}$ satisfies the assumptions of Proposition \ref{sumupprop5},
with $(\widetilde{R}^{-1} \calV \widetilde{R})_\ul=\calW_{0,R^{-1} \calF R}^{(1^\star)}$.
This gives rise to a row matrix $L=\begin{bmatrix}
a_1 & \cdots & a_{n-1}
\end{bmatrix}$ such that, for $Q:=\begin{bmatrix}
I_{n-1} & [0]_{(n-1) \times 1} \\
L & 1
\end{bmatrix}$, the space $Q^{-1} \widetilde{R}^{-1} \calV \widetilde{R} Q$ contains the matrices
$$I_3 \oplus 0_{n-3} \quad \text{and} \quad N \oplus 0_{n-3}, \quad \text{where
$N=\begin{bmatrix}
0 & 1 & 0 \\
0 & 0 & 0 \\
\delta' & 0 & 0
\end{bmatrix} \in R^{-1} \calF R$} \quad \text{for some $\delta' \in \K$.}$$
Setting $L':=\begin{bmatrix}
a_1 & a_2 & a_3
\end{bmatrix} \in \Mat_{1,3}(\K)$, one deduces that $\calV$ contains the two matrices
$$\begin{bmatrix}
I_3 & [0]_{3 \times (n-3)} \\
[0]_{(n-4) \times 3} & [0]_{(n-4) \times (n-3)} \\
L'R^{-1} & [0]_{1 \times (n-3)}
\end{bmatrix} \quad \text{and} \quad
\begin{bmatrix}
RNR^{-1} & [0]_{3 \times (n-3)} \\
[0]_{(n-4) \times 3} & [0]_{(n-4) \times (n-3)} \\
L'NR^{-1} & [0]_{1 \times (n-3)}
\end{bmatrix}.$$
However, we know that $\calV$ contains $I_3 \oplus 0_{n-3}$. As $e_n$ is $\calV$-good,
one deduces that $L'=0$. It follows that $\Delta_1(RNR^{-1})=0$. As $R f_3=f_2-f_3$,
one sees that $RNR^{-1} \in \calF$ is non-zero and vanishes at $f_2-f_3$.
By point (ii) of Lemma \ref{changeofgoodvectorlemma1}, $RNR^{-1}$ has a non-zero entry at the $(1,3)$-spot,
and therefore $(I_3,A,B,RNR^{-1})$ is a basis of $\calF$. As we already know that $\Delta_1$
vanishes at $I_3$, $A$ and $B$, we conclude that $\Delta_1=0$.
\end{proof}

\paragraph{}
Now, we turn to other special matrices in $\calV$.
Using the respective shapes of $\calV_\ul$ and $\calV_\lr$ and the fact that the linear form $a$ vanishes everywhere,
one finds three linear maps $\varphi_1 : \Mat_{1,n-4}(\K) \rightarrow \Mat_{1,2}(\K)$,
$\varphi_2 : \Mat_{1,n-4}(\K) \rightarrow \Mat_{1,n-4}(\K)$ and
$f : \Mat_{1,n-4}(\K) \rightarrow \K$ such that, for every $L \in \Mat_{1,n-4}(\K)$, the space $\calV$ contains
$$A_L:=\begin{bmatrix}
0 & [0]_{1 \times 2} & L & 0 \\
[0]_{(n-2) \times 1} & [0]_{(n-2) \times 2} & [0]_{(n-2) \times (n-4)} & [0]_{(n-2) \times 1} \\
f(L) & \varphi_1(L) & \varphi_2(L) & 0
\end{bmatrix}.$$
Finally, as $\beta$ vanishes at every nilpotent matrix of $\calV_\lr$, we find
two linear maps $\psi : \Mat_{n-2,1}(\K) \rightarrow \Mat_{n-2,1}(\K)$ and $g : \Mat_{n-2,1}(\K) \rightarrow \K$
such that, for every $C \in \Mat_{n-2,1}(\K)$,
the space $\calV$ contains the matrix
$$B_C:=\begin{bmatrix}
0 & [0]_{1 \times (n-2)} & 0 \\
\psi(C) & [0]_{(n-2) \times (n-2)} & C \\
g(C) & [0]_{1 \times (n-2)} & 0
\end{bmatrix}.$$

In the rest of the proof, we will find convenient to split up every column matrix
$C \in \Mat_{n-2,1}(\K)$ as $C=\begin{bmatrix}
C_1 \\
C_2
\end{bmatrix}$ with $C_1 \in \Mat_{2,1}(\K)$ and $C_2 \in \Mat_{n-4,1}(\K)$.

\begin{claim}\label{phi1nulclaim1}
One has $\varphi_1=0$, and there are scalars $\lambda$ and $\mu$
such that
$$\forall L \in \Mat_{1,n-4}(\K), \quad \varphi_2(L)=\lambda\,L \quad
\text{and} \quad
\forall C \in \Mat_{n-2,1}(\K), \quad \psi(C)_2=\mu\,C_2.$$
\end{claim}

\begin{proof}
The result is trivial if $n=4$. Assume that $n \geq 5$.
Let $L \in \Mat_{1,n-4}(\K)$ and $C \in \Mat_{n-2,1}(\K)$ be such that $LC_2=0$ and $C \neq 0$.
Denote by $x$ the last column of $B_C$. Then, we note that both matrices $A_L$ and $B_C$
stabilize the plane $\Vect(x,e_n)$, with induced endomorphisms represented in the basis $(x,e_n)$ by
$\begin{bmatrix}
0 & 0 \\
\varphi_1(L)C_1+\varphi_2(L)C_2 & 0
\end{bmatrix}$ and
$\begin{bmatrix}
0 & 1 \\
0 & 0
\end{bmatrix}$, respectively. Applying Lemma \ref{n=2lemma} yields $\varphi_1(L)C_1+\varphi_2(L)C_2=0$.

This holds also for the column matrix $C'=\begin{bmatrix}
-C_1 \\
C_2
\end{bmatrix}$, and hence $\varphi_1(L)C_1=0$ and $\varphi_2(L)C_2=0$.
The first identity yields $\varphi_1=0$ as it holds for all non-zero $C \in \Mat_{1,n-2}(\K)$ with $C_2=0$.
On the other hand, we find that $\varphi_2(L)X=0$ for all $(L,X) \in \Mat_{1,n-4}(\K) \times \Mat_{n-4,1}(\K)$
satisfying $LX=0$. As in the proof of Lemma \ref{lemmegeneralphietpsi}, this gives rise to a scalar $\lambda$ such that
$\varphi_2(L)=\lambda\,L$ for all $L \in \Mat_{1,n-4}(\K)$.

Again, let $L \in \Mat_{1,n-4}(\K)$ and $C \in \Mat_{n-2,1}(\K)$ be such that $LC_2=0$ and $C \neq 0$.
Setting $y:=\begin{bmatrix}
[0]_{1 \times 3} & L & 0
\end{bmatrix}^T$, we see that $A_L^T$ and $B_C^T$ stabilize the plane $\Vect(e_1,y)$, with induced
endomorphisms represented in the basis $(y,e_1)$ by the matrices $\begin{bmatrix}
0 & 1 \\
0 & 0
\end{bmatrix}$ and $\begin{bmatrix}
0 & 0 \\
L\psi(C)_2 & 0
\end{bmatrix}$, respectively. Thus, Lemma \ref{n=2lemma} shows that $L\psi(C)_2=0$.
The special case $C_2=0$ shows that $\psi(C)_2=0$ for all $C \in \Mat_{n-2,1}(\K)$ satisfying $C_2=0$.
Therefore, there exists an endomorphism $\chi$ of $\Mat_{n-4,1}(\K)$ such that $\psi(C)_2=\chi(C_2)$ for all
$C \in \Mat_{n-2,1}(\K)$. Fixing $C \in \Mat_{n-2,1}(\K)$ such that $C_1=0$, one finds that
$L\chi(C')=0$ for all $(L,C') \in \Mat_{1,n-4}(\K) \times \Mat_{n-4,1}(\K)$ for which $LC'=0$.
As in the proof of Lemma \ref{lemmegeneralphietpsi}, this gives rise to a scalar $\mu$ such that $\chi=\mu\,\id$, i.e.\
$\psi(C)_2=\mu\,C_2$ for all $C \in  \Mat_{n-2,1}(\K)$.
\end{proof}

Obviously, the scalars $\lambda$ and $\mu$ and the linear form $f$ matter only if $n \geq 5$.
In order to analyze them, we recall that $\calV$ contains $\calD_{\lcro 1,3\rcro}$.

\begin{claim}\label{exteriorclaim}
Assume that $n \geq 5$. Then, $\lambda=\mu=0$, $f=0$ and $g$ vanishes at every matrix $C \in \Mat_{n-2,1}(\K)$ for which $C_1=0$.
\end{claim}

\begin{proof}
Set $L_0:=\begin{bmatrix}
1 & 0 & \cdots & 0
\end{bmatrix} \in \Mat_{1,n-4}(\K)$ and $C_0:=\begin{bmatrix}
0 & 0 & 1 & 0 & \cdots & 0
\end{bmatrix}^T \in \Mat_{n-2,1}(\K)$.
Note that the matrices $A_{L_0}$ and $B_{C_0}$ belong to $\calV$,
stabilize the space $\Vect(e_2,e_3,e_5,\dots,e_{n-1})$ and induce endomorphisms of the quotient space
$\K^n/\Vect(e_2,e_3,e_5,\dots,e_{n-1})$ represented in the basis
$(\overline{e_1},\overline{e_4},\overline{e_n})$ by
$$\begin{bmatrix}
0 & 1 & 0 \\
0 & 0 & 0 \\
f(L_0) & \lambda & 0
\end{bmatrix}\quad \text{and} \quad
\begin{bmatrix}
0 & 0 & 0 \\
\mu & 0 & 1 \\
g(C_0) & 0 & 0
\end{bmatrix}, \quad \text{respectively.}$$
Therefore, every linear combination of those $3\times 3$ matrix has at most two eigenvalues in $\Kbar$,
which, by Lemma \ref{lemmeALetBC}, shows that $\lambda+\mu=0$.

Now, let $i \in \lcro 4,n-1\rcro$ and denote by $L \in \Mat_{1,n-4}(\K)$ the row matrix with all entries zero,
except the one at the $(i-3)$-th position, which equals $1$; denote also by $C$ the $(n-2) \times 1$ matrix
with all entries zero, except the one at the $(i-1)$-th position, which equals $1$.
Note that the three matrices $A_L$, $B_C$ and $\calD_{\lcro 1,3\rcro}$ belong to $\calV$,
stabilize the space $\Vect(e_2,\dots,e_{i-1},e_{i+1},\dots,e_{n-1})$ and induce endomorphisms of the quotient
space $\K^n/\Vect(e_2,\dots,e_{i-1},e_{i+1},\dots,e_{n-1})$ represented in the basis
$(\overline{e_1},\overline{e_i},\overline{e_n})$ by
$$\begin{bmatrix}
0 & 1 & 0 \\
0 & 0 & 0 \\
f(L) & \lambda & 0
\end{bmatrix}\quad , \quad
\begin{bmatrix}
0 & 0 & 0 \\
-\lambda & 0 & 1 \\
g(C) & 0 & 0
\end{bmatrix} \quad \text{and} \quad
\begin{bmatrix}
1 & 0 & 0 \\
0 & 0 & 0 \\
0 & 0 & 0
\end{bmatrix}, \quad \text{respectively.}$$
Conjugating those matrices with $P:=\begin{bmatrix}
1 & 0 & 0 \\
0 & 1 & 0 \\
\lambda & 0 & 1
\end{bmatrix}$, we deduce that every linear combination of the three $3 \times 3$ matrices
$$\begin{bmatrix}
0 & 1 & 0 \\
0 & 0 & 0 \\
f(L) & 0 & 0
\end{bmatrix}\quad , \quad
\begin{bmatrix}
0 & 0 & 0 \\
0 & 0 & 1 \\
g(C) & 0 & 0
\end{bmatrix} \quad \text{and} \quad
\begin{bmatrix}
1 & 0 & 0 \\
0 & 0 & 0 \\
-\lambda & 0 & 0
\end{bmatrix}$$
has at most two eigenvalues in $\Kbar$.
Then, Lemma \ref{lemmeALetBCcar3} yields $\lambda=0$ and $f(L)=g(C)=0$.
As $f$ and $g$ are linear, varying $i$ yields $f=0$, and $g(C)=0$ whenever $C_1=0$.
\end{proof}

\begin{claim}\label{vanishingofgcar3case3}
The map $g$ vanishes everywhere on $\Mat_{n-2,1}(\K)$.
\end{claim}

\begin{proof}
The fact that
$g$ vanishes at $\begin{bmatrix}
0 & 1 & 0 & \cdots & 0
\end{bmatrix}^T$ comes as a by-product of the proof of Claim \ref{delta1seminulenA}
(in the notation of that proof, this comes from equality $h=0$).

Then, using Claim \ref{exteriorclaim}, one sees that it suffices to prove that $g$ vanishes
at $C_0:=\begin{bmatrix}
1 & 0 & \cdots & 0
\end{bmatrix}^T$.
Note that the three matrices $M_A$, $B_{C_0}$ and $\calD_{\lcro 1,3\rcro}$ belong to $\calV$ and that they
stabilize $\Vect(e_3,\dots,e_{n-1})$, and therefore induce endomorphisms of the quotient space
$\K^n/\Vect(e_3,\dots,e_{n-1})$ that are represented in the basis $(\overline{e_1},\overline{e_2},\overline{e_n})$
by matrices
$$\begin{bmatrix}
0 & 1 & 0 \\
0 & 0 & 0 \\
0 & 0 & 0
\end{bmatrix}, \quad \begin{bmatrix}
0 & 0 & 0 \\
? & 0 & 1 \\
g(C_0) & 0 & 0
\end{bmatrix} \quad \text{and} \quad
\begin{bmatrix}
1 & 0 & 0 \\
0 & 1 & 0 \\
0 & 0 & 0
\end{bmatrix}, \quad \text{respectively.}$$
Applying Lemma \ref{lemmeALetBC} to the first two matrices, one finds that the second one equals
$\begin{bmatrix}
0 & 0 & 0 \\
0 & 0 & 1 \\
g(C_0) & 0 & 0
\end{bmatrix}$. Then, Lemma \ref{lemmeALetBCcar3} yields $g(C_0)=0$.
\end{proof}

Remember that, for all $C \in \Mat_{n-2,1}(\K)$, we have $\psi(C)_2=0$, so that
$\calV$ contains
$$B_C=\begin{bmatrix}
0 & [0]_{1 \times (n-2)} & 0 \\
\psi(C) & [0]_{(n-2) \times (n-2)} & C \\
0 & [0]_{1 \times (n-2)} & 0
\end{bmatrix} \quad \text{with} \quad \psi(C)=\begin{bmatrix}
\psi(C)_1 \\
[0]_{(n-4) \times 1}
\end{bmatrix}.$$

\begin{claim}\label{vanishingofpsi}
The map $\psi$ vanishes everywhere on $\Mat_{n-2,1}(\K)$.
\end{claim}

\begin{proof}
Let $C \in \Mat_{n-2,1}(\K)$ and $N \in \calF$. Set $S(C):=\begin{bmatrix}
0 & [0]_{1 \times 2} \\
\psi(C)_1 & [0]_{2 \times 2}
\end{bmatrix} \in \Mat_3(\K)$, and note that $M_N+B_C$ belongs to $\calV$ and stabilizes the space $\Vect(e_1,e_2,e_3)$,
with induced endomorphism represented in the basis $(e_1,e_2,e_3)$ by
$N+S(C)$. One deduces that $\calF+\im S$ is a $\overline{1}^\star$-spec subspace of $\Mat_3(\K)$.
However, $\calF$ is a maximal $\overline{1}$-spec subspace of $\Mat_3(\K)$, and therefore $\im S \subset \calF$.
Since $\calF$ is semi-reduced, the vector $e_1$ is $\calF^T$-good, which yields $\psi(C)_1=0$ for all $C \in \Mat_{n-2,1}(\K)$.
Therefore, $\psi=0$.
\end{proof}

The situation is far clearer now:
for all $(N,L,C) \in \calF \times \Mat_{1,n-4}(\K) \times \Mat_{n-2,1}(\K)$, the space $\calF$ contains the matrices
$$N \oplus 0_{n-3}, \quad
\begin{bmatrix}
[0]_{1 \times 3} & L & 0 \\
[0]_{(n-1) \times 3} & [0]_{(n-1) \times (n-4)} & [0]_{(n-1) \times 1}
\end{bmatrix}
\quad \text{and}
\quad
\begin{bmatrix}
[0]_{1 \times (n-1)} & 0 \\
[0]_{(n-2) \times (n-1)} & C \\
[0]_{1 \times (n-1)} & 0 \\
\end{bmatrix}.$$
We deduce the following general result, summarizing a tiny part of our findings:

\begin{prop}\label{sumupprop6}
Let $\calU$ be a $\overline{1}^\star$-spec subspace of $\Mat_n(\K)$ with dimension $\dbinom{n}{2}+1$.
Assume that:
\begin{enumerate}
\item[(i)] Neither $\calU$ nor $\calU^T$ has property (R);
\item[(ii)] $e_n$ is $\calU$-good;
\item[(iii)] $\calU_\ul=\calW_{0,\calG}^{(1^\star)}$ for some semi-reduced exceptional $\overline{1}$-spec subspace $\calG$ of $\Mat_3(\K)$.
\end{enumerate}
Then, for some matrix of the form $Q=\begin{bmatrix}
I_{n-1} & [0]_{(n-1) \times 1} \\
[?]_{1 \times (n-1)} & 1
\end{bmatrix}\in \GL_n(\K)$, the space $Q^{-1} \calU Q$ contains $E_{3,n}$ and $I_3 \oplus 0_{n-3}$.
\end{prop}

Now, we can prove:

\begin{claim}\label{car3case3E1ninV}
The matrix $E_{1,n}$ belongs to $\calV$.
\end{claim}

\begin{proof}
By Lemmas \ref{goodcompletionlemma} and \ref{changeofgoodvectorlemma2}, there exists a non-singular matrix $R \in \GL_3(\K)$ such that
$R^{-1} \calF R$ is semi-reduced and $R e_3=e_1$. Setting $\widetilde{R}:=R \oplus I_{n-3}$,
one deduces that $\calU:=\widetilde{R}^{-1} \calV \widetilde{R}$ satisfies the assumptions of Proposition \ref{sumupprop6}, with
$\calU_\ul=\calW_{0,R^{-1} \calF R}^{(1^\star)}$. Then, for some matrix of the form $Q=\begin{bmatrix}
I_{n-1} & [0]_{(n-1) \times 1} \\
L_1 & 1
\end{bmatrix}$, the space $Q^{-1} \widetilde{R}^{-1} \calV \widetilde{R} Q$ contains $E_{3,n}$ and
$I_3 \oplus 0_{n-3}$. Writing $L_1=\begin{bmatrix}
L_2 & L_3
\end{bmatrix}$ with $L_2 \in \Mat_{1,3}(\K)$ and $L_3 \in \Mat_{1,n-4}(\K)$, one deduces
that $\calV$ contains
$$\begin{bmatrix}
I_3 & [0]_{3 \times (n-3)} \\
[0]_{(n-4) \times 3} & [0]_{(n-4) \times (n-3)} \\
L_2 & [0]_{1 \times (n-3)} \\
\end{bmatrix}.$$
However, $\calV$ also contains $I_3 \oplus 0_{n-3}$. As $e_n$ is $\calV$-good, one deduces that $L_2=0$.
Then, using $E_{3,n} \in Q^{-1}\widetilde{R}^{-1} \calU \widetilde{R} Q $, one deduces that
$\calV$ contains
$$H=\begin{bmatrix}
[0]_{1 \times 3} & -L_3 & 1 \\
[0]_{(n-1) \times 3} & [0]_{(n-1) \times (n-4)} & [0]_{(n-1) \times 1}
\end{bmatrix}.$$
The conclusion follows by adding $A_{L_3}$ to $H$.
\end{proof}

We are almost ready to conclude. Let $U \in \NT_{n-2}(\K)$ be \emph{with a zero second column.}
We know that $\calV_\ul$ contains the nilpotent matrix
$0 \oplus U$, which has zero as its first row. Using $\calV_\lr=\calV_{\{1,2\}}^{(1^\star)}$, one finds some
$e \in \K$ such that $\calV$ contains
$$E_U=\begin{bmatrix}
0 & [0]_{1 \times (n-2)} & 0 \\
[0]_{(n-2) \times 1} & U & [0]_{(n-2) \times 1} \\
e & [0]_{1 \times (n-2)} & 0
\end{bmatrix}$$
(note that $e$ is uniquely determined since $e_n$ is $\calV$-good).
Using $E_{1,n} \in \calV$ and Lemma \ref{n=2lemma}, a similar line of reasoning as in the proof of Claim \ref{hclaim1}
yields $e=0$.

\paragraph{}
Finally, combining matrices of the forms $M_N$, $A_L$, $B_C$, and $E_U$ together with $E_{1,n}$,
one finds the inclusion $\calW_{0,\calF}^{(1^\star)} \subset \calV$, which yields
$\calW_{0,\calF}^{(1^\star)} = \calV$ as the dimensions are equal on both sides.
This concludes Case 3, which was the last remaining one in the proof of Theorem
\ref{car3theo1}. Thus, Theorem \ref{car3theo1} is established.

\section{Large spaces of matrices with at most two eigenvalues (II): characteristic $3$}\label{equalitysection5}

This final section is devoted to the proof of Theorem \ref{car3theo2}.
The overall strategy is similar to that of the proof of Theorem \ref{2specequality}, i.e.\
we will derive Theorem \ref{car3theo2} from Theorem \ref{car3theo1}.
There are three main additional difficulties, compared to the proof of Theorem \ref{car3theo1}:
the results on the compatibility between the $\calV_\ul$ and $\calV_\lr$ spaces are substantially harder to obtain;
the case $n=4$ involves special difficulties (as in the proof of Theorem \ref{2specequality}, with some different details however);
finally, the case $n=6$ is the most technical one, due to the very special case mentioned in Theorem \ref{car3theo2}.

Throughout the section, we assume that the field $\K$ has characteristic $3$.

\subsection{Setting things up}

Let $n \geq 3$ be an integer, and $\calV$ be a $\overline{2}$-spec subspace of $\Mat_n(\K)$ with dimension $\dbinom{n}{2}+2$.
As in Section \ref{setup3section}, one obtains that $I_n \in \calV$.

Denote by $(e_1,\dots,e_n)$ the canonical basis of $\K^n$.
Using Proposition \ref{goodprop}, we see that no generality is lost in assuming that
$e_n$ is $\calV$-good. As in Section \ref{setup3section}, we write every matrix $M$ of $\calV$ as
$$M=\begin{bmatrix}
K(M) & C(M) \\
[?]_{1 \times (n-1)} & a(M)
\end{bmatrix}$$
and set
$$\calW:=\bigl\{M \in \calV : \; C(M)=0 \quad \text{and}\quad a(M)=0\bigr\} \quad \text{and} \quad
\calV_\ul:=K(\calW).$$
As in Section \ref{equalitysection3}, one shows that $\calV_\ul$ is a $\overline{1}^\star$-spec subspace of $\Mat_{n-1}(\K)$
with dimension $\dbinom{n-1}{2}+1$. Therefore, Theorem \ref{car3theo1} yields:
\begin{itemize}
\item[(A)] The $\overline{1}^\star$-spec subspace $\calV_\ul$ of $\Mat_{n-1}(\K)$
has type (I) or type (II).
\end{itemize}
In other words, there exists a matrix $Q \in \GL_{n-1}(\K)$ such that either $Q\calV_\ul Q^{-1}=\calV_I^{(1^\star)}$
for some non-empty subset $I$ of $\lcro 1,n-1\rcro$, or $Q\calV_\ul Q^{-1}=\calW_{p,\calF}^{(1^\star)}$
for some $p \in \lcro 0,n-4\rcro$ and some exceptional $\overline{1}$-spec subspace $\calF$ of $\Mat_3(\K)$.
Conjugating $\calV$ with $P:=Q \oplus 1$, we see that no generality is lost in assuming:
\begin{itemize}
\item[(A')] Either $\calV_\ul=\calV_I^{(1^\star)}$ for some non-empty subset $I$ of $\lcro 1,n-1\rcro$, or there exists
an integer $p \in \lcro 0,n-4\rcro$ and an exceptional $\overline{1}$-spec subspace $\calF$ of $\Mat_3(\K)$
such that $\calV_\ul=\calW_{p,\calF}^{(1^\star)}$. Moreover, in the second case and when $p=0$,
the space $\calF$ is semi-reduced.
\end{itemize}

In any case, the first vector of the canonical basis of $\K^{n-1}$ is $(\calV_\ul)^T$-good,
which, combined with the fact that $e_n$ is $\calV$-good, leads to:
\begin{itemize}
\item[(B)] The vector $e_1$ is $\calV^T$-good.
\end{itemize}
Now, let us write every matrix $M$ of $\calV$ as
$$M=\begin{bmatrix}
b(M) & R(M) \\
[?]_{(n-1) \times 1} & K'(M)
\end{bmatrix},$$
and set
$$\calW':=\bigl\{M \in \calV : \; R(M)=0 \quad \text{and}\quad b(M)=0\bigr\} \quad \text{and} \quad
\calV_\lr:=K'(\calW').$$
As before, the fact that $e_1$ is $\calV^T$-good yields that $\calV_\lr$
is a $\overline{1}^\star$-spec subspace of $\Mat_{n-1}(\K)$ with dimension $\dbinom{n-1}{2}+1$.
Again, using the fact that $e_n$ is $\calV$-good and that the first vector of the canonical basis of
$\K^{n-1}$ is $(\calV_\ul)^T$-good, one finds:
 \begin{itemize}
\item[(C)] The last vector of the canonical basis of $\K^{n-1}$ is $\calV_\lr$-good.
\end{itemize}

If $\calV_\ul$ has type (II), we define
$$I:=\lcro 1,3\rcro.$$
In any case, $\calV$ contains a matrix of the form
$$\begin{bmatrix}
\calD_I & [0]_{(n-1) \times 1} \\
[?]_{1 \times (n-1)} & 0
\end{bmatrix}.$$
If $1 \not\in I$, it follows that $\calV_\lr$ contains a matrix of the form
$$\begin{bmatrix}
\calD_{I-1} & [0]_{(n-2) \times 1} \\
[?]_{1 \times (n-2)} & 0
\end{bmatrix}.$$
If $1 \in I$, subtracting a matrix of the above type from $I_n$ shows that
$\calV_\lr$ contains a matrix of the form
$$\begin{bmatrix}
\calD_{I'} & [0]_{(n-2) \times 1} \\
[?]_{1 \times (n-2)} & 1
\end{bmatrix}, \quad \text{where $I':=(\lcro 1,n-1\rcro \setminus I)-1$.}$$
This leads us to defining a non-empty subset $J$ of $\lcro 1,n-1\rcro$ as follows:
$$J:=\begin{cases}
I-1 & \text{if $1 \not\in I$} \\
\bigl((\lcro 1,n-1\rcro \setminus I)-1\bigr) \cup \{n-1\} & \text{otherwise}.
\end{cases}$$
In any case, one notes that $\calV_\lr$ contains a matrix of the form
$$\calD_J+\begin{bmatrix}
[0]_{(n-2) \times (n-2)} & [0]_{(n-2) \times 1} \\
[?]_{1 \times (n-2)} & 0
\end{bmatrix}.$$
Finally, we define a subspace $\calV_\md$ of $\Mat_{n-2}(\K)$ as follows:

\begin{center}
\begin{tabular}{| c | c |}
\hline
If $\calV_\ul=\cdots$ & then $\calV_\md:=\cdots$ \\
\hline
\hline
$\calV_{I}^{(1^\star)}$, where $I \subset  \lcro 2,n-1\rcro$ & $\calV_{I-1}^{(1^\star)}$ \\
\hline
$\calV_{I}^{(1^\star)}$, where $\{1\} \subset I \subsetneq \lcro 1,n-1\rcro$ & $\calV_{J\setminus \{n-1\}}^{(1^\star)}$ \\
\hline
$\calV_{\lcro 1,n-1\rcro}^{(1^\star)}$ & $\NT_{n-2}(\K)$ \\
\hline
$\calW_{0,\calF}^{(1^\star)}$, for a semi-reduced  & $\calV_{\lcro 3,n-2\rcro}^{(1^\star)}$ \\
exceptional $\overline{1}$-spec subspace $\calF$ of $\Mat_3(\K)$ & \\
\hline
$\calW_{p,\calF}^{(1^\star)}$, where $p \in \lcro 1,n-4\rcro$ and $\calF$ is an  & $\calW_{p-1,\calF}^{(1^\star)}$ \\
exceptional $\overline{1}$-spec subspace of $\Mat_3(\K)$ & \\
\hline
\end{tabular}
\end{center}

In any case, by noticing that if a matrix $M$ of $\calW$ has $\begin{bmatrix}
1 & 0 & \cdots & 0
\end{bmatrix}$ as its first row, then the matrix $I_n-M$ belongs to $\calW'$, one readily obtains that,
for every $U \in \calV_\md$, the space $\calV_\lr$ contains a matrix of the form
$$\begin{bmatrix}
U & [0]_{(n-2) \times 1} \\
[?]_{1 \times (n-2)} & ?
\end{bmatrix}.$$

\subsection{Diagonal-compatibility}

Now, we examine how the type of $\calV_\ul$ determines the one of $\calV_\lr$.
We distinguish between two cases.

\begin{claim}[First compatibility claim]\label{cornerclaim2car3}
Assume that there exists no exceptional $\overline{1}$-spec subspace $\calG$ of $\Mat_3(\K)$ such that
$\calV_\lr \simeq \calW_{n-4,\calG}^{(1^\star)}$. Then,
there exists a matrix of the form
$Q=\begin{bmatrix}
I_{n-2} & [0]_{(n-2) \times 1} \\
[?]_{1 \times (n-2)} & 1
\end{bmatrix} \in \GL_{n-1}(\K)$ such that the following implications hold:
\begin{center}
\begin{tabular}{| c | c |}
\hline
If $\calV_\ul=\cdots$ & then $\calV_\lr=Q\,\calE\,Q^{-1}$, where $\calE=\cdots$ \\
\hline
\hline
$\calV_I^{(1^\star)}$ for some non-empty $I \subset \lcro 1,n-1\rcro$ & $\calV_J^{(1^\star)}$ \\
\hline
$\calW_{0,\calF}^{(1^\star)}$ for some semi-reduced exceptional & $\calV_{\lcro 3,n-1\rcro}^{(1^\star)}$ \\
$\overline{1}$-spec subspace $\calF$ of $\Mat_3(\K)$ &   \\
\hline
$\calW_{p,\calF}^{(1^\star)}$ for some exceptional &  \\
$\overline{1}$-spec subspace $\calF$ of $\Mat_3(\K)$ & $\calW_{p-1,\calF}^{(1^\star)}$ \\
and some $p \in \lcro 1,n-4\rcro$ & \\
\hline
\end{tabular}
\end{center}
\end{claim}

\begin{proof}
If there exists a non-empty subset $J'$ of $\lcro 1,n-1\rcro$ such that $\calV_\lr \simeq \calV_{J'}^{(1^\star)}$,
then we set $\calA:=\calV_{J'}^{(1^\star)}$. Otherwise, there exists a (unique) $q \in \lcro 0,n-5\rcro$ and an
exceptional $\overline{1}$-spec subspace $\calG$ of $\Mat_3(\K)$ such that $\calV_\lr \simeq \calW_{q,\calG}^{(1^\star)}$,
in which case we set $\calA:=\calW_{q,\calG}^{(1^\star)}$.

In any case, we have a non-singular matrix $Q \in \GL_{n-1}(\K)$ such that $\calV_{\lr}=Q\,\calA\,Q^{-1}$.
As in the proof of Claim \ref{cornerclaim1car3}, one reduces the situation to the one where
$Q$ fixes the last vector of the canonical basis of $\K^{n-1}$; then, one splits up
$$Q=Q_1\,Q_2 \quad \text{where $Q_1$ has the form $\begin{bmatrix}
I_{n-2} & [0]_{(n-2) \times 1} \\
[?]_{1 \times (n-2)} & 1
\end{bmatrix}$}$$
and
$$Q_2=R \oplus 1 \quad \text{for some $R \in \GL_{n-2}(\K)$.}$$
From there, we set
$$\calE:=Q_2 \calA Q_2^{-1}.$$
Let $U \in \calV_\md$.
As we know that $\calV_\lr$ contains a matrix of the form
$\begin{bmatrix}
U & [0]_{(n-2) \times 1} \\
[?]_{1 \times (n-2)} & ?
\end{bmatrix}$, we find that
$\calA$ contains a matrix of the form
$$\begin{bmatrix}
R^{-1}UR & [0]_{(n-2) \times 1} \\
[?]_{1 \times (n-2)} & ?
\end{bmatrix}.$$
Finally, we define a subspace $\calB$ of $\Mat_{n-2}(\K)$ as follows:
\begin{center}
\begin{tabular}{| c | c |}
\hline
If $\calA=\cdots$ & then $\calB:=\cdots$ \\
\hline
\hline
$\calV_{\{n-1\}}^{(1^\star)}$ & $\NT_{n-2}(\K)$ \\
\hline
$\calV_{J'}^{(1^\star)}$ for some non-empty subset $J' \subset \lcro 1,n-1\rcro$ & $\calV_{J' \setminus \{n-1\}}^{(1^\star)}$ \\
with $J' \neq \{n-1\}$ & \\
\hline
$\calW_{q,\calG}^{(1^\star)}$ for some exceptional &  \\
$\overline{1}$-spec subspace $\calG$ of $\Mat_3(\K)$ & $\calW_{q,\calG}^{(1^\star)}$.  \\
and some $q \in \lcro 0,n-5\rcro$ & \\
\hline
\end{tabular}
\end{center}
In any case, one checks that $R^{-1} \calV_\md R \subset \calB$. From there, the line of reasoning is the same one as in the proof of Claim \ref{cornerclaim1car3}.
\begin{itemize}
\item If $\calV_\ul$ has type (I) or equals $\calW_{0,\calF}^{(1^\star)}$ for some semi-reduced
exceptional $\overline{1}$-spec subspace $\calF$ of $\Mat_3(\K)$, then one uses Lemma
\ref{nilpotenthyperplane} to obtain that $\calB$ must have type (I) or must equal $\NT_{n-2}(\K)$, so that $\calA$ must have type (I).
Then, one follows the line of reasoning from the proof of Claim \ref{corner3claim} to obtain $\calE=\calV_J^{(1^\star)}$.

\item If $\calV_\ul=\calW_{p,\calF}^{(1^\star)}$ for some $p \in \lcro 1,n-4\rcro$ and some exceptional
$\overline{1}^\star$-spec subspace $\calF$ of $\Mat_3(\K)$, then
 $\calV_\md$ is a $\overline{1}^\star$-spec subspace of $\Mat_{n-2}(\K)$ of type (II) and $R^{-1} \calV_\md R=\calB$; one
 deduces that $\calA$ has type (II), and one concludes that
$$\calE=Q_2(\calB \vee \{0\})Q_2^{-1}=(R\calB R^{-1}) \vee \{0\}=\calV_\md \vee \{0\}=\calW_{p-1,\calF}^{(1^\star)} \vee \{0\}.$$
\end{itemize}
As $\calV_\lr=Q_1 \calE Q_1^{-1}$, this completes the proof.
\end{proof}

\begin{claim}[Second compatibility claim]\label{cornerclaim3car3}
Assume that there exists an exceptional $\overline{1}$-spec subspace $\calG$ of $\Mat_3(\K)$ such that
$\calV_\lr \simeq \calW^{(1^\star)}_{n-4,\calG}$. Then, $\calV_\ul=\calV_{\lcro 1,n-3\rcro}^{(1^\star)}$ unless
$n=6$, in which case another possibility is that $\calV_\ul=\calW_{0,\calF}^{(1^\star)}$ for some
exceptional $\overline{1}$-spec subspace $\calF$ of $\Mat_3(\K)$. \\
Moreover, there exists a semi-reduced exceptional $\overline{1}$-spec subspace $\calG'$ of $\Mat_3(\K)$
together with a matrix of the form
$Q=\begin{bmatrix}
I_{n-2} & [0]_{(n-2) \times 1} \\
[?]_{1 \times (n-2)} & 1
\end{bmatrix}\in \GL_{n-1}(\K)$ such that
$$\calV_\lr=Q\,\calW_{n-4,\calG'}^{(1^\star)}\, Q^{-1}.$$
\end{claim}

\begin{proof}
Set $\calA:=\calW_{n-4,\calG}^{(1^\star)}$. Then, we have a matrix $Q \in \GL_{n-1}(\K)$ such that $\calV_\lr=Q \calA Q^{-1}$.
Denote by $(f_1,\dots,f_{n-1})$ the canonical basis of $\K^{n-1}$.
Note that no vector of $\Vect(f_1,\dots,f_{n-4})$ is $\calA$-good. As $f_{n-1}$ is $\calV_\lr$-good, one deduces that
$Q^{-1} f_{n-1}$ is $\calA$-good. Then, we write $Q^{-1} f_{n-1}=x+y$ where $x \in \Vect(f_1,\dots,f_{n-4})$
and $y \in \Vect(f_{n-3},f_{n-2},f_{n-1})$, with respective lists of coordinates $x' \in \K^{n-4}$ and $y' \in \K^3$
in the bases $(f_1,\dots,f_{n-4})$ and $(f_{n-3},f_{n-2},f_{n-1})$. As $x+y$ is $\calA$-good,
the vector $y'$ is $\calG$-good. By Lemma \ref{goodcompletionlemma}, one may find a non-singular matrix $S \in \GL_3(\K)$ such that
$S\times \begin{bmatrix}
0 \\
0 \\
1
\end{bmatrix}=y'$ and $\calG':=S^{-1}\calG S$ is semi-reduced.
Setting
$$C:=\begin{bmatrix}
[0]_{(n-4) \times 2} & x'
\end{bmatrix}\in \Mat_{n-4,3}(\K) \quad \text{and} \quad
Q':=\begin{bmatrix}
I_{n-4} & C \\
[0]_{3 \times (n-4)} & S
\end{bmatrix}\in \GL_{n-1}(\K),$$
one finds that $Q'f_{n-1}=x+y=Q^{-1} f_{n-1}$ and
$$(Q')^{-1} \calA Q'=\calW_{n-4,\calG'}^{(1^\star)} \; .$$
Then, $QQ' f_{n-1}=f_{n-1}$ and $\calV_\lr=(QQ')\calW_{n-4,\calG'}^{(1^\star)} (QQ')^{-1}$.
Therefore, we see that no generality is lost in making the following additional assumptions:
\begin{center}
$Qf_{n-1}=f_{n-1}$ and $\calG$ is semi-reduced.
\end{center}
In that case, we may find some $R \in \GL_{n-2}(\K)$ such that $Q$ splits up as
$$Q=Q_1\,Q_2 \quad \text{where $Q_1=\begin{bmatrix}
I_{n-2} & [0]_{(n-2) \times 1} \\
[?]_{1 \times (n-2)} & 1
\end{bmatrix}$ and $Q_2=R \oplus 1$.}$$
Then, we set
$$\calE:=Q_2\,\calW_{n-4,\calG}^{(1^\star)}\,Q_2^{-1}=Q_1^{-1} \calV_\lr Q_1.$$

Next, we prove that $R$ is upper-triangular. As $\calG$ is semi-reduced, it contains a matrix of the form
$\begin{bmatrix}
0 & 1 & 0 \\
0 & 0 & 0 \\
? & 0 & 0
\end{bmatrix}$, which yields that $\calA$ contains a matrix of the form
$\begin{bmatrix}
U & [0]_{(n-2) \times 1} \\
[?]_{1 \times (n-2)} & 0
\end{bmatrix}$ for every $U \in \NT_{n-2}(\K)$. Judging from the shape of $Q_1$, one deduces that, for every $U \in \NT_{n-2}(\K)$, the space
$\calV_\ul$ contains a matrix of the form
$$\begin{bmatrix}
0 & [0]_{1 \times (n-2)} \\
[?]_{(n-2) \times 1} & RUR^{-1}
\end{bmatrix}.$$
If $\calV_\ul=\calW_{p,\calF}^{(1^\star)}$ for some $p \in \lcro 1,n-4\rcro$ and some
exceptional $\overline{1}$-spec subspace of $\Mat_3(\K)$, then, we deduce that
$R \NT_{n-2}(\K) R^{-1}$ is a hyperplane of the subspace $\calW_{p-1,\calF}^{(1^\star)}$ of $\Mat_{n-2}(\K)$,
which contradicts Lemma \ref{nilpotenthyperplane}. Thus, one deduces from property
(A') that every matrix of $\calV_\ul$ with a zero first row has the form
$\begin{bmatrix}
0 & [0]_{1 \times (n-2)} \\
[?]_{(n-2) \times 1} & V
\end{bmatrix}$ for some upper-triangular matrix $V \in \Mat_{n-2}(\K)$.
Therefore, for every $U \in \NT_{n-2}(\K)$, the matrix $R U R^{-1}$ is upper-triangular, and thus it is strictly upper-triangular
as it is also nilpotent. The inclusion $R \NT_{n-2}(\K)R^{-1} \subset \NT_{n-2}(\K)$ ensues, and hence
Lemma \ref{normalizerlemma} yields that $R$ is upper-triangular.

As $\calA$ contains $0_{n-4}\oplus I_3$ and $R$ is upper-triangular, we find that
$\calV_\lr$ contains a matrix of the form
$$N=\begin{bmatrix}
T & [0]_{(n-2) \times 1} \\
[?]_{1 \times (n-2)} & 1
\end{bmatrix},$$
where $T \in \Mat_{n-2}(\K)$ is upper-triangular with diagonal entries $0,\dots,0,1,1$.
Then, the last column of $I_{n-1}-N$ is zero, which yields that $\calV_\ul$ contains a matrix of the form
$$N'=\begin{bmatrix}
1 & [0]_{1 \times (n-2)} \\
[?]_{(n-2) \times 1} & T'
\end{bmatrix},$$
where $T' \in \Mat_{n-2}(\K)$ is upper-triangular with diagonal entries $1,\dots,1,0,0$.
Judging from the shape of $\calV_\ul$ (see property (A')), one deduces from $N' \in \calV_\ul$ that either
$\calV_\ul$ has type (I) and therefore $\calV_\ul=\calV_{\lcro 1,n-3\rcro}^{(1^\star)}$, or
$\calV_\ul=\calW_{0,\calF}^{(1^\star)}$ for some exceptional $\overline{1}$-spec subspace $\calF$ of $\Mat_3(\K)$
and then $n=6$ (this is obtained by counting the $1$'s on the diagonal of $N'$).

As $R$ is upper-triangular, we may split it as
$R=R_1\,R_2$, where
$$R_1=I_{n-4} \oplus Z \quad \text{and} \quad
R_2=\begin{bmatrix}
[?]_{(n-4) \times (n-4)} & [?]_{(n-4) \times 2} \\
[0]_{2 \times (n-4)} & I_2
\end{bmatrix},$$
with $Z \in \GL_2(\K)$ upper-triangular.
Then, noticing that $(R_2 \oplus 1)\,\calA\,(R_2 \oplus 1)^{-1}=\calA$, one finds
$$\calE=\calW_{n-4,\calG''}^{(1^\star)}, \; \text{where}\;
\calG'':=(Z \oplus 1) \,\calG \,(Z \oplus 1)^{-1}.$$
As $\calG$ is semi-reduced, one checks that $\calG''$ satisfies the assumptions of Lemma \ref{goodcompletionlemma2}
(note that $e_3$ is $\calG$-good, and therefore $\calG''$-good; the second assumption of Lemma \ref{goodcompletionlemma2} is easily obtained).
This yields two scalars $\alpha$ and $\beta$ such that the space
$$\calH:=\begin{bmatrix}
1 & 0 & 0 \\
0 & 1 & 0 \\
\alpha & \beta & 1
\end{bmatrix}^{-1}\, \calG'' \begin{bmatrix}
1 & 0 & 0 \\
0 & 1 & 0 \\
\alpha & \beta & 1
\end{bmatrix}$$
is semi-reduced.
Finally, set $Q_3:=I_{n-4} \oplus \begin{bmatrix}
1 & 0 & 0 \\
0 & 1 & 0 \\
\alpha & \beta & 1
\end{bmatrix}$ and $Q_4:=Q_1Q_3$. Then,
$$Q_3^{-1} \calW_{n-4,\calG''}^{(1^\star)} Q_3=\calW_{n-4,\calH}^{(1^\star)}$$
and hence
$$\calV_\lr=(Q_1Q_3) \calW_{n-4,\calH}^{(1^\star)} (Q_1Q_3)^{-1}.$$
Our proof is complete as $Q_1Q_3$ has the form
$\begin{bmatrix}
I_{n-2} & [0]_{(n-2) \times 1} \\
[?]_{1 \times (n-2)} & 1
\end{bmatrix}$.
\end{proof}

Now, we have a clear picture of the possible pairs $(\calV_\ul,\calV_\lr)$.
Using a matrix $Q \in \GL_{n-1}(\K)$ given by Claim \ref{cornerclaim2car3} or Claim
\ref{cornerclaim3car3}, setting $\widetilde{Q}:=1 \oplus Q$, and replacing
$\calV$ with $\widetilde{Q}^{-1} \calV \widetilde{Q}$, one sees that
our basic assumptions are unchanged - i.e.\ $e_n$ is $\calV$-good and $e_1$ is $\calV^T$-good -
but now one of the following situations holds:

\begin{center}
\begin{tabular}{| c | c | c  | c |}
\hline
Case & $n$ & $\calV_\ul= \cdots$ & $\calV_\lr= \cdots$ \\
\hline
\hline
1 & $n \geq 4$ & $\calW_{0,\calF}^{(1^\star)}$, where $\calF$ is an exceptional  & $\calV_{\lcro 3,n-1\rcro}^{(1^\star)}$ \\
& & semi-reduced $\overline{1}$-spec subspace  &  \\
& &   of $\Mat_3(\K)$ & \\
\hline
2 & $n=6$ & $\calW_{0,\calF}^{(1^\star)}$, where $\calF$ is an exceptional  & $\calW_{2,\calG}^{(1^\star)}$, where $\calG$ is an exceptional  \\
  &    & semi-reduced $\overline{1}$-spec subspace & semi-reduced $\overline{1}$-spec subspace \\
  & &   of $\Mat_3(\K)$ & of $\Mat_3(\K)$ \\
\hline
3 & indifferent & $\calV_I^{(1^\star)}$, where $I \subset \lcro 1,n-1\rcro$   &
$\calV_J^{(1^\star)}$, where \\
& & is non-empty & $J:=I-1$ if $1 \in I$; \\
& & & $J:=\bigl((\lcro 1,n-1\rcro \setminus I)-1\bigr)\cup \{n-1\}$ \\
& & & otherwise \\
\hline
4 & $n \geq 5$ & $\calW_{p,\calF}^{(1^\star)}$, where $\calF$ is an exceptional  & $\calW_{p-1,\calF}^{(1^\star)}$ \\
  &    & $\overline{1}$-spec subspace  of $\Mat_3(\K)$, & \\
  & & and $p \in \lcro 1,n-4\rcro$ &  \\
\hline
5 & $n \geq 4$ & $\calV_{\lcro 1,n-3\rcro}^{(1^\star)}$ & $\calW_{n-4,\calG}^{(1^\star)}$, where $\calG$ is an exceptional  \\
  &  &  & semi-reduced $\overline{1}$-spec subspace \\
  & & & of $\Mat_3(\K)$ \\
\hline
\end{tabular}
\end{center}

Note that, in Cases 1 and 2, the space $\calF$ may even be assumed to be fully-reduced.

As in the end of Section \ref{setupcar31star}, we can eliminate some cases by using the
``transpose and conjugate" argument. Denote by $K_n$ the permutation matrix of $\GL_n(\K)$
associated with the permutation $\sigma_n: i \mapsto n+1-i$, and note that:
\begin{itemize}
\item Given a non-empty and proper subset $\widetilde{I}$ of $\lcro 1,n\rcro$, one has
$$K_n \bigl(\calV_{\widetilde{I}}^{(2)}\bigr)^T K_n^{-1}=\calV_{\sigma_n(\widetilde{I})}^{(2)}.$$
\item Given an exceptional $\overline{1}$-spec subspace $\calF$ of $\Mat_3(\K)$ and an integer
$p \in \lcro 0,n-3\rcro$, one has
$$K_n \bigl(\calW_{p,\calF}^{(2)}\bigr)^T K_n^{-1}=\calW_{n-3-p,K_3\calF^T K_3^{-1}}^{(2)}\quad ,$$
and $K_3 \calF^T K_3^{-1}$ is an exceptional $\overline{1}$-spec subspace of $\Mat_3(\K)$.
Moreover, $K_3 \calF^T K_3^{-1}$ is semi-reduced if $\calF$ is semi-reduced.
\item Given two exceptional $\overline{1}$-spec subspaces $\calF$ and $\calG$ of $\Mat_3(\K)$,
one has
$$K_6(\calF \vee \calG)^TK_6^{-1}=(K_3 \calG^T K_3^{-1}) \vee (K_3 \calF^T K_3^{-1}).$$
\end{itemize}
With this method, we can discard the following special cases in the rest of our study:

\begin{enumerate}[(i)]
\item Case $5$: by the ``transpose and conjugate" argument,  
we reduce this case to Case 1 (note that $K_3 \calG^T K_3^{-1}$ is semi-reduced). 

\item Case 3 with $I=\lcro 3,n-1\rcro$: we reduce it to Case $3$ with
$I=\lcro 2,n-2\rcro$.

\item Case 4 with $n=6$ and $p=2$: we reduce it to Case $4$ with $n=6$ and $p=1$.
\end{enumerate}

\vskip 2mm
Before we discuss the above cases separately, let us finish with a general result:
we have shown that $\calV$ contains a matrix of the form
$M=\begin{bmatrix}
\calD_I & [0]_{(n-1) \times 1} \\
[?]_{1 \times (n-1)} & 0
\end{bmatrix}$. If $1 \not\in I$, then the lower-right $(n-1) \times (n-1)$ block of this matrix has the form
$\begin{bmatrix}
\calD_{I-1} & [0]_{(n-2) \times 1} \\
[?]_{1 \times (n-2)} & 0
\end{bmatrix}$, and, in any of the above cases, the shape of $\calV_\lr$ shows that the last row of this matrix equals zero.

Similarly, if $1 \in I$, then, with $I':=(\lcro 1,n-1\rcro \setminus I)-1$, one finds that
the lower-right $(n-1) \times (n-1)$ block of $I_n-M$ has the form
$\begin{bmatrix}
\calD_{I'} & [0]_{(n-2) \times 1} \\
[?]_{1 \times (n-2)} & 1
\end{bmatrix}$, and one deduces from the shape of $\calV_\lr$ that the last row of this matrix has all entries zero with the
exception of the last one. \\
In any case, this shows:
\begin{center} $\calV$ contains $\calD_I+d\,E_{n,1}$ for some $d \in \K$.
\end{center}

\vskip 2mm
From there, we shall examine the above five cases separately. Our main problem is that the shape of $\calV_\ul$
does not fully determine that of $\calV_\lr$, so that the invariance arguments we have resorted to
in the previous parts could now be in jeopardy.
The discussion is organized as follows:
\begin{itemize}
\item In Section \ref{2case3car3section}, we solve the problem in Cases 3 and 4 with $n \neq 4$
and the restrictive assumption that, for all spaces $\calU$ that are similar to $\calV$,
neither one of Cases 1, 2 and 5 occurs. In those situations, we obtain the expected similarities $\calV \simeq \calV_I^{(2)}$
and $\calV \simeq \calW_{p,\calF}^{(2)}$, respectively.

\item For Cases 1 and 2, the discussion is split into four separate sections:
in Section \ref{preliminarycase1and2}, we prove results that apply to both cases.
We examine Case 1 more specifically in Section \ref{additionalcase1}, and then Case 2 in Section \ref{additionalcase2}.
We wrap up the proof for both cases in Section \ref{completioncases1and2}, where we obtain
the excepted conclusion that $\calV \simeq \calW_{0,\calF}^{(2)}$ in Case 1, and
$\calV \simeq \calF \vee \calG$ in Case 2.

\item In Section \ref{nnot4car3}, we complete the proof in all cases provided that $n \neq 4$.

\item The final section \ref{n=4car3} completes the proof by solving the case $n=4$.
\end{itemize}

\subsection{Restricted proofs in Cases 3 and 4 with $n \neq 4$}\label{2case3car3section}

In this paragraph, we examine $\calV$ under a quite restrictive assumption.
Let us assume that, for every
subspace $\calU$ of $\Mat_n(\K)$ which is similar to $\calV$,
neither one of Cases 1, 2 and 5 occurs. Assume in addition that $n \neq 4$.
In particular, Case 3 or 4 holds for $\calV$, and, in any case, $I$ cannot equal $\lcro 3,n-1\rcro$ (see the discarded cases (ii) and (iii) above).

From there, the proof is largely similar to the one in Section \ref{1starcase1and2car3section}.
We know that, for every $L \in \Mat_{1,n-2}(\K)$, the space
$\calV_\ul$ contains a matrix of the form $\begin{bmatrix}
0 & L \\
[0]_{(n-2) \times 1} & [0]_{(n-2) \times (n-2)}
\end{bmatrix}$ and, for every $C \in \Mat_{n-2,1}(\K)$, the space
$\calV_\lr$ contains a matrix of the form $\begin{bmatrix}
[0]_{(n-2) \times (n-2)} & C \\
[0]_{1 \times (n-2)} & 0
\end{bmatrix}$.
This gives rise to two endomorphisms $\varphi : \Mat_{1,n-2}(\K) \rightarrow  \Mat_{1,n-2}(\K)$ and
$\psi : \Mat_{n-2,1}(\K) \rightarrow  \Mat_{n-2,1}(\K)$ together with two linear forms
$f : \Mat_{1,n-2}(\K) \rightarrow \K$ and $g : \Mat_{n-2,1}(\K) \rightarrow \K$ such that, for every
$(L,C) \in \Mat_{1,n-2}(\K) \times \Mat_{n-2,1}(\K)$, the space $\calV$ contains the matrices
$$A_L=\begin{bmatrix}
0 & L & 0 \\
0 & 0 & 0 \\
f(L) & \varphi(L) & 0
\end{bmatrix} \quad \text{and} \quad
B_C=\begin{bmatrix}
0 & 0 & 0 \\
\psi(C) & 0 & 0 \\
g(C) & 0 & 0
\end{bmatrix}.$$
Set
$$\calV'_m:=\begin{cases}
\NT_{n-2}(\K) & \text{in Case 3} \\
\calV_\md & \text{in Case 4.}
\end{cases}$$
In any case, using the respective shapes of $\calV_\ul$ and $\calV_\lr$,
we find a linear form $h : \calV'_m \rightarrow \K$ such that,
for every $U \in \calV'_m$, the space $\calV$ contains the matrix
$$E_U=\begin{bmatrix}
0 & 0 & 0 \\
0 & U & 0 \\
h(U) & 0 & 0
\end{bmatrix}.$$
As $n \neq 4$, Lemma \ref{homotheticsprop} yields scalars $\lambda$ and $\mu$ such that
$$\forall (L,C) \in \Mat_{1,n-2}(\K) \times \Mat_{n-2,1}(\K), \quad \varphi(L)=\lambda\,L \quad \text{and} \quad
\psi(C)=\mu\,C.$$
From Lemma \ref{lemmeALetBC}, one deduces that $\lambda+\mu=0$.
Setting $Q:=\begin{bmatrix}
1 & 0 & 0 \\
0 & I_{n-2} & 0 \\
\lambda & 0 & 1
\end{bmatrix}$ and replacing $\calV$ with $Q^{-1} \calV Q$, one finds that
$\calV_\ul$ and $\calV_\lr$ are left entirely unchanged, but now we have,
for all $(L,C) \in \Mat_{1,n-2}(\K) \times \Mat_{n-2,1}(\K)$,
$$A_L=\begin{bmatrix}
0 & L & 0 \\
0 & 0 & 0 \\
f(L) & 0 & 0
\end{bmatrix} \quad \text{and} \quad
B_C=\begin{bmatrix}
0 & 0 & 0 \\
0 & 0 & C \\
g(L) & 0 & 0
\end{bmatrix}.$$
Let us analyze $f$ and $g$.

\begin{claim}\label{basicclaimcar3case3and4}
The space $\calV$ contains $\calD_I$. Moreover, $f=0$ and $g$ vanishes at the column matrix $\begin{bmatrix}
1 & 0 & \cdots & 0
\end{bmatrix}^T$.
\end{claim}

\begin{proof}
We already have a scalar $d$ such that $\calV$ contains $\calD_I+d\,E_{n,1}$.
From there, the proof is widely similar to the one of Claim \ref{1starspeccar3basicclaim1}.

Set $L_0:=\begin{bmatrix}
1 & 0 & \cdots & 0
\end{bmatrix} \in \Mat_{1,n-2}(\K)$.
Denote by $d_1,\dots,d_n$ the diagonal entries of $\calD_I$, and note that $d_n=0$.
Let $(x,y,z)\in \K^3$. Then, $x\,D+y\,A_{L_0}+z\,B_{L_0^T}$
stabilizes $\Vect(e_1,e_2,e_n)$, and the matrix of its induced endomorphism in the basis $(e_1,e_2,e_n)$
is
$$x\,\begin{bmatrix}
d_1 & 0 & 0 \\
0 & d_2 & 0 \\
d & 0 & 0
\end{bmatrix}+y\,\begin{bmatrix}
0 & 1 & 0 \\
0 & 0 & 0 \\
f(L_0) & 0 & 0
\end{bmatrix}+z\,\begin{bmatrix}
0 & 0 & 0 \\
0 & 0 & 1 \\
g(L_0^T) & 0 & 0
\end{bmatrix}.$$
Thus, this matrix has at most two eigenvalues in $\Kbar$.
If one of $d_1$ and $d_2$ equals $1$, then Lemma \ref{lemmeALetBCcar3} yields $d=0$ and $f(L_0)=g(L_0^T)=0$.

Now, assume that $d_1=d_2=0$. As the case $I=\lcro 3,n-1\rcro$ has been discarded and as
$I$ is non-empty, one deduces that
$\{d_3,\dots,d_{n-1}\}=\{0,1\}$. Yet, $e_3,\dots,e_{n-1}$
are all eigenvectors of $D+x\,A_{L_0}+y\,B_{L_0^T}$, with respective
eigenvalues $d_3,\dots,d_{n-1}$. Therefore, the only possible eigenvalues of
$D+x\,A_{L_0}+y\,B_{L_0^T}$ in $\Kbar$ are $0$ and $1$.
Applying Lemma \ref{lemmeALetBCcar3} again, we deduce that $d=0$ and $f(L_0)=g(L_0^T)=0$.

From there, the claimed results ensue by following the same line of reasoning as in the proof of Claim \ref{1starspeccar3basicclaim1}.
\end{proof}

Let us sum up a portion of those results:

\begin{prop}\label{sumuppropcase3and4}
Let $\calU$ be a linear subspace of $\Mat_n(\K)$ which is similar to $\calV$.
Assume that $e_n$ is $\calU$-good, that $e_1$ is $\calU^T$-good and that $\calU$ falls into either one of Cases
3 or 4, with an associated set $I' \subset \lcro 1,n-1\rcro$ which is different from $\lcro 3,n-1\rcro$.
Then, there exists a matrix $Q=\begin{bmatrix}
I_{n-1} & [0]_{(n-1) \times 1} \\
[?]_{1 \times (n-1)} & 1
\end{bmatrix}$ such that $Q^{-1}\calU Q$ contains $E_{2,n}$.
\end{prop}

\begin{claim}\label{E1ninVcase99}
Assume that $n \neq 5$ or that $\calV_\md$ does not have type (II).
Then, the space $\calV$ contains $E_{1,n}$.
\end{claim}

\begin{proof}
We use the invariance argument already featured in the proof of Claim \ref{E1ninVclaim1}.
Set $P_1:=\begin{bmatrix}
1 & 1 \\
0 & 1
\end{bmatrix} \oplus I_{n-2}$ and consider the space $\calV':=P_1^{-1} \calV P_1$.
Note that $P_1e_n=e_n$, so that $e_n$ is $\calV'$-good.
Using the shape of $P_1$, one notices that $\calV'_\ul=\calV_\ul$,
and hence $e_1$ is $(\calV')^T$-good and the subset of $\lcro 1,n\rcro$ which is associated with $\calV'$
is $I$ (no modification there!).
Remembering our initial assumption that no space similar to $\calV$ falls into one of Cases 1, 2 and 5,
one deduces from Claim \ref{cornerclaim2car3} and Proposition \ref{sumuppropcase3and4} that there exists
a matrix of the form
$P_2=\begin{bmatrix}
I_{n-1} & [0]_{(n-1) \times 1} \\
[?]_{1 \times (n-1)} & 1
\end{bmatrix} \in \GL_n(\K)$ such that $P_2^{-1}P_1^{-1} \calV P_1P_2$ contains $E_{2,n}$.
Then, one follows the line of reasoning from the proof of Claim \ref{car3case1and2containsE1n}
to obtain that
$\calV$ contains a matrix of the form
$$H'=\begin{bmatrix}
0 & 0 & 1 \\
0 & T & 0 \\
0 & 0 & 0
\end{bmatrix},$$
where $T$ is a strictly upper-triangular matrix of $\Mat_{n-2}(\K)$ with all rows zero starting from the second one.
 We claim that $T \in \calV_\md'$ in all cases.
 \begin{itemize}
\item If $\calV_\md'=\NT_{n-2}(\K)$ or $\calV_\md'=\calW_{p,\calF}^{(1^\star)}$ for some $p \in \lcro 1,n-5\rcro$ and some exceptional
$\overline{1}$-spec subspace $\calF$ of $\Mat_3(\K)$, then it is obvious that $T \in \calV'_\md$.
\item Assume that $\calV_\md'=\calW_{0,\calF}^{(1^\star)}$ for some exceptional
$\overline{1}$-spec subspace $\calF$ of $\Mat_3(\K)$. Then, $\calV_\ul$ has type (II) and hence $n \neq 5$.
Note that $\calV_\md'$ is a $\overline{1}^\star$-spec subspace of
$\Mat_{n-2}(\K)$ with dimension $\dbinom{n-2}{2}+1$, so it is maximal among the $\overline{1}^\star$-spec subspaces of $\Mat_{n-2}(\K)$.
Let $U \in \calV_\md'$ and $x \in \K$. The matrix $x\,H'+E_U$ has at most two eigenvalues in $\K$; it is singular -
its $(n-1)$-th row being zero - and stabilizes the subspace $\Vect(e_2,\dots,e_{n-1})$: one deduces that
$x\,T+U$ has at most one non-zero eigenvalue in $\overline{\K}$. Therefore,
$\K T+\calV_\md'$ is a $\overline{1}^\star$-spec subspace of $\Mat_{n-2}(\K)$ which contains $\calV_\md'$. This yields
$T \in \calV_\md'$.
\end{itemize}
Adding an appropriate matrix of type $E_U$ to $H'$, one finds some $e \in \K$ such that $\calV$ contains $E_{1,n}+e\,E_{n,1}$.
However, the characteristic polynomial of this last matrix is $t^{n-2}(t^2+e)$, which has three distinct roots in $\Kbar$
whenever $e \neq 0$ (remember that $n \geq 3$). Thus, $e=0$, and hence $\calV$ contains $E_{1,n}$.
\end{proof}

\begin{claim}\label{vanishgandhcar3case3and4}
Assume that $\calV$ contains $E_{1,n}$. Then, $g=0$ and $h=0$.
\end{claim}

\begin{proof}
Let $U \in \calV'_\md$ and $(x,y) \in \K^2$. Note that $x\,E_{1,n}+y\,E_U$ belongs to $\calV$: therefore,
it has at most two eigenvalues in $\overline{\K}$.
Assume that $U$ is singular. Then, one sees that $x\,E_{1,n}+y\,E_U$
is singular (by looking at the middle $(n-2) \times (n-2)$-block), and therefore $x\,E_{1,n}+y\,E_U$
has at most one non-zero eigenvalue in $\Kbar$. However, the matrices $E_{1,n}$ and $E_U$ all stabilize $\Vect(e_1,e_n)$,
with induced endomorphisms represented in the basis $(e_1,e_n)$ by $\begin{bmatrix}
0 & 1 \\
0 & 0
\end{bmatrix}$ and $\begin{bmatrix}
0 & 0 \\
h(U) & 0
\end{bmatrix}$, respectively. Lemma \ref{n=2lemma} yields $h(U)=0$.
As $h$ is linear and $\calV_\md'$ is spanned by its singular elements (use Lemma \ref{spansingular}),
one concludes that $h=0$.

Let $C \in \Mat_{n-2,1}(\K)$. Note that $B_C$ and $E_{1,n}$ vanish everywhere on $\Vect(e_2,\dots,e_{n-1})$
and therefore $0$ is an eigenvalue of each linear combination of them.
Borrowing the line of reasoning from the proof of Claim \ref{vanishgandhcar3case1and2}, one concludes that $g(C)=0$.
\end{proof}

From there, we can conclude except in one special case.
Assume that $n \neq 5$ or that $\calV_\md$ does not have type (II). Then, the above results show that,
for every $(L,C,U) \in \Mat_{1,n-2}(\K) \times \Mat_{n-2,1}(\K) \times \calV'_\md$,
the space $\calV$ contains the matrix
$$\begin{bmatrix}
0 & L & 0 \\
0 & U & C \\
0 & 0 & 0
\end{bmatrix}.$$
We also know that $\calV$ contains $\calD_I$ and that it contains $E_{1,n}$ and $I_n$. By linearly combining such matrices,
we deduce that:
\begin{itemize}
\item Either $\calV_\ul$ has type (I) and therefore $\calV_I^{(2)} \subset \calV$;
\item Or $\calV_\ul=\calW_{p,\calF}^{(1^\star)}$ for some $p \in \lcro 1,n-4\rcro$ and some exceptional $\overline{1}$-spec subspace $\calF$ of
$\Mat_3(\K)$, and therefore $\calW_{p,\calF}^{(2)} \subset \calV$.
\end{itemize}
In any case, the equality of dimensions shows that $\calV$ has type (I) or type (II).
In particular, Theorem \ref{car3theo2} is now established in the case $n=3$.
This yields the following result, which will prove useful in the study of the remaining cases:

\begin{prop}\label{exceptionalis2specmaximal}
Let $\calF$ be an exceptional $\overline{1}$-spec subspace of $\Mat_3(\K)$.
Then, $\calF$ is maximal among the $\overline{2}$-spec linear subspaces of $\Mat_3(\K)$.
\end{prop}

\begin{proof}
By a \emph{reductio ad absurdum}, let us assume that there exists a $\overline{2}$-spec linear subspace $\calG$ of $\Mat_3(\K)$
such that $\calF \subsetneq \calG$. By Theorem \ref{2specinequality},
one deduces that $\dim \calG=5$. By Theorem \ref{car3theo2} for $n=3$,
there is a matrix $P \in \GL_3(\K)$ such that all the matrices of $P \calG P^{-1}$ are upper-triangular.
This yields the inclusion $P \calF P^{-1} \subset \K I_3 \oplus \NT_3(\K)$ - as all
the matrices of $\calF$ have at most one eigenvalue in $\Kbar$ - and the equality of dimensions
leads to $\calF \simeq \K I_3 \oplus \NT_3(\K)$, which contradicts the assumption that $\calF$ be exceptional.
\end{proof}

Now, we can complete the study of the special case when $n=5$ and $\calV_\md$ has type (II).
In that case, we see that all that needs to be done is to prove that $\calV$ contains $E_{1,n}$.
Coming back to the proof of Claim \ref{E1ninVcase99}, we find that
$\calV$ contains a matrix of the form
$$H'=\begin{bmatrix}
0 & 0 & 1 \\
0 & T & 0 \\
0 & 0 & 0
\end{bmatrix},$$
where $T \in \Mat_3(\K)$. On the other hand, $\calV_\md'$ is an exceptional $\overline{1}$-spec subspace of $\Mat_3(\K)$.
All we need to do is to prove that $T \in \calV_\md'$ as, with the line of reasoning from Claim \ref{E1ninVcase99},
the inclusion $E_{1,n} \in \calV$ will ensue. Let $x \in \K$ and $U \in \calV_\md'$.
As the matrix $xH'+E_U$ belongs to $\calV$, it has at most two eigenvalues in $\Kbar$, and hence
$\K T+\calV_\md'$ is a $\overline{2}$-spec subspace of $\Mat_3(\K)$. Proposition \ref{exceptionalis2specmaximal}
yields that $T \in \calV_\md'$, and we conclude that $\calV$ is a $\overline{2}$-spec subspace of type (II).

Therefore, Theorem \ref{car3theo2} is established under the restrictive assumptions of this section.

\subsection{Preliminary results for Cases 1 and 2}\label{preliminarycase1and2}

Here, we assume that $\calV_\ul=\calW_{0,\calF}^{(1^\star)}$ for some semi-reduced exceptional $\overline{1}$-spec subspace
$\calF$ of $\Mat_3(\K)$. Let $N \in \calF$. As $e_n$ is $\calV$-good,
we deduce that there are uniquely-defined row matrices $\Delta_1(N) \in \Mat_{1,3}(\K)$
and $\Delta_2(N) \in \Mat_{1,n-4}(\K)$ such that
$\calV$ contains the matrix
$$M_N=\begin{bmatrix}
N & [0]_{3 \times (n-4)} & [0]_{3 \times 1} \\
[0]_{(n-4) \times 3} & [0]_{(n-4) \times (n-4)} & [0]_{(n-4) \times 1} \\
\Delta_1(N) & \Delta_2(N) & 0
\end{bmatrix}.$$
Note that $\Delta_1$ and $\Delta_2$ are linear maps from $\calF$.

Remember also that we have a scalar $d$ such that $\calV$ contains $\calD_{\lcro 1,3\rcro}+d\,E_{n,1}$.
In other words,
$$\Delta_1(I_3)=\begin{bmatrix}
d & 0 & 0
\end{bmatrix} \quad \text{and} \quad  \Delta_2(I_3)=0.$$
In the rest of the paragraph, we shall make a stronger assumption on $\calF$:
\begin{center}
We assume that $\calF$ is well-reduced.
\end{center}

\begin{claim}
The map $\Delta_2$ vanishes everywhere on $\Mat_3(\K)$.
\end{claim}

\begin{proof}
Note that the result is trivial if $n=4$. Now, assume that $n \geq 5$.
Let $N \in \calF$.
\begin{itemize}
\item Assume that $\calV_\lr$ has type (I). Let $C \in \Mat_{1,n-4}(\K) \setminus \{0\}$.
Then, $\calV$ contains a matrix of the form $B=\begin{bmatrix}
[0]_{3 \times 1} & [0]_{3 \times (n-2)} & [0]_{3 \times 1} \\
[?]_{(n-4) \times 1} & [0]_{(n-4) \times (n-2)} & C \\
? & [0]_{1 \times (n-2)} & 0
\end{bmatrix}$. Note that $M_N+B$ is singular as its $(n-1)$-th column is zero (remember that $n \geq 5$).
Thus, $M_N+B$ has at most one non-zero eigenvalue in $\Kbar$. Then, one may follow the line of reasoning from the proof of
Claim \ref{Delta2nulclaim1} to obtain that $\Delta_2(N)C=0$. Varying $C$ yields $\Delta_2(N)=0$.
\item Assume that $\calV_\lr$ has type (II), i.e.\ that Case 2 holds.
Denote by $\calG$ the semi-reduced exceptional $\overline{1}$-spec subspace of $\Mat_3(\K)$ such that
$\calV_\lr=\calW_{2,\calG}^{(1^\star)}$. Set $A:=\begin{bmatrix}
[0]_{2 \times 2} & [0]_{2 \times 1} \\
\Delta_2(N) & 0
\end{bmatrix}\in \Mat_3(\K)$.
Let $N' \in \calG$.
Then, we know that $\calV$ contains a matrix of the form
$\begin{bmatrix}
[?]_{3 \times 3} & [0]_{3 \times 3} \\
[?]_{3 \times 3} & N' \\
\end{bmatrix}$. By linearly combining such a matrix with $M_N$, we obtain that $\calV$ contains a matrix of the form
$$\begin{bmatrix}
[?]_{3 \times 3} & [0]_{3 \times 3} \\
[?]_{3 \times 3} & \lambda\,A+N' \\
\end{bmatrix}$$
 for all $\lambda \in \K$. It follows that $\K A+\calG$ is a $\overline{2}$-spec linear subspace of $\Mat_3(\K)$.
Therefore, $A \in \calG$ by Proposition \ref{exceptionalis2specmaximal}.
As $\begin{bmatrix}
0 & 0 & 1
\end{bmatrix}^T$ is $\calG$-good, one concludes that $\Delta_2(N)=0$.
\end{itemize}
\end{proof}

Now, as $\calF$ is well-reduced, it contains two matrices of the following forms:
$$A=\begin{bmatrix}
0 & 1 & 0 \\
0 & 0 & 0 \\
\delta & 0 & 0
\end{bmatrix} \quad \text{and} \quad
B=\begin{bmatrix}
0 & 0 & 0 \\
0 & 0 & 1 \\
\epsilon & 0 & 0
\end{bmatrix} \quad \text{with $(\delta,\epsilon) \neq (0,0)$.}$$

\begin{claim}\label{formofDelta1}
One has $\Delta_1(B)=0$,  $\Delta_1(A)=
\begin{bmatrix}
0 & d & 0
\end{bmatrix}$ and $\Delta_1(I_3)=\begin{bmatrix}
d & 0 & 0
\end{bmatrix}$.
\end{claim}

\begin{proof}
Judging from the shape of $\calV_\lr$, we find that $\calV$ contains a matrix of the form
$$H=\begin{bmatrix}
0 & [0]_{1 \times (n-2)} & 0 \\
[?]_{(n-2) \times 1} & [0]_{(n-2) \times (n-2)} & C_0 \\
? & [0] & 0
\end{bmatrix}, \quad \text{where $C_0=\begin{bmatrix}
0 & 1 & 0 & \cdots & 0 \end{bmatrix}^T$.}$$
On the other hand, the shape of $\calV_\lr$ shows that $\Delta_1(B)=\begin{bmatrix}
e & 0 & 0
\end{bmatrix}$ for some $e \in \K$. Let us write $\Delta_1(A)=\begin{bmatrix}
a & b & c
\end{bmatrix}$.

Note that all four matrices $M_A$, $M_B$, $H$ and $\calD_{\lcro 1,3\rcro}+d\,E_{n,1}$ vanish everywhere on $\Vect(e_4,\dots,e_{n-1})$.
The matrices of their induced endomorphisms in the basis $(\overline{e_1},\overline{e_2},\overline{e_3},\overline{e_n})$
of $\K^n/\Vect(e_4,\dots,e_{n-1})$ are, respectively,
$$\begin{bmatrix}
0 & 1 & 0 & 0 \\
0 & 0 & 0 & 0 \\
\delta & 0 & 0 & 0 \\
a & b & c & 0
\end{bmatrix}, \quad
\begin{bmatrix}
0 & 0 & 0 & 0 \\
0 & 0 & 1 & 0 \\
\epsilon & 0 & 0 & 0 \\
e & 0 & 0 & 0
\end{bmatrix} , \quad
\begin{bmatrix}
0 & 0 & 0 & 0 \\
f & 0 & 0 & 0 \\
g & 0 & 0 & 1 \\
h & 0 & 0 & 0
\end{bmatrix}\quad \text{and} \quad
\begin{bmatrix}
1 & 0 & 0 & 0 \\
0 & 1 & 0 & 0 \\
0 & 0 & 1 & 0 \\
d & 0 & 0 & 0
\end{bmatrix}$$
for some $(f,g,h) \in \K^3$.
For every $(x,y) \in \K^2$, it follows that the matrix
$$\begin{bmatrix}
0 & x & 0 & 0 \\
f & 0 & y & 0 \\
\delta x+\epsilon y+g & 0 & 0 & 1 \\
ax+ey+h & bx & cx & 0
\end{bmatrix}$$
has at most two eigenvalues in $\Kbar$; its characteristic polynomial equals
$$p(t)=t^4-x(f+c)t^2-xy(\delta x+\epsilon y+g+b)\,t+x(cfx-y(ax+ey+h)).$$
In the special case when $y=0$ and $x=1$, this polynomial factorizes as $(t^2-f)(t^2-c)$, which has
at least three roots in $\Kbar$ if $f \neq c$. It follows that $f=c$.

Assume that $c \neq 0$. Take $x\neq 0$ and $y \neq 0$.
Then, $x(f+c)=2xc \neq 0$, and hence Lemma \ref{degree4car3} shows that the coefficient of $p(t)$ alongside $t$ must vanish.
Therefore,
$$\forall (x,y) \in (\K \setminus \{0\})^2, \quad \delta x+\epsilon y+g+b=0.$$
As $\K$ has more than $2$ elements, this would yield $\delta=\epsilon=0$, which is false.

One deduces that $c=0$, and hence $f=0$. Applying Lemma \ref{degree4car3} to $p(t)$, one finds that
$$\forall (x,y)\in \K^2, \; x\bigl(cfx-y(ax+ey+h)\bigr)=0,$$
and hence
$$\forall (x,y)\in (\K \setminus \{0\})^2, \; ax+ey+h=0.$$
As above, one deduces that $a=e=h=0$. This yields the claimed results on $\Delta_1(B)$.

We finish by proving that $d=b$. By adding the four $4 \times 4$ matrices above, we find that the matrix
$$\begin{bmatrix}
1 & 1 & 0 & 0 \\
0 & 1 & 1 & 0 \\
\delta+\epsilon+g & 0 & 1 & 1 \\
d &  b & 0  & 0 \\
\end{bmatrix}$$
has at most two eigenvalues in $\Kbar$.
One checks that its characteristic polynomial has the form $q(t)=t^4+? t+(b-d)$.
Using Lemma \ref{degree4car3}, one concludes that $b-d=0$, which completes the proof.
\end{proof}

Now, let us set $P:=I_n+d\,E_{n,1} \in \GL_n(\K)$ and replace $\calV$ with $P^{-1} \calV P$.
The new space $\calV$ satisfies all the previous assumptions, with the same $\calV_\ul$ and $\calV_\lr$ spaces,
but now we have $\Delta_1(I_3)=\Delta_1(A)=\Delta_1(B)=0$.
Put differently, the new map $\Delta_1$ vanishes at every matrix of $\calF$ which has a zero entry at the $(1,3)$-spot.
Taking into account the previous conjugations, we can sum up our results as follows:

\begin{prop}\label{sumupprop1case1and2car3}
Let $\calU$ be a $\overline{2}$-spec subspace of $\Mat_n(\K)$ with dimension $\dbinom{n}{2}+2$.
Assume that $e_n$ is $\calU$-good and that there exists a well-reduced exceptional $\overline{1}$-spec subspace
$\calF$ of $\Mat_3(\K)$ such that $\calU_\ul=\calW_{0,\calF}^{(1^\star)}$. Then, there exists a
matrix of the form $Q=\begin{bmatrix}
I_{n-1} & [0]_{(n-1) \times 1} \\
[?]_{1 \times (n-1)} & 1
\end{bmatrix}\in \GL_n(\K)$ such that, for every $N \in \calF$ with entry $0$ at the $(1,3)$-spot, the space
$Q^{-1} \calU Q$ contains the matrix $N \oplus 0_{n-3}$.
\end{prop}

Now, let us come back to $\calV$, \emph{with the additional assumption that $\calF$ be fully-reduced.}
Denote by $(f_1,f_2,f_3)$ the canonical basis of $\K^3$.
Lemma \ref{changeofgoodvectorlemma1} yields a non-zero scalar $\alpha$ and a matrix
$R \in \GL_3(\K)$ such that $R^{-1}\calF R$ is well-reduced
and the last column of $R$ is $\begin{bmatrix}
0 \\
1 \\
\alpha
\end{bmatrix}$. Setting $\widetilde{R}:=R \oplus I_{n-3} \in \GL_n(\K)$, we note that
$\widetilde{R}^{-1} \calV \widetilde{R}$ satisfies the assumptions of Proposition \ref{sumupprop1case1and2car3},
with $(\widetilde{R}^{-1} \calV \widetilde{R})_\ul=\calW_{0,R^{-1} \calF R}^{(1^\star)}$.
As in the proof of Claim \ref{claim1011}, one finds a matrix of the form $P=\begin{bmatrix}
I_{n-1} & [0]_{(n-1) \times 1} \\
[?]_{1 \times (n-1)} & 1
\end{bmatrix} \in \GL_n(\K)$ such that $P^{-1} \widetilde{R}^{-1} \calV \widetilde{R} P$ contains $I_3 \oplus 0_{n-3}$ together with a matrix of the form
$N \oplus 0_{n-3}$ where $N=\begin{bmatrix}
0 & 1 & 0 \\
0 & 0 & 0 \\
? & 0 & 0
\end{bmatrix} \in R^{-1} \calF R$.
Then, we find a row matrix $L' \in \Mat_{1,3}(\K)$ such that
$\calV$ contains the two matrices:
$$\begin{bmatrix}
I_3 & [0]_{3 \times (n-3)} \\
[0]_{(n-4) \times 3} & [0]_{(n-4) \times (n-3)} \\
L'R^{-1} & [0]_{1 \times (n-3)}
\end{bmatrix} \quad \text{and} \quad
\begin{bmatrix}
RNR^{-1} & [0]_{3 \times (n-3)} \\
[0]_{(n-4) \times 3} & [0]_{(n-4) \times (n-3)} \\
L'NR^{-1} & [0]_{1 \times (n-3)}
\end{bmatrix}.$$
As $\calV$ contains $I_3 \oplus 0_{n-3}$ and $e_n$ is $\calV$-good,
we obtain $L'=0$. On the other hand,
point (ii) of Lemma \ref{changeofgoodvectorlemma1} yields that $(I_3,A,B,RNR^{-1})$ is a basis of $\calF$,
and one concludes that $\Delta_1=0$.
Let us sum up:

\begin{prop}\label{sumupprop2case1and2car3}
Let $\calU$ be a $\overline{2}$-spec subspace of $\Mat_n(\K)$ with dimension $\dbinom{n}{2}+2$.
Assume that $e_n$ is $\calU$-good and that there exists a fully-reduced exceptional $\overline{1}$-spec subspace
$\calF$ of $\Mat_3(\K)$ such that $\calU_\ul=\calW_{0,\calF}^{(1^\star)}$. Then, there exists a
matrix of the form $Q=\begin{bmatrix}
I_{n-1} & [0]_{(n-1) \times 1} \\
[?]_{1 \times (n-1)} & 1
\end{bmatrix}\in \GL_n(\K)$ such that, for every $N \in \calF$, the space
$Q^{-1} \calU Q$ contains $N \oplus 0_{n-3}$.
\end{prop}

We conclude with an even more general result:

\begin{cor}\label{sumupprop3case1and2car3}
Let $\calU$ be a $\overline{2}$-spec subspace of $\Mat_n(\K)$ with dimension $\dbinom{n}{2}+2$.
Assume that $e_n$ is $\calU$-good and that there exists a semi-reduced exceptional $\overline{1}$-spec subspace
$\calF$ of $\Mat_3(\K)$ such that $\calU_\ul=\calW_{0,\calF}^{(1^\star)}$. Then, there exists
a row matrix $L_0 \in \Mat_{1,3}(\K)$ such that, for every $N \in \calF$, the space $\calU$ contains the matrix
$$\begin{bmatrix}
N & [0]_{3 \times (n-3)} \\
[0]_{(n-4) \times 3} & [0]_{(n-4) \times (n-3)} \\
L_0N & [0]_{1 \times (n-3)}
\end{bmatrix}.$$
In particular, if $\calU$ contains $I_3 \oplus 0_{n-3}$, then $L_0=0$, and
hence $\calU$ contains $N \oplus 0_{n-3}$ for all $N \in \calF$.
\end{cor}

\begin{proof}
We may find some $R \in \GL_3(\K)$ such that $\calF':=R^{-1} \calF R$ is fully-reduced.
Set $P_1:=R \oplus I_{n-3}$ and $\calU':=P_1^{-1} \calU P_1$. Then, $e_n$ is $\calU$-good and one checks
that $\calU'_\ul=\calW_{0,\calF'}^{(1^\star)}$. By Proposition \ref{sumupprop2case1and2car3},
this yields a matrix of the form $P_2=\begin{bmatrix}
I_{n-1} & [0]_{(n-1) \times 1} \\
[?]_{1 \times (n-1)} & 1
\end{bmatrix}\in \GL_n(\K)$ such that, for every $N \in \calF'$, the space
$P_2^{-1} \calU' P_2$ contains the matrix $N \oplus 0_{n-3}$.
Then, for every $N \in \calF'$, the matrix $(P_1P_2)(N \oplus 0_{n-3})(P_1P_2)^{-1}$ belongs to $\calU$,
which yields the claimed result.
\end{proof}

Now, let us come back to the initial situation of Cases 1 and 2. In that one,
we have a semi-reduced exceptional $\overline{1}$-spec subspace $\calF$ of $\Mat_3(\K)$ such that $\calV_\ul=\calW_{0,\calF}^{(1^\star)}$,
and we have a scalar $d$ such that $\calV$ contains $\calD_{\lcro 1,3\rcro}+d E_{n,1}$.
Corollary \ref{sumupprop3case1and2car3} yields that $\Delta_2$ vanishes everywhere on $\calF$, while
there is some $L_0 \in \Mat_{1,3}(\K)$ for which $\Delta_1(N)=L_0N$ for all $N \in \calF$.
Then, $L_0=\Delta_1(I_3)=\begin{bmatrix}
d & 0 & 0
\end{bmatrix}$ for some $d \in \K$. Setting
$P:=I_n+d E_{n,1}$ and replacing $\calV$ with $P^{-1}\calV P$, we do not modify the fact that $e_n$ is $\calV$-good,
and the spaces $\calV_\ul$ and $\calV_\lr$ are left unchanged. However, we have the additional property:
\begin{itemize}
\item[(D)] The space $\calV$ contains $N\oplus 0_{n-3}$ for all $N \in \calF$.
\end{itemize}

\subsection{Additional results for Case $1$}\label{additionalcase1}

Here, we assume that Case 1 holds, along with property (D).
Judging from the respective shapes of $\calV_\ul$ and $\calV_\lr$, this yields
linear maps $\varphi_1 : \Mat_{1,n-4}(\K) \rightarrow \Mat_{1,2}(\K)$,
$\varphi_2 : \Mat_{1,n-4}(\K) \rightarrow \Mat_{1,n-4}(\K)$, $f : \Mat_{1,n-4}(\K) \rightarrow \K$,
$\psi : \Mat_{n-2,1}(\K) \rightarrow \Mat_{n-2,1}(\K)$, $g : \Mat_{n-2,1}(\K) \rightarrow \K$
such that, for all $(L,C) \in \Mat_{1,n-4}(\K) \times \Mat_{n-2,1}(\K)$ the space $\calV$ contains the matrices
$$A_L=\begin{bmatrix}
0 & [0]_{1 \times 2} & L & 0 \\
[0]_{(n-2) \times 1} & [0]_{(n-2) \times 2} & [0]_{(n-2) \times (n-4)} & [0]_{(n-2) \times 1} \\
f(L) & \varphi_1(L) & \varphi_2(L) & 0
\end{bmatrix}$$
and
$$B_C=\begin{bmatrix}
0 & [0]_{1 \times (n-2)} & 0 \\
\psi(C) & [0]_{(n-2) \times (n-2)} & C \\
g(C) & [0]_{1 \times (n-2)} & 0
\end{bmatrix}.$$

In the rest of the proof, we write every column matrix
$C \in \Mat_{n-2,1}(\K)$ as $C=\begin{bmatrix}
C_1 \\
C_2
\end{bmatrix}$ with $C_1 \in \Mat_{2,1}(\K)$ and $C_2 \in \Mat_{n-4,1}(\K)$.

\begin{claim}
One has $\varphi_1=0$, and
there are two scalars $\lambda$ and $\mu$
such that
$$\forall L \in \Mat_{1,n-4}(\K), \quad \varphi_2(L)=\lambda\,L \quad
\text{and} \quad
\forall C \in \Mat_{n-2,1}(\K), \quad \psi(C)_2=\mu\,C_2.$$
\end{claim}

\begin{proof}
The result is trivial if $n=4$. Assume that $n \geq 5$.
Then, for every $(L,C)\in \Mat_{1,n-4}(\K) \times \Mat_{n-2,1}(\K)$,
every linear combination of $A_L$ and $B_C$ has rank at most $4$, and is therefore singular;
thus, every matrix of $\Vect(A_L,B_C)$ has at most one non-zero eigenvalue in $\Kbar$.
Then, the results are obtained by following the chain of arguments of the proof of Claim \ref{phi1nulclaim1}.
\end{proof}

\begin{claim}\label{vanishingforn>4}
Assume that $n \geq 5$. Then, $\lambda=\mu=0$, $f=0$, and $g$ vanishes at every matrix $C \in \Mat_{n-2,1}(\K)$ for which $C_1=0$.
\end{claim}

\begin{proof}
Let $(L,C)\in \Mat_{1,n-4}(\K) \times \Mat_{n-2,1}(\K)$.
Remembering that $\calV$ contains $\calD_{\lcro 1,3\rcro}$,
the line of reasoning from the proof of Claim \ref{exteriorclaim} yields $\lambda+\mu=0$, 
$\lambda=0$ and $f(L)=g(C)=0$, successively. 
\end{proof}

\begin{claim}\label{vanishingofgcase1car3}
One has $g=0$.
\end{claim}

\begin{proof}
As a by-product of the proof of Claim \ref{formofDelta1}, we find that $g$ vanishes at
the matrix $\begin{bmatrix}
0 & 1 & 0 & \cdots & 0
\end{bmatrix}^T$. The vanishing of $g$ at $\begin{bmatrix}
1 & 0 & 0 & \cdots & 0
\end{bmatrix}^T$ is obtained in the same manner as in the proof of Claim \ref{vanishingofgcar3case3}.
Combining this with Claim \ref{vanishingforn>4}, we conclude that $g=0$.
\end{proof}

\begin{claim}\label{vanishingofpsi1case1car3}
One has $\psi=0$.
\end{claim}

\begin{proof}
Let $C \in \Mat_{1,n-2}(\K)$ and set $E_C:=\begin{bmatrix}
0 & [0]_{1 \times 2} \\
\psi(C)_1 & [0]_{2 \times 2}
\end{bmatrix}\in \Mat_3(\K)$.
Remember that we have shown that $\psi(C)_2=0$ and $g(C)=0$.
Let $s \in \K$ and $N \in \calF$. Then, by adding $s\,B_C$ to $N\oplus 0_{n-3}$, we find that
$\calV$ contains a matrix of the form
$$\begin{bmatrix}
sE_C+N & [?]_{3 \times (n-3)} \\
[0]_{(n-3) \times 3} & [?]_{(n-3) \times (n-3)}
\end{bmatrix}.$$
It follows that $\K E_C+\calF$ is a $\overline{2}$-spec subspace of $\Mat_3(\K)$.
Using Proposition \ref{exceptionalis2specmaximal}, we deduce that $E_C \in \calF$.
As $e_1$ is $\calF^T$-good, one concludes that $\psi(C)_1=0$ and hence $\psi(C)=0$.
\end{proof}

We finish by summing up the above results:

\begin{prop}\label{sumupprop3case1car3}
Let $\calU$ be a $\overline{2}$-spec subspace of $\Mat_n(\K)$ with dimension $\dbinom{n}{2}+2$.
Assume that $e_n$ is $\calU$-good, that there exists a semi-reduced exceptional $\overline{1}$-spec subspace $\calF$
of $\Mat_3(\K)$ such that $\calU_\ul=\calW_{0,\calF}^{(1^\star)}$, and that $\calU_\lr$ has type (I). Then, there exists
a matrix of the form $Q=\begin{bmatrix}
I_{n-1} & [0]_{(n-1) \times 1} \\
[?]_{1 \times (n-1)} & 1
\end{bmatrix}\in \GL_n(\K)$ such that, for every $(N,L,C) \in \calF \times \Mat_{1,n-4}(\K) \times \Mat_{n-2,1}(\K)$, the space
$Q^{-1} \calU Q$ contains the matrices $N \oplus 0_{n-3}$,
$$\begin{bmatrix}
[0]_{1 \times 3} & L & 0 \\
[0]_{(n-1) \times 3} & [0]_{(n-1) \times (n-4)} & [0]_{(n-1) \times 1}
\end{bmatrix} \quad \text{and} \quad
\begin{bmatrix}
0 & [0]_{1 \times (n-2)} & 0 \\
[0]_{(n-2) \times 1} & [0]_{(n-2) \times (n-2)} & C \\
0 & [0]_{1 \times (n-2)} & 0
\end{bmatrix}.$$
\end{prop}

\subsection{Additional results for Case $2$}\label{additionalcase2}

Here, we assume that Case 2 holds, along with property (D).
Denote by $\calF$ and $\calG$ the semi-reduced exceptional $\overline{1}$-spec subspaces of $\Mat_3(\K)$
for which $\calV_\ul=\calW_{0,\calF}^{(1^\star)}$ and $\calV_\lr=\calW_{2,\calG}^{(1^\star)}$.

\begin{claim}
For all $N \in \calG$, the space $\calV$ contains $0_3 \oplus N$.
\end{claim}

\begin{proof}
Denote by $K_k\in \GL_k(\K)$ the permutation matrix associated with the permutation $i \mapsto k+1-i$.
Setting $\calU:=K_6 \calV^T K_6^{-1}$, $\calF':=K_3 \calG^T K_3^{-1}$ and $\calG':=K_3 \calF^T K_3^{-1}$, one finds that
$e_6$ is $\calU$-good (since $e_1$ is $\calV^T$-good), with $\calU_\ul=\calW_{0,\calG'}^{(1^\star)}$
and $\calU_\lr=\calW_{2,\calF'}^{(1^\star)}$. Note that $\calG'$ and $\calF'$ are both semi-reduced.
Finally, as $\calV$ contains $I_3 \oplus 0_3$, it also contains $I_6-(I_3 \oplus 0_3)=0_3\oplus I_3$,
and hence $\calU$ contains $I_3 \oplus 0_3$. Using Corollary \ref{sumupprop3case1and2car3},
one deduces that $\calU$ contains $N \oplus 0_3$ for all $N \in \calG'$.
The claimed statement ensues by coming back to $\calV$.
\end{proof}

Now, we know that $\calV$ contains $A\oplus B$ for all $(A,B) \in \calF \times \calG$.
The respective shapes of $\calV_\ul$ and $\calV_\lr$ yield (uniquely-determined) linear maps
$\varphi_1 : \Mat_{1,2}(\K) \rightarrow \Mat_{1,2}(\K)$, $\varphi_2 : \Mat_{1,2}(\K) \rightarrow \Mat_{1,2}(\K)$,
$f : \Mat_{1,2}(\K) \rightarrow \K$, $\psi_1 : \Mat_{2,1}(\K) \rightarrow \Mat_{2,1}(\K)$, $\psi_2 : \Mat_{2,1}(\K) \rightarrow \Mat_{2,1}(\K)$,
and $g : \Mat_{2,1}(\K) \rightarrow \K$ such that, for all $(L,C) \in \Mat_{1,2}(\K) \times \Mat_{2,1}(\K)$, the space $\calV$ contains the matrices
$$A_L:=\begin{bmatrix}
0 & [0]_{1 \times 2} & L & 0 \\
[0]_{4 \times 1} & [0]_{4 \times 2} & [0]_{4 \times 2} & [0]_{4 \times 1} \\
f(L) & \varphi_1(L) & \varphi_2(L) & 0 \\
\end{bmatrix} \quad \text{and} \quad
B_C:=\begin{bmatrix}
0 & [0]_{1 \times 4} & 0 \\
\psi_1(C) & [0]_{2 \times 4} & C \\
\psi_2(C) & [0]_{2 \times 4} & [0]_{2 \times 1} \\
g(C) & [0]_{1 \times 4} & 0
\end{bmatrix}.$$

\begin{claim}
One has $\psi_2(C)=0$ for all $C \in \Mat_{2,1}(\K)$.
\end{claim}

\begin{proof}
Let $(L,C) \in \Mat_{1,2}(\K) \times \Mat_{2,1}(\K)$ be with $L \neq 0$.
Set $x:=\begin{bmatrix}
[0]_{1 \times 3} &  L & 0
\end{bmatrix}^T$. The matrices $A_L^T$ and $B_C^T$ induce endomorphisms of the space $\Vect(e_1,x)$,
which are represented in $(e_1,x)$ by the matrices
$M:=\begin{bmatrix}
0 & 0 \\
1 & 0
\end{bmatrix}$ and
$N:=\begin{bmatrix}
0 & L\psi_2(C) \\
0 & 0
\end{bmatrix}$, respectively.
However, every linear combination of $A_L^T$ and $B_C^T$ has rank less than $6$, so $0$ is an eigenvalue of it.
Therefore, $\Vect(M,N)$ is a $\overline{1}^\star$-spec subspace of $\Mat_2(\K)$. By Lemma \ref{n=2lemma}, this yields $L\psi_2(C)=0$.
Varying $C$ yields $\psi_2(C)=0$.
\end{proof}

With a proof that is essentially similar to the one of Claim \ref{vanishingofgcase1car3}, one finds:

\begin{claim}\label{vanishingofgcase2car3}
One has $g=0$.
\end{claim}

With the same proof as Claim \ref{vanishingofpsi1case1car3}, one deduces:

\begin{claim}\label{vanishingofpsi1case2car3}
One has $\psi_1=0$.
\end{claim}

To sum up, for all $C \in \Mat_{2,1}(\K)$, the space $\calV$ contains the matrix
$$\begin{bmatrix}
[0]_{1 \times 5} & 0 \\
[0]_{2 \times 5} & C \\
[0]_{3 \times 5} & [0]_{3 \times 1}
\end{bmatrix}.$$
Applying this result to the space $K_6 \calV^T K_6^{-1}$, one deduces that, for all
$L \in \Mat_{1,2}(\K)$, the space $\calV$ contains the matrix
$$\begin{bmatrix}
[0]_{1 \times 3} & L & 0 \\
[0]_{5 \times 3} & [0]_{5 \times 2} & [0]_{5 \times 1}
\end{bmatrix}.$$

Let us sum up:

\begin{prop}\label{sumupprop4case2car3}
Let $\calU$ be a $\overline{2}$-spec subspace of $\Mat_6(\K)$ with dimension $\dbinom{6}{2}+2$.
Assume that $e_6$ is $\calU$-good, that there exists a semi-reduced exceptional $\overline{1}$-spec subspace $\calF$
of $\Mat_3(\K)$ such that $\calU_\ul=\calW_{0,\calF}^{(1^\star)}$, and that $\calU_\lr$ has type (II). Then, there exists
a matrix of the form $Q=\begin{bmatrix}
I_5 & [0]_{5 \times 1} \\
[?]_{1 \times 5} & 1
\end{bmatrix}\in \GL_6(\K)$, together with a semi-reduced exceptional $\overline{1}$-spec subspace $\calG$ of $\Mat_3(\K)$ such that, for every
$(A,B,L,C) \in \calF \times \calG \times \Mat_{1,2}(\K) \times \Mat_{2,1}(\K)$, the space $Q^{-1}\calU Q$ contains the matrix
$$\begin{bmatrix}
A & K_{L,C} \\
[0]_{3 \times 3} & B
\end{bmatrix} \quad \text{with $K_{L,C}=\begin{bmatrix}
L & 0 \\
[0]_{2 \times 2} & C
\end{bmatrix}$.}$$
\end{prop}

\subsection{Completing Cases 1 and 2}\label{completioncases1and2}

First of all, we sum up some of the previous results on Cases 1 and 2:

\begin{prop}\label{sumupprop5case1and2car3}
Let $\calU$ be a $\overline{2}$-spec subspace of $\Mat_n(\K)$ with dimension $\dbinom{n}{2}+2$.
Assume that $e_n$ is $\calU$-good and that there exists a semi-reduced exceptional $\overline{1}$-spec subspace $\calF$
of $\Mat_3(\K)$ such that $\calU_\ul=\calW_{0,\calF}^{(1^\star)}$.
Then, there exists a matrix of the form $Q=\begin{bmatrix}
I_{n-1} & [0]_{(n-1) \times 1} \\
[?]_{1 \times (n-1)} & 1
\end{bmatrix} \in \GL_n(\K)$ such that $Q^{-1} \calU Q$ contains $I_3 \oplus 0_{n-3}$ and $E_{3,n}$.
\end{prop}

Now, assume that either one of Cases 1 and Case 2 holds for $\calV$, with the addition of property (D).
Assume furthermore that $\calV_\ul=\calW_{0,\calF}^{(1^\star)}$ for some \emph{fully-reduced}
exceptional $\overline{1}$-spec subspace $\calF$ of $\Mat_3(\K)$.
In any case, we know that $\calV$ contains $I_3 \oplus 0_{n-3}$ and that, for every
$L \in \Mat_{1,n-4}(\K)$, it contains a matrix with first row $\begin{bmatrix}
[0]_{1 \times 3} & L & 0
\end{bmatrix}$ and all the other rows zero.
With the same invariance argument as in the proof of Claim \ref{car3case3E1ninV}, one
uses Proposition \ref{sumupprop5case1and2car3} to obtain:

\begin{claim}
The space $\calV$ contains $E_{1,n}$.
\end{claim}

From there, we split the discussion between Cases 1 and 2:
\begin{itemize}
\item Assume that Case 1 holds. Then, with the same line of reasoning as in the end of Section \ref{1starcase3car3section},
one finds that, for all $U \in \NT_{n-2}(\K)$ with a zero second column, the matrix
$\begin{bmatrix}
0 & [0] & 0 \\
[0] & U & [0] \\
0 & [0] & 0
\end{bmatrix}$ belongs to $\calV$. By linearly combining such matrices with $E_{1,n}$ and with all the special matrices found in $\calV$
(see Section \ref{additionalcase1}),
we finally find that $\calW_{0,\calF}^{(2)} \subset \calV$. As the dimensions are equal, we conclude that $\calV=\calW_{0,\calF}^{(2)}$.

\item Assume that Case 2 holds, and denote by $\calG$ the semi-reduced exceptional $\overline{1}$-spec subspace of $\Mat_3(\K)$ such that
$\calV_\lr=\calW_{2,\calG}^{(1^\star)}$.
Then, with the same line of reasoning as in end of Section \ref{1starcase3car3section},
one finds that, for every $U \in \NT_{4}(\K)$ with a zero second column and a zero third row,
the matrix $\begin{bmatrix}
0 & [0] & 0 \\
[0] & U & [0] \\
0 & [0] & 0
\end{bmatrix}$ belongs to $\calV$. By linearly combining such matrices with $E_{1,n}$ and all the special matrices found in $\calV$
(see Section \ref{additionalcase2}),
we find that $\calF \vee \calG \subset \calV$, and the equality of dimensions yields $\calV=\calF \vee \calG$.
\end{itemize}

Let us conclude:

\begin{prop}\label{conclusioncases1and2}
Let $\calU$ be a $\overline{2}$-spec subspace of $\Mat_n(\K)$ with dimension $\dbinom{n}{2}+2$.
Assume that $e_n$ is $\calU$-good and that there exists a fully-reduced exceptional $\overline{1}$-spec subspace $\calF$
of $\Mat_3(\K)$ such that $\calU_\ul=\calW_{0,\calF}^{(1^\star)}$. \\
Then, either $\calU \simeq \calW_{0,\calF}^{(2)}$ or there exists an exceptional $\overline{1}$-spec subspace $\calG$ of
$\Mat_3(\K)$ such that $\calU \simeq \calF \vee \calG$.
\end{prop}

Conjugating by well-chosen matrices of the form $R \oplus I_{n-3}$ with $R \in \GL_3(\K)$,
this yields:

\begin{cor}\label{cor1conclusioncases1and2}
Let $\calU$ be a $\overline{2}$-spec subspace of $\Mat_n(\K)$ with dimension $\dbinom{n}{2}+2$.
Assume that $e_n$ is $\calU$-good and that there exists a semi-reduced exceptional $\overline{1}$-spec subspace $\calF$
of $\Mat_3(\K)$ such that $\calU_\ul=\calW_{0,\calF}^{(1^\star)}$. \\
Then, either $\calU \simeq \calW_{0,\calF}^{(2)}$ or there exists an exceptional $\overline{1}$-spec subspace $\calG$ of
$\Mat_3(\K)$ such that $\calU \simeq \calF \vee \calG$.
\end{cor}

With the ``transpose and conjugate" argument, this yields:

\begin{cor}\label{cor2conclusioncases1and2}
Let $\calU$ be a $\overline{2}$-spec subspace of $\Mat_n(\K)$ with dimension $\dbinom{n}{2}+2$.
Assume that $e_1$ is $\calU^T$-good, and
that there exists a semi-reduced exceptional $\overline{1}$-spec subspace $\calG$
of $\Mat_3(\K)$ such that $\calU_\lr=\calW_{n-4,\calG}^{(1^\star)}$. \\
Then, either $\calU \simeq \calW_{n-3,\calG}^{(2)}$ or there exists an exceptional $\overline{1}$-spec subspace $\calF$ of
$\Mat_3(\K)$ such that $\calU \simeq \calF \vee \calG$.
\end{cor}

\subsection{Completing the case $n \neq 4$}\label{nnot4car3}

Here, we assume that $n\neq 4$. If there exists a space $\calU \simeq \calV$
which satisfies the assumptions of Corollary \ref{cor1conclusioncases1and2}
or the ones of Corollary \ref{cor2conclusioncases1and2}, then we are done.
If the contrary holds, the results from Section \ref{2case3car3section} yield
that $\calV$ is similar to a space of type (I) or (II).
This completes the proof of Theorem \ref{car3theo2} in the case when $n \neq 4$.

\subsection{The case $n=4$}\label{n=4car3}

Here, we assume that $n=4$. We may also assume that
no space $\calU$ which is similar to $\calV$ satisfies the assumptions of Corollary \ref{cor1conclusioncases1and2}
or the ones of Corollary \ref{cor2conclusioncases1and2}. In other words, if a space
$\calU \simeq \calV$ is such that $e_n$ is $\calU$-good and one of Cases 1 to 5 holds,
then only Case 3 can hold (remember that $n=4$, which discards Case 4).
In particular, Case 3 is now assumed to hold for $\calV$.

From there, we follow the line of reasoning of Section \ref{n=4car<>3}.
Using the respective shapes of $\calV_\ul$ and $\calV_\lr$, we find two matrices $A$ and $B$ of $\Mat_2(\K)$ together with two linear forms
$f : \Mat_{1,2}(\K) \rightarrow \K$ and $g : \Mat_{2,1}(\K) \rightarrow \K$ such that, for all $(L,C) \in \Mat_{1,2}(\K) \times \Mat_{2,1}(\K)$,
the space $\calV$ contains the matrices
$$A_L:=\begin{bmatrix}
0 & L & 0 \\
[0]_{2 \times 1} & [0]_{2 \times 2} & [0]_{2 \times 1} \\
f(L) & LA & 0
\end{bmatrix} \quad \text{and} \quad
B_C:=\begin{bmatrix}
0 & [0]_{1 \times 2} & 0 \\
BC & [0]_{2 \times 2} & C \\
g(C) & [0]_{1 \times 2}  & 0
\end{bmatrix}.$$
The characteristic polynomial of $A_L+B_C$ is
$$p(t)=t^4-L(A+B)C\, t^2-(f(L)+g(C))\,LC\,t+\bigl((LAC)(LBC)-(LC)(LABC)\bigr).$$
This polynomial has at most two roots in $\Kbar$.
Using Lemma \ref{degree4car3} instead of point (a) of Lemma \ref{4by4carnot3lemma},
the line of reasoning of the proof of Claim \ref{AmoinsBscalairecar<>3} may be adapted so as to
yield a scalar $\alpha \in \K$ such that $A-B=\alpha I_2$.
As in Section \ref{n=4car<>3}, we then find that no generality is lost in assuming that $A=B$.

\begin{center}
Until further notice, we assume that $A \neq 0$.
\end{center}
Then, for an arbitrary $(L,C)\in \Mat_{1,2}(\K) \times \Mat_{2,1}(\K)$,
we may write
$$p(t)=t^4-2LAC\, t^2-(f(L)+g(C))\,LC\,t+\bigl((LAC)^2-(LC)(LA^2C)\bigr).$$

Here is the next step:

\begin{claim}\label{A2egal0car3}
One has $A^2=0$.
\end{claim}

\begin{proof}
Let $(L,C) \in \Mat_{1,2}(\K) \times \Mat_{2,1}(\K)$. Lemma \ref{degree4car3} applied to $p(t)$ yields
$$4(LAC)^2=4\bigl((LAC)^2-(LC)(LA^2C)\bigr),$$
and therefore $LC(LA^2C)=0$.
With the same line of reasoning as in the proof of Claim \ref{A2egal0car<>3}, one deduces that $A^2=0$.
\end{proof}

\begin{claim}
One has $f=0$ and $g=0$.
\end{claim}

\begin{proof}
Lemma \ref{degree4car3} applied to $p(t)$ yields that, for all
$(L,C)\in \Mat_{1,2}(\K) \times \Mat_{2,1}(\K)$, the conditions $LAC \neq 0$ and $LC\neq 0$ imply $f(L)+g(C)=0$.
Let $C \in \K^2$ be such that $AC \neq 0$. As $A^2C=0$, we find that $C$ is not an eigenvector of $A$,
and hence $(C,AC)$ is a basis of $\K^2$. Therefore, we may find some $L \in \Mat_{1,2}(\K)$ for which $L(AC) \neq 0$ and $LC \neq 0$. \\
For every $\lambda \in \K \setminus \{0\}$, the row vector $\lambda L$ also satisfies those conditions, which yields:
$$\lambda\, f(L)+g(C)=0.$$
Varying $\lambda$, we deduce that $g(C)=0$ (as $\K$ has more than $2$ elements).
However, $\Ker A \neq \K^2$, and hence $\K^2 \setminus \Ker A$ spans $\K^2$.
As $g$ is linear, we deduce that $g=0$. With the same line of reasoning, one finds $f=0$.
\end{proof}

By Claim \ref{A2egal0car3}, $A$ is a rank $1$ nilpotent matrix. Let us determine its kernel.

\begin{claim}\label{KerAcar3}
The space $\Ker A$ is spanned by $\begin{bmatrix}
1 & 0
\end{bmatrix}^T$.
\end{claim}

\begin{proof}
We assume that the contrary holds.
With the same line of reasoning as in the proof of Claim \ref{KerAcar<>3}, one sees that no generality is lost in assuming that
$A=\begin{bmatrix}
0 & 0 \\
1 & 0
\end{bmatrix}$. On the other hand, using the respective shapes of $\calV_\ul$ and $\calV_\lr$, we find some $\delta \in \K$
such that $\calV$ contains the matrix
$\begin{bmatrix}
0 & 0 & 0 & 0 \\
0 & 0 & 1 & 0 \\
0 & 0 & 0 & 0 \\
\delta & 0 & 0 & 0
\end{bmatrix}$. Then, with the same arguments as in the proof of Claim \ref{KerAcar<>3}, one finds that $\delta = 0$.
We deduce that $\calV$ contains the matrix
$$\begin{bmatrix}
0 & 1 & 1 & 0 \\
0 & 0 & 1 & 1 \\
1 & 0 & 0 & 0 \\
0 & 1 & 0 & 0
\end{bmatrix},$$
the characteristic polynomial of which is $t^4-2t^2-t+1$. This contradicts Lemma \ref{degree4car3}.
Therefore, $\Ker A=\K \begin{bmatrix}
1 \\
0
\end{bmatrix}$.
\end{proof}

We deduce that $A=\begin{bmatrix}
0 & \alpha \\
0 & 0
\end{bmatrix}$ for some $\alpha \in \K \setminus \{0\}$.
Taking the discarded case $A=0$ into account, we sum up some of the previous results as follows:

\begin{prop}\label{sumuppropn=4car<>3}
Let $\calU$ be a linear subspace of $\Mat_4(\K)$ which is similar to $\calV$.
Assume that $e_4$ is $\calU$-good and that there
exists a non-empty subset $I$ of $\lcro 1,3\rcro$ such that $\calU_{\ul}=\calV_I^{(1^\star)}$.
Then, there exists a matrix $Q=\begin{bmatrix}
I_3 & [0]_{3 \times 1} \\
[?]_{1 \times 3} & 1
\end{bmatrix}\in \GL_4(\K)$ such that $Q^{-1}\calU Q$ contains $E_{1,3}$ and $E_{2,4}$.
\end{prop}

From there, one may follow the line of reasoning from the end of Section \ref{n=4car<>3}
and check that all the arguments apply to our current situation. One concludes that:
\begin{itemize}
\item Either $A=0$ and then $\calV=\calV_I^{(2)}$;
\item Or $A \neq 0$ and then $\calV \simeq \calG_4(\K)$ or $\calV \simeq \calG_4'(\K)$.
\end{itemize}

\paragraph{}
This completes the proof of Theorem \ref{car3theo2}.

\end{document}